\documentclass[a4paper,10pt,reqno,final]{amsart}
\usepackage{amsmath,amsfonts,amssymb,amsthm,bm}
\usepackage[hmargin={28mm,28mm},vmargin={30mm,35mm}]{geometry}
\usepackage[ocgcolorlinks,allcolors={blue},breaklinks]{hyperref}
\usepackage{paralist}
\usepackage{xcolor,graphicx}
\usepackage[bold]{hhtensor}
\usepackage{xparse}
\usepackage{array}

\usepackage[color,notref,notcite]{showkeys}
\definecolor{labelkey}{rgb}{0.6,0,1}
\numberwithin{equation}{section}

\newcounter{testcaseR}
\renewcommand{\thetestcaseR}{R}
\newcommand{\tcaseR}[2]{\medskip\noindent\refstepcounter{testcaseR}\label{#1}\emph{Test case \thetestcaseR: #2.}}
\newcommand{\refR}[1]{\ref{#1}}

\newcounter{testcaseP}
\renewcommand{\thetestcaseP}{P\arabic{testcaseP}}
\newcommand{\tcaseP}[2]{\medskip\noindent\refstepcounter{testcaseP}\label{#1}\emph{Test case \thetestcaseP: #2.}}
\newcommand{\refP}[1]{\ref{#1}}

\newcounter{testcaseS}
\renewcommand{\thetestcaseS}{S\arabic{testcaseS}}
\newcommand{\tcaseS}[2]{\medskip\noindent\refstepcounter{testcaseS}\label{#1}\emph{Test case \thetestcaseS: #2.}}
\newcommand{\refS}[1]{\ref{#1}}

\newtheorem{theorem}{Theorem}[section]
\newtheorem{lemma}[theorem]{Lemma}

\newtheorem{corollary}[theorem]{Corollary}
\theoremstyle{remark}
\newtheorem{remark}[theorem]{Remark}
\theoremstyle{definition}
\newtheorem{assumption}[theorem]{Assumption}
\newtheorem{definition}[theorem]{Definition}

\newcommand{\R}{\mathbb{R}}
\newcommand{\N}{\mathbb{N}}
\newcommand{\Poly}[1]{\mathbb{P}^{#1}}

\NewDocumentCommand{\disc}{ O{} O{}}{{\mathcal D_{#1}^{#2}}}
\newcommand{\discs}[1][]{\disc[*][#1]}

\renewcommand{\div}{\mathrm{div}}

\newcommand{\XDz}{X_{\disc,0}}
\newcommand{\PiD}{\Pi_{\disc}}

\newcommand{\QD}[1][]{Q_{\disc[#1]}}

\newcommand{\PiDs}[1][]{{\Pi_{\discs[#1]}}}
\newcommand{\nablaD}{\nabla_{\disc}}
\newcommand{\ID}{\mathcal I_{\disc}}

\newcommand{\XDmz}{X_{\disc[m],0}}
\newcommand{\PiDm}{\Pi_{\disc[m]}}
\newcommand{\nablaDm}{\nabla_{\disc[m]}}

\newcommand{\norm}[2][]{\|#2\|_{#1}}

\newcommand{\bw}{\overline{w}}
\newcommand{\bu}{\overline{u}}
\newcommand{\bv}{\overline{v}}
\newcommand{\x}{{\bm x}}

\def\term{\mathfrak{T}}   

\newcommand{\pwpart}{{U}}
\newcommand{\proj}[1][\pwpart]{\mathrm{Pr}_{#1}}

\newcommand{\mesh}{\mathcal M}
\newcommand{\gen}{{\rm g}}
\newcommand{\rand}{{\rm r}}
\newcommand{\unif}{{\rm u}}
\newcommand{\equil}{{\rm e}}
\newcommand{\spl}{{\rm s}}
\newcommand{\di}{{\rm 1/4}}

\newcommand{\fe}{\mbox{\textsc{\scriptsize fe}}}
\newcommand{\dg}{\mbox{\textsc{\scriptsize dg}}}
\newcommand{\umesh}{\mesh^\unif_N}
\newcommand{\rmesh}{\mesh^\rand_N}
\newcommand{\faces}{{\mathcal F}}           
\newcommand{\vertices}{\mathcal V}

\usepackage{constants}
\newconstantfamily{ctrcst}{symbol=C,reset={chapter}}
\def\ctel#1{\ensuremath{\Cl[ctrcst]{#1}}}
\def\cter#1{\ensuremath{\Cr{#1}}}

\begin{document}

\title[High-order mass-lumped schemes for nonlinear degenerate elliptic equations]{High-order mass-lumped schemes \\ for nonlinear degenerate elliptic equations}
\author{J\'er\^ome Droniou}
\address[J. Droniou]{School of Mathematics, Monash University, Melbourne, Australia.} 
\email{jerome.droniou@monash.edu}

\author{Robert Eymard}
\address[R. Eymard]{Laboratoire d'Analyse et de Math\'ematiques Appliqu\'ees, Universit\'e Paris-Est Marne-la-Vall\'ee
Champs-sur-Marne, France.}
\email{Robert.Eymard@u-pem.fr}

\thanks{
This work was supported by the Australian Government through the Australian Research Council's Discovery Projects funding scheme (project DP170100605), and by the French Government through the Agence Nationale de la Recherche (project CHARMS, ANR-16-CE06-0009).
}

\begin{abstract}
	We present and analyse a numerical framework for the approximation of nonlinear degenerate elliptic equations of the Stefan or porous medium types. This framework is based on piecewise constant approximations for the functions, which we show are essentially necessary to obtain convergence and error estimates. Convergence is established without regularity assumption on the solution. 
A detailed analysis is then performed to understand the design properties that enable a scheme, despite these piecewise constant approximations and the degeneracy of the model, to satisfy high-order error estimates if the solution is piecewise smooth. Numerical tests, based on continuous and discontinuous approximation methods, are provided on a variety of 1D and 2D problems, showing the influence on the convergence rate of the nature of the degeneracy and of the design choices.

\end{abstract}

\maketitle


\section{Introduction}

The goal of numerical methods for partial differential equations is to approximate, as accurately as possible, the continuous solution. For mesh-based methods, it is well-known that when the problem is linear and the solution has sufficient regularity properties, for a fixed number of degrees of freedom high-order methods provide more accurate solutions than low-order methods. This result must however be questioned in the case of nonlinear problems for which, even if the solution is smooth enough, stability and high-order estimates might not be achievable without the proper structure of the chosen discretisation.
We propose in this work to explore this question, considering the following nonlinear degenerate elliptic equation as the basis of our discussion:
\begin{equation}\label{stefan:strong}
	\begin{aligned}
	\beta(\bu)-\div(\Lambda\nabla \zeta(\bu))={}&f+\div(F)&\mbox{ in $\Omega$},\\
	\zeta(\bu)={}&0&\mbox{ on $\partial\Omega$}.
	\end{aligned}
\end{equation}
The corresponding weak formulation for this problem is
\begin{equation}\label{stefan:weak}
\begin{aligned}
&\mbox{Find $\bu\in L^2(\Omega)$ such that $\zeta(\bu)\in H^1_0(\Omega)$ and}\\
&\int_\Omega \beta(\bu) \bv +\int_\Omega \Lambda\nabla\zeta(\bu) \cdot\nabla \bv = \int_\Omega f\bv-\int_\Omega F\cdot\nabla \bv\,,\qquad\forall \bv\in H^1_0(\Omega).
\end{aligned}
\end{equation}
Throughout the paper, we denote by $\norm[L^2]{{\cdot}}$ the norms in $L^2(\Omega)$ or $L^2(\Omega)^d$, and we make the following assumptions:
\begin{subequations}
\begin{align}
\bullet~&\Omega \mbox{ is an open bounded connected subset of $\R^d$ ($d\in\N^\star$)}, \label{assum:Omega} \\
\bullet~&\zeta:\R\to\R\mbox{ is continuous and non-decreasing, $\zeta(0) = 0$ and,}\nonumber\\
&\mbox{ for some $M_0,M_1>0$, }|\zeta(s)| \ge M_0 |s| - M_1\mbox{  for all } s\in\R,
\label{assum:zeta}\\
\bullet~&
\beta:\R\to\R\hbox{ is continuous and non-decreasing, $\beta(0) = 0$ and,}\nonumber\\
&\mbox{ for some $K_0,K_1>0$, }|\beta(s)| \le K_0 |s| + K_1\mbox{  for all } s\in\R,
\label{assum:beta}\\
\bullet~&\label{assum:incr}\beta+\zeta:\R\to\R\mbox{ is strictly increasing,}\\
 \intertext{}
\bullet~& 
\Lambda:\Omega\to \mathcal M_d(\R) \hbox{ is measurable and
there exists $\overline{\lambda}\ge\underline{\lambda}>0$ such that,}\nonumber\\
&\mbox{for a.e. $\x\in\Omega$, $\Lambda(\x)$ is symmetric with eigenvalues in $[\underline{\lambda},\overline{\lambda}]$.}
\label{assum:lambda}\\
\bullet~&
f \in  L^2(\Omega)\quad\mbox{and}\quad F\in L^2(\Omega)^d. 
\label{assum:f}
\end{align}
\label{eq:assum}
\end{subequations}
\begin{remark}[Growth assumptions]
The super-linearity of $\zeta$ assumed in \eqref{assum:zeta} ensures that, even though the model is degenerate, it allows for proper a priori estimates on the solution: since $\zeta(\bu)\in H^1_0(\Omega)$, the super-linearity ensures that $\bu$ belongs to $L^2(\Omega)$ (at least). The sub-linearity of $\beta$ is assumed in \eqref{assum:beta} to make sure that $\beta(\bu)$ belongs to the same Lebesgue space as $\bu$, as that this non-linearity is continuous in this space; this is essential to pass to the limit in the numerical approximations.
\end{remark}

The theoretical study of \eqref{stefan:strong} is covered by the pioneering paper \cite{carrillo}, extended in \cite{andreianov2009well}, on problems also including a nonlinear convection term; using techniques that require to multiply the equation by various functions of the unknown, existence and uniqueness of an entropy solution are obtained.
An existence and uniqueness result for the simpler problem considered here is given by Theorem \ref{th:exist.uniq} in Appendix \ref{appen:exist.uniq}, without referring to entropy solutions. 

\medskip

The case $\zeta = {\rm Id}$ fits into quasilinear second-order elliptic problems, the approximation of which is covered in a rather large literature, see e.g. \cite{milner1985mixed,chen1998expanded,ZengYu,bi2014global}.
The case $\zeta \neq {\rm Id}$, on which we focus in this paper, raises severe issues and is less often considered in the literature, especially when considering the question of high-order schemes.
First, for such a problem, the solution can display discontinuities when $\zeta$ has plateaux. Moreover, the nonlinearities challenge the design of numerical methods that simultaneously 
\begin{inparaenum}[(i)]
\item only require computing integrals of polynomials (integrals that can be exactly computed in general),
\item are amenable to error estimates (or, at the very least, proven to be convergent), and
\item are of order higher than 1.
\end{inparaenum}

\medskip

Extending the entropy method used in \cite{carrillo} to the notion of entropy process solutions, the convergence of a Two-Point Flux Approximation (TPFA) Finite Volume method is proved in \cite{convpardeg} for a time-dependent version of \eqref{stefan:strong} with $\Lambda={\rm Id}$. The entropy method requires us to consider $\phi(u)$ as test functions for various nonlinear functions $\phi$, a process that can only be reproduced at the discrete level for the TPFA scheme, see \cite[Section 7]{review} and \cite{TPnotTP}. Unfortunately, the TPFA scheme is only applicable on very specific grids, which usually forces $\Lambda = {\rm Id}$, and only low-order error estimates can be expected from the application of the doubling variable technique as in  \cite{eymard2002error}.

\medskip

In the general case of an anisotropic heterogeneous field $\Lambda$, we need to consider more versatile schemes than the TPFA scheme, which will necessarily reduce the range of admissible test functions. Nevertheless, an important feature to preserve, if one wants to ensure the stability of the discretisation, is the capacity to choose appropriate test functions to simultaneously get diffusion estimates from the gradient terms, and a positive sign from the reaction term. 

\medskip

Let us first consider the case of conforming Galerkin methods. Given a subspace $V_h$ of $H^1_0(\Omega)$, a conforming scheme for \eqref{stefan:weak} is written
\begin{equation}\label{stefan:cs.intro}
\mbox{Find $u\in V_h$ such that:}\ \int_\Omega \beta(u) v +\int_\Omega \Lambda\nabla \zeta(u) \cdot\nabla v=\int_\Omega f v- \int_\Omega F\cdot\nabla v\,,\quad\forall v\in V_h.
\end{equation}
If $u\in V_h$ and $\zeta$ is globally Lipschitz continuous, we have $\zeta(u)\in H^1_0(\Omega)$ and the key of the convergence analysis is that the chain rule $\nabla\zeta(u)=\zeta'(u)\nabla u$ enables us to take $v=u\in V_h$ as test function in the scheme. This choice creates from the diffusion term the quantity of interest $\zeta'(u)|\nabla u|^2$, while the reaction term is non-negative since $\beta(u)u\ge 0$. However, to deduce any sort of estimate from this choice of test function, we are forced to set $F=0$, since the term $\int_\Omega F\cdot\nabla u$ cannot in general be estimated using $\int_\Omega \zeta'(u) |\nabla u|^2$. A better choice of test function to estimate the term resulting from the presence of $F$ would be $v=\zeta(u)$ since the diffusion term would provide the quantity $\int_\Omega |\nabla\zeta(u)|^2$, which can be used to estimate $\int_\Omega F\cdot\nabla \zeta(u)$. However, $v=\zeta(u)$ is not a valid test function in the scheme since it does not belong to $V_h$ in general. Fixing $F=0$, the convergence of \eqref{stefan:cs.intro} can nonetheless be proved, but no error estimate can be derived --- the reason for this being, again, the lack of freedom in choosing suitable test functions in the scheme. The analysis of conforming approximations is sketched in Appendix \ref{sec:conforming} (in which \eqref{stefan:weak} is first recast before applying the Galerkin method).

\medskip

Coming back to the general case of \eqref{stefan:weak} with possibly $F\neq 0$, we consider numerical methods for which the chain rule does not hold at the discrete level (as is the case for the majority of non-conforming methods). Unless the model problem is recast with a different form of nonlinearity as in \cite{cances-guichard,CNV}, the only reasonable test function to consider in order to get estimates is $v=\zeta(u)$, which formally provides $|\nabla\zeta(u)|^2$ from the diffusion term. More precisely, let us consider a scheme where the discrete unknowns $z=(z_i)_{i\in I}$ represent pointwise values of the solutions at certain nodes, and functions $z_h$ are reconstructed from these values and used in the weak formulation (this is the choice made in \cite{eliott1987error,amiez1992error} in the case of the transient problem, through the use of ``Lagrange interpolation operators''). Then, for $u=(u_i)_{i\in I}$, one can easily define $v=\zeta(u)$ pointwise, setting $v_i=\zeta(u_i)$ for all $i\in I$. The weak  formulation then involves $\nabla (\zeta(u))_h\cdot\nabla v_h$ and, taking $v=\zeta(u)$, this diffusion term generates the quantity $|\nabla (\zeta(u))_h|^2$. 

With this choice of $v$, the reaction term creates the quantity $\beta(u_h)(\zeta(u))_h$. This function is, at the considered nodes, equal to $\beta(u_i)\zeta(u_i)\ge 0$ (see \eqref{assum:zeta}--\eqref{assum:beta}). However, outside the nodes, no particular sign can be ensured for $\beta(u_h)(\zeta(u))_h$ and it is not clear that this reaction term will indeed lead to proper estimates on the solution to the scheme.

The way to solve this conundrum is, in the reaction term, to use a different reconstruction of functions than the natural reconstruction $(\cdot)_h$ used for the diffusion term. Utilising, for example, a piecewise constant reconstruction, in which the only values taken by the reconstruction are nodal values, ensures that the positivity of the reaction term --- valid at these nodal values --- extends to the entire domain (this is again done for low-order methods in \cite{eliott1987error,amiez1992error} to handle the accumulation terms issued from the time derivative, and in \cite{elliott1981fin} on a variational inequality equivalent to \eqref{stefan:strong} with $\zeta=\rm{Id}$ and $\beta$ multi-valued). For linear models, using piecewise constant reconstructions for reaction/accumulation terms leads to what is called mass-lumped schemes. There is a large literature on the mass-lumping of Finite Element methods for second order problems, see e.g.~\cite{CJRT,GMV,RW17,JS} and references therein. In most of these references, though, the construction of mass-lumped versions of high-order methods is justified by a need to reduce computational costs: for explicit discretisations of time-dependent linear problems, a mass-lumped scheme ensures a diagonal mass matrix which, unlike the standard mass matrix, is trivial to invert at each time step. This property of diagonal mass matrix has also been heavily used in schemes for eigenvalues problems related to linear elliptic operators (see for example \cite{andreev1992lumped} and references therein). On the contrary, for a nonlinear degenerate model as \eqref{stefan:strong}, as explained above the mass-lumping is not just a way to improve the method's efficiency, but appears as an imperative to establish convergence and error estimates --- and thus rigorously ensure that the scheme has high-order approximation properties. Additionally, the usual interpretation of mass-lumping as a specific choice of quadrature rules for the mass matrix is meaningful mostly in the linear setting. For nonlinear models, the less standard interpretation based on piecewise constant reconstructions is more appropriate (even though, as we will see, there is still some link to exploit with local quadrature rules). Finally, let us notice that, to our best knowledge, mass-lumping techniques seem to only be considered in the literature on Finite Element methods, not in the literature covering other high-order polynomial-based methods such as Discontinuous Galerkin. This is understandable when the goal is to simplify the inversion of the mass matrix; mass-lumping is then not much useful to methods such as Discontinuous Galerkin schemes, for which the standard mass matrix is easy to invert due to its block diagonal structure (which can also easily be made fully diagonal by a simple choice of orthogonal local polynomial basis). However, when the primary objective of mass-lumping is to  enable convergence and error estimates for nonlinear models, the question of designing mass-lumped Discontinuous Galerkin (or other methods based on local polynomials) is fully relevant.

\medskip

Our goal in this paper is to design high-order mass-lumped schemes for the nonlinear degenerate model \eqref{stefan:strong}. Our main contributions can be summarised as follows:
\begin{itemize}[\hspace*{.4em}$\bullet$]
\item design of a general analysis framework that treats in a unified way many different methods, including Finite Elements and Discontinuous Galerkin methods (and others);
\item proof of error estimates in this general framework;
\item identification of conditions on the mass-lumping to ensure high-order convergences (when the exact solution is piecewise smooth), despite the nonlinearities and degeneracy in the model;
\item extensive numerical tests, using both $\Poly{k}$ Finite Elements and Discontinuous Galerkin schemes, on realistic test cases (porous medium, Stefan) to validate the theoretical analysis.
\end{itemize}

\medskip

Let us describe the organisation of this paper. We first provide in Section \ref{sec:gs} a general formulation of numerical schemes, based on schemes written in fully discrete form: approximate functions and gradients are reconstructed without direct relation, and the approximate functions are piecewise constant. This construction is performed in the Gradient Discretisation Method \cite{gdm}, a framework that provides efficient notations and notions for the design and analysis of such schemes. After proving a first convergence result (Theorem \ref{th:convergence.scheme}) in Section \ref{sec:cvgs}, we establish in Section \ref{sec:errest} error estimates on the approximation of $\zeta(\bu)$ when using mass-lumped schemes (Theorem \ref{th:error.est} and Corollary \ref{cor:rates}).  As demonstrated in Section \ref{sec:GS.quad}, this general error estimate yields a high-order convergence rate (Theorem \ref{th:high.order}) for piecewise smooth solutions to \eqref{stefan:weak}, provided the mass-lumping is performed in a way that corresponds to sufficiently high-order local quadrature rules. These conditions on the local quadrature rules are similar to those highlighted for Finite Elements in \cite{ciarlet,CJRT} but, interestingly, they appear here from the need of estimating quite different error terms than in the case of linear models as in these references.
Extensive numerical tests are presented in Section \ref{sec:num}, both on porous medium equations and on Stefan problems, using mass-lumped Finite Element and Discontinuous Galerkin schemes; the results confirm that high-order approximations are obtained only if the aforementioned local quadrature rules hold, even if the theoretical assumptions are not fully satisfied (e.g.\ the solution is not piecewise smooth). The paper is completed with a short conclusion (Section \ref{sec:conclusion}) and two appendices. In Appendix \ref{appen:exist.uniq}, the properties of the continuous problem are analysed, and Appendix \ref{sec:conforming} sketches the study of the conforming scheme \eqref{stefan:cs.intro} with $F=0$, and highlights its limitations compared to the method in Section \ref{sec:gs}: strong convergence of the gradients only under some regularity assumption on the continuous solution, no error estimate, no uniqueness of the discrete solution.
%

\section{Schemes with piecewise constant approximation}\label{sec:gs}

To present the discretisation of \eqref{stefan:strong}, we use the Gradient Discretisation Method (GDM) \cite{gdm}, a generic numerical analysis framework for diffusion equations that encompasses many different discretisations: Finite Element, Finite Volumes, etc. Using this framework enables a unified treatment of all these different schemes, and also gives an efficient setting and tools to deal with them, including the notion of mass-lumping that will be essential to design a scheme for which an error estimate can be established.

The principle of the GDM is to introduce discrete elements --- a finite dimensional space, an operator that reconstructs functions, and an operator that reconstructs gradients --- together called a Gradient Discretisation (GD), and to replace the continuous counterparts in the weak formulation \eqref{stefan:weak} with these discrete elements, leading to a Gradient Scheme (GS) for \eqref{stefan:strong}.

\begin{definition}[Gradient Discretisation]\label{def:gd}
A Gradient Discretisation is $\disc=(\XDz,\PiD,\nablaD,\QD)$ such that
\begin{enumerate}[\hspace*{.4em}$\bullet$]
 \item $\XDz$ a finite-dimensional space.
 \item $\PiD:\XDz\to L^2(\Omega)$ and $\nablaD:\XDz\to L^2(\Omega)^d$ are linear operators reconstructing, respectively, a function and a gradient; $\nablaD$ must be chosen such
that $\norm[\disc]{{\cdot}}:=\norm[L^2]{\nablaD {\cdot}}$ is a norm on $\XDz$.
 \item $\QD:L^2(\Omega)\to L^2(\Omega)$ is a quadrature operator.
\end{enumerate}
\end{definition}

\begin{remark}[Quadrature operator]\label{rem:QD}
Quadrature rules for source terms are usually not accounted for in the definition and analysis of Gradient Schemes. In the context of mass-lumped schemes, however, accounting for quadrature rules is essential to establishing optimal high-order error estimates.

Note that $\QD$ is not assumed to be bounded. This enables different choices of quadrature rules depending on the regularity of the considered functions: $\QD f$ could be computed using pointwise values of $f$ if $f$ is continuous, or using averaged values of $f$ otherwise.
\end{remark}

To deal with the nonlinearity in the derivatives in \eqref{stefan:strong} we need the following notion.

\begin{definition}[Piecewise constant reconstruction]\label{def:pw}
Let $\disc$ be a Gradient Discretisation such that, for some finite sets $I$ and $I_{\partial\Omega}\subset I$, it holds
\[
\XDz=\{v=(v_i)_{i\in I}\,:\, v_i\in\R\quad\forall i\in I\,,\;v_i=0\quad\forall i\in I_{\partial\Omega}\}.
\]
We say that the reconstruction $\PiD$ is piecewise constant if there exists a partition $\pwpart=(\pwpart_i)_{i\in I}$ of $\Omega$ (some of the $\pwpart_i$ can be empty) such that
\begin{equation}\label{def:PiD.ML}
\forall v=(v_i)_{i\in I}\in\XDz\,,\quad \PiD v=\sum_{i\in I}v_i\mathbf{1}_{\pwpart_i},
\end{equation}
where $\mathbf{1}_{\pwpart_i}$ is the characteristic function of $\pwpart_i$. In other words, $(\PiD v)_{|\pwpart_i}=v_i$ for all $i\in I$.
\end{definition}

\begin{remark}[Reconstruction operator]
Note that if some $\pwpart_i$ are empty, then $\PiD$ is not injective. This is a classical situation in the GDM, see e.g. for example the mass-lumped $\Poly{2}$ scheme in Remark \ref{rem:P2ml}, or the HMM method in Remark \ref{rem:u.not.unique} and \cite[Chapter 13]{gdm}.
\end{remark}

In the setting of this definition, if $g:\R\to\R$ is a function satisfying $g(0)=0$, we define (with an abuse of notation) $g:\XDz\to\XDz$ by applying $g$ coefficient by coefficient:
\begin{equation}\label{eq:nonlin.function}
\forall v=(v_i)_{i\in I}\,,\; g(v)=(g(v_i))_{i\in I}.
\end{equation}
We note that this definition actually depends on the choice of the basis of $\XDz$. In practice, this basis being canonical and chosen once and for all, we do not make explicit the dependency of $g(v)$ with respect to it.
If $\PiD$ is a piecewise constant reconstruction, then  \eqref{eq:nonlin.function} leads to
\begin{equation}\label{eq:commut}
\forall v\in \XDz\,,\; \PiD g(v)=g(\PiD v).
\end{equation}

The accuracy properties of a GD are assessed through the following quantities. The first one measures a discrete Poincar\'e constant of $\disc$, the second one is an interpolation error, whilst the last one measures the conformity defect of the method (how well a discrete divergence formula holds).
\begin{align}
&C_{\disc}:=\max_{v\in\XDz\backslash\{0\}}\frac{\norm[L^2]{\PiD v}}{\norm[\disc]{v}},\label{def:CD}\\
&\forall\varphi\in H^1_0(\Omega)\,,\quad S_{\disc}(\varphi)=\min_{v\in\XDz}\left(\norm[L^2]{\nablaD v-\nabla \varphi}
+\norm[L^2]{\PiD v-\varphi}\right),\label{def:SD}\\
&\forall\vec{\psi}\in H_\div(\Omega)\,,\;W_{\disc}(\vec{\psi}):=\max_{v\in\XDz\backslash\{0\}}\frac{1}{\norm[{\disc}]{v}}\displaystyle\left|\int_\Omega \nablaD v\cdot\vec{\psi}+\PiD v\div\vec{\psi}\right|.
\label{def:WD}
\end{align}
In the following, unless otherwise specified, the notation $a\lesssim b$ means that $a\le Cb$ with $C>0$ depending only on the data in Assumption \eqref{eq:assum} and on an upper bound of $C_{\disc}$.

Given a Gradient Discretisation $\disc=(\XDz,\PiD,\nablaD,\QD)$ with piecewise constant reconstruction  as in Definition \ref{def:pw}, the Gradient Scheme for \eqref{stefan:strong} is (compare with the weak formulation \eqref{stefan:weak}):
\begin{equation}\label{stefan:gs}
\begin{aligned}
&\mbox{Find $u\in \XDz$ such that}\\
&\int_\Omega \beta(\PiD u) \PiD v +\int_\Omega \Lambda\nablaD\zeta(u) \cdot\nablaD v=\int_\Omega \QD f\,\PiD v
-\int_\Omega F\cdot\nablaD v\,,\quad\forall v\in \XDz.
\end{aligned}
\end{equation}

\begin{remark}[Quadrature for $\div(F)$]
A quadrature operator could also be introduced for $F$ (for example, considering $\QD$ component-wise, or selecting a different quadrature operator more appropriate to the structure of the gradient reconstruction). For simplicity of the presentation we decide not to include it in the analysis.
\end{remark}

\subsection{Convergence analysis}\label{sec:cvgs}

We first prove an \emph{a priori} estimate on the solution to the Gradient Scheme. This estimate is used to prove the existence of this solution, its convergence, and the error estimate \eqref{eq:error.est.new}.

\begin{lemma}[Bounds on the solution to the GS]\label{lem:estim.disc}
Let $\disc$ be a GD with piecewise constant reconstruction as in Definition \ref{def:pw}, and let $u\in \XDz$ be a solution to the Gradient Scheme \eqref{stefan:gs}. Then
\begin{equation}
 \norm[L^2]{\PiD u}+\norm[L^2]{\PiD\beta(u)} +\norm[\disc]{\zeta(u)}\lesssim \norm[L^2]{\QD f}+\norm[L^2]{F}+1. 
\label{eq:estim}
\end{equation}
\end{lemma}

\begin{proof}
Letting $v=\zeta(u)$ in \eqref{stefan:gs} we get
\[
 \int_\Omega \beta(\PiD u)  \zeta(\PiD u) +\int_\Omega \Lambda\nablaD\zeta(u) \cdot\nablaD\zeta(u) =\int_\Omega \QD f\,\PiD \zeta(u)
-\int_\Omega F\cdot\nablaD\zeta(u),
\]
where we have used \eqref{eq:commut} to write $\PiD \zeta(u)=\zeta(\PiD u)$ in the first integral term.
By monotonicity of $\beta,\zeta$ and $\beta(0)=\zeta(0)=0$, we have $\beta(s)\zeta(s)\ge 0$ and the equation above thus gives, by definition of $C_{\disc}$ and Assumption \eqref{assum:lambda},
\[
\begin{aligned}
\underline{\lambda}\norm[L^2]{\nablaD\zeta(u)}^2\le{}& \norm[L^2]{\QD f}\norm[L^2]{\PiD\zeta(u)}
+ \norm[L^2]{F}\norm[L^2]{\nablaD\zeta(u)}\\
\le{}& \left(C_{\disc}\norm[L^2]{\QD f}+\norm[L^2]{F}\right)\norm[L^2]{\nablaD\zeta(u)}.
\end{aligned}
\]
Recalling that $\norm[\disc]{\zeta(u)}=\norm[L^2]{\nablaD\zeta(u)}$, this estimate yields the bound on $\zeta(u)$ in \eqref{eq:estim}. Using again the definition of $C_{\disc}$, we infer that $\norm[L^2]{\PiD\zeta(u)}\lesssim \norm[L^2]{\QD f}+\norm[L^2]{F}$. By \eqref{eq:commut} this gives an $L^2(\Omega)$-estimate on $\zeta(\PiD u)$ and, using Assumption \eqref{assum:zeta}, translates into the bound on $\PiD u$ in \eqref{eq:estim}. The estimate on $\beta(\PiD u)$ follows from the sub-linearity of $\beta$ stated in Assumption \eqref{assum:beta}.
\end{proof}

\begin{lemma}[Existence and uniqueness for the GS]\label{lem:exuniq}
Assume \eqref{eq:assum} and let $\disc$ be a Gradient Discretisation with a piecewise constant reconstruction as in Definition \ref{def:pw}. Then there exists a solution to the Gradient Scheme \eqref{stefan:gs} and, if $(u_1,u_2)$ are two solutions to this scheme, then $\zeta(u_1)=\zeta(u_2)$ and $\PiD u_1=\PiD u_2$.
\end{lemma}

\begin{remark}[Counter-example to $u_1=u_2$]\label{rem:u.not.unique}
In general, we cannot claim that $u_1=u_2$, as the following counter-example shows. Consider $\beta(s)=s$ and $\zeta:\R\to\R$  such that $\zeta(s)=0$ for all $s\in [0,1]$. Take $F=0$ and $f\in L^2(\Omega)$ such that $0\le f\le 1$ almost everywhere, and consider an HMM Gradient Scheme \cite[Chapter 13]{gdm} on a polytopal mesh of $\Omega$ (which assumes that $\Omega$ is polytopal). Denoting by $\mesh$ and $\faces$, respectively, the sets of cells and faces of this mesh, the corresponding Gradient Discretisation satisfies
\[
\XDz=\{v=((v_K)_{K\in\mesh},(v_\sigma)_{\sigma\in\faces})\,:\,v_K\in\R\,,\;v_\sigma\in\R\,,\;v_\sigma=0\mbox{ if }\sigma\subset \partial\Omega\}
\]
and $(\PiD v)_{|K}=v_K$ for all $K\in\mesh$. We select $\QD = {\rm Id}$, and the precise expression of $\nablaD v$ is irrelevant to our counter-example. Then any $u=((u_K)_{K\in\mesh},(u_\sigma)_{\sigma\in\faces})\in\XDz$ that satisfies
\begin{equation}\label{def:u.sol.nonunique}
u_K=\frac{1}{|K|}\int_K f\,\quad\forall K\in\mesh\,,\quad u_\sigma\in [0,1]\quad\forall \sigma\in\faces\,,\quad u_\sigma=0\mbox{ if $\sigma\subset\partial\Omega$}
\end{equation}
is a solution to the Gradient Scheme \eqref{stefan:gs}. Indeed, all the components of such a vector belong to $[0,1]$ and thus $\zeta(u)=0$. The scheme equation on $u$ thus reduces to 
\[
\int_\Omega \PiD u \PiD v = \int_\Omega f\PiD v\,,\quad\forall v\in\XDz,
\]
which holds given the choice of the cell values $(u_K)_{K\in\mesh}$. Since there is an infinite number of $u$ satisfying \eqref{def:u.sol.nonunique} (as the values on internal faces are free in $[0,1]$), this establishes that, when considering the HMM scheme, uniqueness fails for \eqref{stefan:gs} with these $f$, $\beta$, $\zeta$.
\end{remark}

\begin{proof}[Proof of Lemma \ref{lem:exuniq}]
The existence is obtained via a topological degree argument; we refer the reader to \cite{deimling} for the definition and properties of this degree. Fix an arbitrary Euclidean structure, with inner product $\langle\cdot,\cdot\rangle$, on the finite dimensional space $\XDz$. For $a\in [0,1]$ let $\zeta_a(s)=a\zeta(u)+(1-a)u$. Define $\mathfrak F:[0,1]\times \XDz\to\XDz$ the following way: for $a\in [0,1]$ and $u\in \XDz$, $\mathfrak F(a,u)$ is the unique element of $\XDz$ such that, for all $v\in \XDz$,
\[
\langle \mathfrak F(a,u),v\rangle=\int_\Omega  a\beta(\PiD u) \PiD v +\int_\Omega \Lambda\nablaD\zeta_a(u) \cdot\nablaD v-\int_\Omega a \QD f\,\PiD v-\int_\Omega aF\cdot\nablaD v.
\]
We note that $u$ is a solution to the Gradient Scheme \eqref{stefan:gs} if and only if $\mathfrak F(1,u)=0$. 

By continuity of $\beta$ and $\zeta$, and the finite dimension of $\XDz$, the mapping $\mathfrak F$ is clearly continuous. Assume that $\mathfrak F(a,u)=0$ for some $a\in [0,1]$. The arguments in the proof of Lemma \ref{lem:estim.disc}, using $v=\zeta_a(u)$ as a test function, show that $\norm[\disc]{\zeta_a(u)}\le \ctel{cst1}$ with $\cter{cst1}$ not depending on $a$; by equivalence of norms on the finite dimensional space $\XDz$, this shows that $\norm[\infty]{\zeta_a(u)}\le \ctel{cst2}$
with $\cter{cst2}$ still independent on $a$ and $\norm[\infty]{{\cdot}}$ the supremum norm in $\XDz$ on an arbitrary basis. The mapping $\zeta_a$ satisfies \eqref{assum:zeta} with $M_0,M_1$ independent of $a$. As a consequence, the bound on $\norm[\infty]{\zeta_a(u)}$ shows that $\norm[\infty]{u}< R$ with $R$ independent of $a$.

Hence, any solution to $\mathfrak F(a,u)=0$ lies in the open ball $B_R$ of $\XDz$, centered at $0$ and of radius $R$ in the norm $\norm[\infty]{{\cdot}}$. This ball being independent of $a$, the topological degree theory ensures that ${\rm deg}(\mathfrak F(1,\cdot),B_R,0)={\rm deg}(\mathfrak F(0,\cdot),B_R,0)$. The mapping $\mathfrak F(0,\cdot):\XDz\to\XDz$ is linear and the estimate obtained on the solutions to $\mathfrak F(0,u)=0$ shows that $\mathfrak F(0,\cdot)$ has a trivial kernel, and is therefore invertible. This implies ${\rm deg}(\mathfrak F(0,\cdot),B_R,0)\neq 0$ and thus ${\rm deg}(\mathfrak F(1,\cdot),B_R,0)\neq 0$, which proves that the equation $\mathfrak F(1,u)=0$ has a solution $u\in B_R$.

\medskip

We now consider the uniqueness of the solution to the scheme. Subtracting the equations satisfied by $u_1$ and $u_2$ and taking $v=\zeta(u_1)-\zeta(u_2)\in \XDz$ as a test function, we have
\[
\int_\Omega (\beta(\PiD u_1)-\beta(\PiD u_2)) \PiD (\zeta(u_1)-\zeta(u_2)) 
+ \int_\Omega \Lambda\nablaD (\zeta(u_1)-\zeta(u_2))\cdot\nablaD (\zeta(u_1)-\zeta(u_2))=0.
\]
Property \eqref{eq:commut} and the monotonicity of $\beta$ and $\zeta$ show that 
\[
(\beta(\PiD u_1)-\beta(\PiD u_2)) \PiD (\zeta(u_1)-\zeta(u_2))=(\beta(\PiD u_1)-\beta(\PiD u_2)) (\zeta(\PiD u_1)-\zeta(\PiD u_2))\ge 0.
\]
Hence, $\norm[L^2]{\nablaD (\zeta(u_1)-\zeta(u_2))}=0$ which, by property of $\nablaD$, ensures that $\zeta(u_1)=\zeta(u_2)$.

We now come back to the equations satisfied by $u_1$ and $u_2$, subtract them and take $v=\beta(u_1)-\beta(u_2)\in \XDz$ as a test function to get
\[
\int_\Omega (\beta(\PiD u_1)-\beta(\PiD u_2))^2 + \int_\Omega \nablaD (\zeta(u_1)-\zeta(u_2))\cdot\nablaD (\beta(u_1)-\beta(u_2))=0.
\]
Since $\zeta(u_1)=\zeta(u_2)$, we infer that $\beta(\PiD u_1)-\beta(\PiD u_2)=0$. Owing to Hypothesis \eqref{assum:incr}, we conclude that $\PiD u_1=\PiD u_2$ from  $\beta(\PiD u_1)+\zeta(\PiD u_1)=\beta(\PiD u_2)+\zeta(\PiD u_2)$. \end{proof}

The next theorem is our first main convergence result. It states the strong convergence of the solution to the Gradient Scheme without assuming any regularity property on the continuous solution.

\begin{theorem}[Convergence of the scheme]\label{th:convergence.scheme}
Assume \eqref{eq:assum} and let $(\disc[m])_{m\in\N}$ be a sequence of Gradient Discretisations with piecewise constant reconstructions as in Definition \ref{def:pw}. Assume moreover that the following properties hold:
\begin{enumerate}[\hspace*{.4em}$\bullet$]
 \item \emph{(Coercivity)} The sequence $(C_{\disc[m]})_{m\in\N}$ is bounded, where $C_{\disc[m]}$ is defined by \eqref{def:CD} for $\disc=\disc[m]$.
 \item \emph{(Consistency)} Recalling the definition \eqref{def:SD}, there holds
 \begin{equation}
  \forall \varphi\in H^1_0(\Omega), \ \lim_{m\to\infty}S_{\disc[m]}(\varphi) = 0\quad\mbox{ and }\quad
  \lim_{m\to\infty} \norm[L^2]{\QD[m] f -f} = 0.\label{def:gdconsistency:pdisc}
 \end{equation}
 \item \emph{(Limit-conformity)} Recalling the definition \eqref{def:WD}, there holds
 \begin{equation}
  \forall\vec{\psi}\in H_\div(\Omega)\,,\;\lim_{m\to\infty}W_{\disc[m]}(\vec{\psi})=0.\label{def:gddonformity:pdisc}
 \end{equation}
 \item \emph{(Compactness)} For any $(v_m)_{m\in\N}$ such that $v_m\in \XDmz$ for all $m\in\N$ and $(\nablaDm v_m)_{m\in\N}$ is bounded in $L^2(\Omega)^d$, the set $\{\PiDm v_m\,:\, m\in\N\}$ is
 relatively compact in $L^2(\Omega)$.
\end{enumerate}
For any $m\in\N$, let $u_m$ be a solution of Scheme \eqref{stefan:gs}. Then there exists $\bu$ solution to \eqref{stefan:weak} such that, as $m\to\infty$, $\PiDm\zeta(u_m)\to \zeta(\bu)$ strongly in $L^2(\Omega)$, $\nablaDm\zeta(u_m)\to \nabla\zeta(\bu)$ strongly in $L^2(\Omega)^d$, and $\PiDm\beta(u_m)\to \beta(\bu)$ weakly in $L^2(\Omega)$.  
\end{theorem}

\begin{proof}
Using Estimate \eqref{eq:estim} and \eqref{def:gdconsistency:pdisc} (which shows that $\norm[L^2]{\QD[m] f}$ is bounded), the compactness and the limit-conformity of $(\disc[m])_{m\in\N}$, \cite[Lemma 2.15]{gdm} gives $Z\in H^1_0(\Omega)$ and $B\in L^2(\Omega)$ such that, up to a subsequence (not made explicit in the following), $\PiDm \zeta(u_m)\to Z$ strongly in $L^2(\Omega)$, $\nablaDm \zeta(u_m)\to \nabla Z$ weakly in $L^2(\Omega)^d$ and $\beta(\PiDm u_m)\to B$ weakly in $L^2(\Omega)$. By weak/strong convergence we infer that
\[
 \lim_{m\to\infty} \int_\Omega \PiDm \beta(u_m)\PiDm\zeta(u_m) = \int_\Omega B Z.
\]
The monotonicity properties of $\beta$ and $\zeta$ then enable us to apply \cite[Lemma D.10]{gdm} (a Minty's trick) to get $\bu\in L^2(\Omega)$ such that $Z = \zeta(\bu)$ and $B = \beta(\bu)$. 

We now show that $\bu$ solves \eqref{stefan:weak}. Let $\varphi\in H^1_0(\Omega)$ and let $v_m\in\XDmz$ be an element that realises the minimum defining $S_{\disc[m]}(\varphi)$. By \eqref{def:gdconsistency:pdisc} we have $\PiDm v_m\to \varphi$ in $L^2(\Omega)$ and $\nablaDm v_m\to \nabla\varphi$ in $L^2(\Omega)^d$. Use $v_m$ as a test function in GS \eqref{stefan:gs} satisfied by $u_m$. The convergence properties of $\beta(\PiDm u_m)$ and $\nablaDm\zeta(u_m)$ towards $B=\beta(\bu)$ and $\nabla Z=\nabla\zeta(\bu)$, together with the convergence $\QD[m]f\to f$ in $L^2(\Omega)$ stated in \eqref{def:gdconsistency:pdisc}, enable us to take the limit $m\to\infty$ of the scheme to see that $\bu$ is a solution to \eqref{stefan:weak}. The uniqueness of $\bu$ (see Theorem \ref{th:exist.uniq}) shows that the convergence properties holds for the whole sequence $(u_m)_{m\in\N}$ instead of just along the subsequence previously extracted.

It remains to establish the strong convergence of $\nablaDm\zeta(u_m)$. We let $m\to +\infty$ in the GS \eqref{stefan:gs} with $\disc=\disc[m]$ and $v=\zeta(u_m)$, that is,
\[
 \int_\Omega \beta(\PiDm u_m) \PiDm \zeta(u_m) +\int_\Omega \Lambda\nablaDm\zeta(u_m) \cdot\nablaDm\zeta(u_m) 
=\int_\Omega \QD[m] f\,\PiDm \zeta(u_m)-\int_\Omega F\cdot\nablaD\zeta(u_m).
\]
This yields
\[
\lim_{m\to\infty}\int_\Omega \Lambda\nablaDm\zeta(u_m) \cdot\nablaDm\zeta(u_m) =\int_\Omega (f\zeta(\bu) - F\cdot\nabla\zeta(\bu)-  \beta(\bu) \zeta(\bu)) 
= \int_\Omega \Lambda\nabla\zeta(\bu) \cdot\nabla\zeta(\bu),
\]
where the conclusion follows using $\bv=\bu$ in \eqref{stefan:weak}. Since $(\xi,\eta)\mapsto \int_\Omega \Lambda\xi\cdot\eta$ is an inner product on $L^2(\Omega)$, this relation and the weak convergence of $(\nablaDm\zeta(u_m))_{m\in\N}$ imply the strong convergence of $\nablaDm\zeta(u_m)$ to $\nabla\zeta(\bu)$ in $L^2$.
\end{proof}

\subsection{Error estimate}\label{sec:errest}

The analysis above pinpoints the required structure on a numerical scheme to ensure proper bounds and the convergence of the solution --- namely, the piecewise constant reconstruction property. We now want to establish error estimates to better assess this convergence. In practice, one usually starts from a given numerical method and would like to apply it to the model under consideration. Following our discussion above, if the given method does not have a piecewise constant function reconstruction, it has to be modified into a method that has such a reconstruction. This process is called the mass-lumping of the original scheme. In the context of the GDM, this notion is translated in the following definition.

\begin{definition}[Mass-lumped GD]\label{def:mlgd} 
Let $\discs=(\XDz,\PiDs,\nablaD,\QD)$ be a Gradient Discretisation. A Gradient Discretisation $\disc$ is a mass-lumped version of $\discs$ if it only differs from $\discs$ through the function reconstruction (that is, $\disc=(\XDz,\PiD,\nablaD,\QD)$), and if $\PiD$ is a piecewise constant reconstruction in the sense of Definition \ref{def:pw}.
\end{definition}

\begin{remark}
Of course, if $\discs$ already has a piecewise constant reconstruction as in Definition \ref{def:pw}, one can take $\disc = \discs$.
\end{remark}

The following theorem states a general error estimate on the Gradient Scheme \eqref{stefan:gs}.

\begin{theorem}[Error estimate for the GS]\label{th:error.est}
Assume \eqref{eq:assum} and let $\disc$ be a mass-lumped version of a Gradient Discretisation $\discs$, in the sense of Definition \ref{def:mlgd}.
 Let $u$ be a solution to the Gradient Scheme \eqref{stefan:gs}, and let $\bu$ be the solution to \eqref{stefan:weak} (see Theorem \ref{th:exist.uniq}). Then, for any $\ID\zeta(\bu)\in\XDz$, there holds
\begin{multline}
\norm[L^2]{\nablaD [\ID\zeta(\bu) - \zeta(u)]}\\
\lesssim  W_{\discs}(\Lambda\nabla \zeta(\bu)+F) + \norm[L^2]{\nablaD  \ID\zeta(\bu)-\nabla\zeta(\bu)}
+R_{\disc,\discs}(\bu,f)+ \term_{\disc}(\bu,u),
\label{eq:error.est.new}
\end{multline}
where
\begin{equation}
 R_{\disc,\discs}(\bu,f) = \max_{v\in \XDz\backslash\{0\}}\frac{1}{\norm[L^2]{\nablaD v}}\left| \int_\Omega \PiD v [ \beta(\QD\bu) - \QD f]-  \PiDs v [\beta(\bu)-f]\right|,
\label{def:rdiscdiscs}\end{equation}
and
\begin{equation}\label{eq:termD}
\term_{\disc}(\bu,u)=\left(\max\left\{\int_\Omega\left[\beta(\QD\bu) -  \beta(\PiD u) \right] \left[\zeta(\QD\bu)-\PiD\ID\zeta(\bu)\right];0\right\}\right)^{1/2}.
\end{equation}
\end{theorem}

\begin{remark}[Choice of {$\ID\zeta(\bu)$}]
The element $\ID\zeta(\bu)$ can be any vector in $\XDz$. However, the estimate \eqref{eq:error.est.new} is obviously useful only if $\nablaD  \ID\zeta(\bu)$ is close to $\nabla\zeta(\bu)$. This is usually achieved selecting for $\ID\zeta(\bu)$ a suitable interpolate of $\zeta(\bu)$, which is why we used this notation.
\end{remark}

\begin{remark}[Approximation of {$\zeta(\bu)$}]
Introducing $\pm\nablaD(\ID\zeta(\bu))$ and using a triangle inequality, we have
\[
\norm[L^2]{\nablaD\zeta(u)-\nabla\zeta(\bu)}\le \norm[L^2]{\nablaD[\zeta(u) - \ID \zeta(\bu)]} + 
\norm[L^2]{\nablaD\ID \zeta(\bu)-\nabla\zeta(\bu)}.
\]
Similarly, introducing $\pm\PiDs \ID \zeta(\bu)$ and using the triangle inequality and the definition of $C_{\discs}$, we have
\begin{align*}
\norm[L^2]{\PiDs \zeta(u)-\zeta(\bu)}\le{}& \norm[L^2]{\PiDs[] [\zeta(u)-\ID\zeta(\bu)]}+\norm[L^2]{\PiDs \ID\zeta(\bu)-\zeta(\bu)}\\
\le{}& C_{\discs}\norm[L^2]{\nablaD [\zeta(u)-\ID\zeta(\bu)]}+\norm[L^2]{\PiDs \ID\zeta(\bu)-\zeta(\bu)}.
\end{align*}
An estimate on $\nablaD [\zeta(u)-\ID\zeta(\bu)]$ as in Theorem \ref{th:error.est} therefore also yields an estimate on $\nablaD\zeta(u)-\nabla\zeta(\bu)$ and $\PiDs \zeta(u)-\zeta(\bu)$, modulo the additional interpolation errors $\nablaD\ID \zeta(\bu)-\nabla\zeta(\bu)$ and $\PiDs \ID\zeta(\bu)-\zeta(\bu)$. If $\discs$ has function and gradient reconstructions that are piecewise polynomial of high-order, these interpolation errors can be expected to have a high rate of convergence with respect to the mesh size.

The same argument also gives an error estimate on $\PiD\zeta(u)-\zeta(\bu)$, but the corresponding interpolation error $\PiD \ID\zeta(\bu)-\zeta(\bu)$ is limited to a first-order convergence since $\PiD$ is a piecewise constant reconstruction.
\end{remark}

\begin{proof}[Proof of Theorem \ref{th:error.est}]
Since $\bu$ is the solution to \eqref{stefan:weak}, we have $\div(\Lambda \nabla\zeta(\bu)+F)=\beta(\bu)-f\in L^2(\Omega)$. Hence, by definition \eqref{def:WD} of $W_{\disc}$ applied to $\discs$, for any $v\in \XDz$,
\begin{align*}
 \norm[\disc]{v}W_{\discs}(\Lambda\nabla \zeta(\bu)+F)\ge{}&
 \int_\Omega \nablaD v\cdot (\Lambda\nabla\zeta(\bu)+F) + \PiDs v \div(\Lambda\nabla\zeta(\bu)+F)   \\
={}& \int_\Omega \nablaD v\cdot \Lambda\nabla\zeta(\bu) + F\cdot\nablaD v+ \PiDs v [\beta(\bu)-f].
\end{align*}
Substituting the term involving $F$ using \eqref{stefan:gs}, we get
\begin{multline}
\label{for.non.homo}
\int_\Omega \Lambda \nablaD v\cdot (\nabla \zeta(\bu)- \nablaD \zeta(u)) +  (\QD f -  \beta(\PiD u))\PiD v + \PiDs v [\beta(\bu)-f] \\
\le\norm[\disc]{v} W_{\discs}(\Lambda\nabla \zeta(\bu)+F).
\end{multline}
Introducing $\pm\Lambda\nablaD v\cdot\nablaD \ID\zeta(\bu)$ and $\pm \beta(\QD \bu)\PiD v$ in the left-hand side, using the Cauchy--Schwarz inequality and recalling that $\norm[\disc]{v}=\norm[L^2]{\nablaD v}$, we infer
\begin{multline}\label{est:presquefin}
\int_\Omega\Lambda \nablaD v\cdot (\nablaD \ID\zeta(\bu) - \nablaD \zeta(u)) + (\beta(\QD\bu) -   \beta(\PiD u))\PiD v \\
+ \int_\Omega \PiDs v [\beta(\bu)-f] - \PiD v [\beta(\QD\bu) - \QD f] \\
\lesssim \norm[L^2]{\nablaD v}\left[W_{\discs}(\Lambda\nabla \zeta(\bu)+F) + \norm[L^2]{\nablaD  \ID\zeta(\bu)-\nabla\zeta(\bu)}\right].
\end{multline}
Choose $v = \ID\zeta(\bu) - \zeta(u)$. Introducing $\pm\zeta(\QD\bu)$ and using the monotonicity of $\zeta$ and $\beta$
(which yields $[\beta(b)-\beta(a)][\zeta(b)-\zeta(a)]\ge 0$ for all $a,b\in\R$) together with \eqref{eq:commut}, we have
\begin{align*}
(\beta(\QD\bu) -   \beta(\PiD u))\PiD v ={}& [\beta(\QD\bu) -   \beta(\PiD u)]\,[\PiD\ID\zeta(\bu)-\zeta(\QD\bu)]\\
&+[\beta(\QD\bu) -   \beta(\PiD u)]\,[\zeta(\QD\bu)-\zeta(\PiD u)]\\
\ge{}&[\beta(\QD\bu) -   \beta(\PiD u)]\,[\PiD\ID\zeta(\bu)-\zeta(\QD\bu)].
\end{align*}
Plugging this into \eqref{est:presquefin} and using \eqref{assum:lambda} leads to
\begin{align*}
&\norm[L^2]{\nablaD [\ID\zeta(\bu) - \zeta(u)]}^2
\lesssim  \norm[L^2]{\nablaD v} \left[W_{\discs}(\Lambda\nabla \zeta(\bu)+F) + \norm[L^2]{\nablaD  \ID\zeta(\bu)-\nabla\zeta(\bu)}\right]\\
{}&\quad\quad +\int_\Omega \PiD v [\beta(\QD\bu) - \QD f] -  \PiDs v [\beta(\bu)-f] + \int_\Omega[\beta(\QD\bu) -  \beta(\PiD u)]\,[\zeta(\QD\bu)-\PiD\ID\zeta(\bu)]\\
&\quad \lesssim  \norm[L^2]{\nablaD v} \big[W_{\discs}(\Lambda\nabla \zeta(\bu)+F) + \norm[L^2]{\nablaD  \ID\zeta(\bu)-\nabla\zeta(\bu)}+R_{\disc,\discs}(\bu,f)\big]+ \term_{\disc}(\bu,u)^2.
\end{align*}
Using the Young inequality on the first term in the right-hand side and recalling that $v=\ID\zeta(\bu) - \zeta(u)$ leads to
\begin{multline*}
\norm[L^2]{\nablaD [\ID\zeta(\bu) - \zeta(u)]}^2\\
\lesssim \left[W_{\discs}(\Lambda\nabla \zeta(\bu)+F) + \norm[L^2]{\nablaD  \ID\zeta(\bu)-\nabla\zeta(\bu)} + R_{\disc,\discs}(\bu,f)\right]^2 + \term_{\disc}(\bu,u)^2.
\end{multline*}
The proof of \eqref{eq:error.est.new} is complete taking the square root of this estimate and using
$\sqrt{a^2+b^2}\le a+b$ for all $a,b\ge 0$. \end{proof}

From the general estimate \eqref{eq:error.est.new} we deduce the following bound on the error, which often leads to (low-order) rates of convergence as noted in Remark \ref{rem:rate.low}. This estimate will be improved, for situations corresponding to classical mass-lumping versions of schemes with nodal interpolates, in Section \ref{sec:GS.quad}.

\begin{corollary}\label{cor:rates}
Under the assumptions of Theorem \ref{th:error.est}, define
\[
\alpha_{\disc,\discs}=\max_{v\in\XDz\backslash\{0\}}\frac{\norm[L^2]{\PiD v-\PiDs v}}{\norm[L^2]{\nablaD v}}
\]
and let $\ID\zeta(\bu)$ be given by $\ID \zeta(\bu)=\mathop{\rm argmin}_{v\in \XDz}\left(\norm[L^2]{\nablaD v-\nabla \zeta(\bu)}+\norm[L^2]{\PiD v-\zeta(\bu)}\right)$.
Then,
\begin{equation}\label{eq:est.general}
\begin{aligned}
\norm[L^2]{\nablaD [\ID\zeta(\bu) - \zeta(u)]}
\lesssim{}&  W_{\discs}(\Lambda\nabla \zeta(\bu)+F) + S_{\disc}(\zeta(\bu))\\
&+\alpha_{\disc,\discs}+\norm[L^2]{\beta(\QD\bu) - \beta(\bu)} +\norm[L^2]{\QD f-f}\\
&+ \left(S_{\disc}(\zeta(\bu))+\norm[L^2]{\zeta(\bu)-\zeta(\QD \bu)}\right)^{\frac12},
\end{aligned}
\end{equation}
where the hidden multiplicative constant in $\lesssim$ additionally depends on $\norm[L^2]{\beta(\QD \bu)}$ and $\norm[L^2]{\QD f}$.
\end{corollary}

\begin{remark}[Rate of convergence]\label{rem:rate.low}
For all classical mass-lumping of schemes based on a mesh of size $h$, we have $\alpha_{\disc,\discs}=\mathcal O(h)$ (see, e.g., \cite[Eqs. (8.18) and (9.46)]{gdm}). Likewise, any reasonable quadrature rule is locally exact on piecewise constant functions and thus, if $\beta,\zeta$ are globally Lipschitz-continuous and $\bu,f$ are locally $H^1$, we expect
$\mathcal O(h)$ estimates on $\norm[L^2]{\beta(\QD\bu) - \beta(\bu)}$, $\norm[L^2]{\QD f-f}$ and $\norm[L^2]{\zeta(\bu)-\zeta(\QD \bu)}$. The estimate \eqref{eq:est.general} can thus be expected, most of the time, to provide an $\mathcal O(h^{\frac12})$ rate of convergence, the limiting factor in the right-hand side of \eqref{eq:est.general} being the last one, coming from $\term_{\disc}(\bu,u)$. We will see in Section \ref{sec:GS.quad} that this estimate is however very pessimistic and, in many cases, can be improved to higher powers of $h$ (see Remark \ref{rem:rate.high}).
\end{remark}

\begin{proof}
We estimate each term, except the first one, in the right-hand side of \eqref{eq:error.est.new}.
By choice of $\ID\zeta(\bu)$ and definition \eqref{def:SD} of $S_{\disc}$,
\[
\norm[L^2]{\nablaD \ID\zeta(\bu)-\nabla \zeta(\bu)}+\norm[L^2]{\PiD \ID\zeta(\bu)-\zeta(\bu)}=S_{\disc}(\zeta(\bu)).
\]
Hence the term $\norm[L^2]{\nablaD  \ID\zeta(\bu)-\nabla\zeta(\bu)}$ in \eqref{eq:error.est.new} is bounded above by $S_{\disc}(\zeta(\bu))$.

Using the definition of $\alpha_{\disc,\discs}$ and of $C_{\disc}$, we have
\begin{align*}
R_{\disc,\discs}(\bu,f)\le{}& \alpha_{\disc,\discs}\norm[L^2]{\beta(\bu)-f} \\
&+ \max_{v\in \XDz\backslash\{0\}}\frac{1}{\norm[L^2]{\nablaD v}}\left| \int_\Omega \PiD v [ \beta(\QD\bu) - \QD f]-  \PiD v [\beta(\bu)-f]\right|\\
\lesssim{}&\alpha_{\disc,\discs}+C_{\disc}\norm[L^2]{[ \beta(\QD\bu) - \QD f]-[\beta(\bu)-f]}\\
\lesssim{}&\alpha_{\disc,\discs}+\norm[L^2]{\beta(\QD\bu)-\beta(\bu)}+\norm[L^2]{\QD f -f}.
\end{align*}
This gives the third, fourth, and fifth terms in the right-hand side of \eqref{eq:est.general}.

For the last term in this estimate, we write, using the Cauchy--Schwarz inequality and the a priori bound \eqref{eq:estim} on $\beta(\PiD u)$,
\begin{align*}
\term_{\disc}(\bu,u)^2\le{}& \norm[L^2]{\beta(\QD\bu)-\beta(\PiD u)}\norm[L^2]{\PiD\ID\zeta(\bu)-\zeta(\QD\bu)}\\
\lesssim{}& (\norm[L^2]{\beta(\QD\bu)}+\norm[L^2]{\QD f}+\norm[L^2]{F})\left(\norm[L^2]{\PiD\ID\zeta(\bu)-\zeta(\bu)}+\norm[L^2]{\zeta(\bu)-\zeta(\QD\bu)}\right).
\end{align*}
The proof is complete taking the square root and recalling that $\norm[L^2]{\PiD\ID\zeta(\bu)-\zeta(\bu)}\le S_{\disc}(\zeta(\bu))$. \end{proof}

\subsection{Suitable quadrature rules lead to high-order estimates}\label{sec:GS.quad}

Let us first make the following broken regularity assumption on the data and solution.

\begin{assumption}[Data and exact solution]\label{assum:data.bu}
 $F=0$ and, $\bu$ being the solution to \eqref{stefan:weak} and $s\ge 1$ being an integer, $f$ and $\beta(\bu)$ belong to the broken Sobolev space
\[
W^{s,\infty}(\mesh):=\Big\{g\in L^\infty(\Omega)\,:\,g_{|K}\in W^{s,\infty}(K)\quad \forall K\in\mesh\Big\}.
\]
This space is endowed with the norm $\norm[W^{s,\infty}(\mesh)]{g}:=\max_{K\in\mesh} \norm[W^{s,\infty}(K)]{g}$. 
\end{assumption}

\begin{remark}[Piecewise continuity and local smoothness]\label{rem:gWs.cont}
$W^{s,\infty}(\mesh)$ is a subspace of 
\begin{equation}\label{def:Cmesh}
C(\mesh):=\{g\in L^\infty(\Omega)\,:\,g_{|K}\in C(\overline{K})\quad\forall K\in\mesh\}.
\end{equation}
Assumption \ref{assum:data.bu} only imposes a local smoothness of $f$ and $\beta(\bu)$, which can in particular be discontinuous across cell interfaces.
It is also worthwhile noticing that, since $F=0$, $\zeta(\bu)$ is continuous (see Theorem \ref{th:exist.uniq}). Hence, the values of $\bu$ at one of its discontinuities must belong to a plateau of $\zeta$; in particular, if $\zeta$ does not have any plateau, then $\bu$ is globally continuous.

Note that if $F\not=0$ (and $F$ is not smooth), as in Test Case \refS{tc:casundstefhmun} in Section \ref{sec:num}, $\zeta(\bu)$ is not expected to have any additional regularity beyond $H^1$ and therefore high-order estimates, even if they can theoretically be established, are of little use. Actually, numerical tests show that high-order schemes can deliver estimates that are no better than low-order schemes, see e.g. Table \ref{tab:casundstefhmun}.
\end{remark}

\begin{remark}[$\bu$ is in $C(\mesh)$]\label{rem:uCmesh}
Theorem \ref{th:exist.uniq} and Assumption \ref{assum:data.bu} show that $\zeta(\bu)\in C(\overline{\Omega})$ and $\beta(\bu)\in C(\mesh)$. By \eqref{assum:zeta}--\eqref{assum:incr}, $\beta+\zeta$ is a homeomorphism of $\R$. Hence, $\bu=(\beta+\zeta)^{-1}(\beta(\bu)+\zeta(\bu))$ and thus $\bu \in C(\mesh)$.
\end{remark}

In the rest of this section, we consider a slightly more precise setting than in Section \ref{sec:errest}. We assume that $\discs$ has a piecewise polynomial function reconstruction (possibly of high-order) and unknowns associated to nodes in the domain, and that specific local quadrature rules can be chosen. Typically, $\Poly{k}$ or $\mathbb{Q}^k$ Finite Elements and  Symmetric Interior Penalty discontinuous Galerkin (SIPG) schemes, with mass-lumping constructed using dual meshes around the nodes, fit into this setting. In what follows, $h_X$ denotes the diameter of a set $X\subset \R^d$.

\begin{assumption}[Structure of $\discs$, $\disc$ and $\ID\zeta(\bu)$]\label{assum:discs} ~
\begin{enumerate}[\hspace*{.4em}(1)]
\item \emph{(Mesh)} $\Omega\subset \R^d$ (with $d\le 3$) is a polytopal open set and $\mesh$ is a polytopal mesh of $\Omega$, in the sense of \cite[Definition 7.2]{gdm} (this definition actually represents the mesh as a quadruple of sets of cells, faces, points and vertices, that will not be useful to our purpose; we therefore confuse the mesh with the set of cells). The mesh size is $h=\max_{K\in\mesh}h_K$.
\item \emph{(Space)} There is a finite set $I$, partitioned into $I_{\Omega}$ and $I_{\partial\Omega}$, such that
\[
\XDz=\left\{v=(v_i)_{i\in I}\,:\, v_i\in \R\quad\forall i\in I\,,\; v_i=0\quad\forall i\in I_{\partial\Omega}\right\}.
\]
\item\label{PK} \emph{(Local polynomial reconstructions)} There is a polynomial degree $k\ge 1$ such that, for all $K\in\mesh$ and all $v\in \XDz$, $(\PiDs v)_{|K}\in \Poly{k}$.
\item\label{bound.grad} \emph{(Broken gradient bound)} There is $C_{\nabla}\ge 0$ such that, for all $v\in\XDz$, $\norm[L^2]{\nabla_h(\PiDs v)}\le C_{\nabla} \norm[L^2]{\nablaD v}$, where $\nabla_h$ is the usual broken gradient on $\mesh$.
\item\label{nodes} \emph{(Nodes)} There is a family $(\x_i)_{i\in I}$ of points in $\overline{\Omega}$, and subsets $(I_K)_{K\in\mesh}$ of $I$, such that $I=(\cup_{K\in\mesh}I_K)\cup I_{\partial\Omega}$ and, for all $v=(v_i)_{i\in I}\in\XDz$, all $K\in\mesh$ and all $i\in I_K$, we have $\x_i\in\overline{K}$ and $v_i=(\PiDs v)_{|K}(\x_i)$. Additionally, $\x_i\in\partial\Omega$ whenever $i\in I_{\partial\Omega}$.
\item\label{UiK} \emph{(Mass-lumping)} $\disc=(\XDz,\PiD,\nablaD,\QD)$ is a mass-lumped version of the Gradient Discretisation $\discs=(\XDz,\PiDs,\nablaD,\QD)$ in the sense of Definition \ref{def:mlgd}, which means that $\PiD$ is piecewise constant on a partition $\pwpart=(\pwpart_i)_{i\in I}$ in the sense of Definition \ref{def:pw}. We further assume that $\pwpart_i\cap K\neq \emptyset$ only if $i \in I_K$.
\item\label{def:ID} \emph{(Interpolate)} $\ID\zeta(\bu)\in \XDz$ is given by the nodal values of $\zeta(\bu)$, that is, $(\ID\zeta(\bu))_i=\zeta(\bu)(\x_i)$ for all $i\in I$. This is well-defined since $\zeta(\bu)\in C(\overline{\Omega})$ (see Remark \ref{rem:gWs.cont}).
\item\label{def:QD} \emph{(Quadrature rule)} The quadrature $\QD$ is defined on $C(\mesh)$ (see \eqref{def:Cmesh}) by
\begin{equation}\label{eq:QD.nodes}
\forall g\in C(\mesh)\,,\;\forall K\in\mesh\,,\quad (\QD g)_{|K}=\sum_{i\in I_K} g_{|K}(\x_i)\mathbf{1}_{\pwpart_i\cap K}.
\end{equation}
\item \label{assum:rho} \emph{(Mesh regularity)} There exists $\rho>0$ such that:
\begin{itemize}[\hspace*{0em}$\star$]
\item Any $K\in\mesh $ is star-shaped with respect to all points in a ball of radius $\rho h_K$,
\item For all $i\in I$, $\rho h_{\pwpart_i}\le h$.
\end{itemize}
\end{enumerate}
\end{assumption}

A few remarks are of order.

\begin{remark}[Local polynomial space]\label{rem:replace.PK}
The space $\Poly{k}$ in Item (\ref{PK}) could be replaced by any of its subspace $P_K$ that contains $\Poly{1}$; the analysis would not be hindered, and some assumptions could even be weakened (see Remark \ref{rem:PK.quad}). We chose to use $\Poly{k}$ to simplify the presentation.
\end{remark}

\begin{remark}[Nodes]
The same $i$ can belong to several $I_K$, as is the case for conforming Finite Elements. 
Conversely,  in the case of DG schemes for example, the following may occur (see the numerical example in Section \ref{sec:DG1D}):
\begin{itemize}[\hspace*{.4em}$\bullet$]
 \item one can have $\x_i=\x_j$ for $i\neq j$,
 \item $I_K$ does not necessarily contain all the indices $i\in I$ such that $\x_i\in\overline{K}$,
 \item there can exist $i\in I_{\partial\Omega}\setminus ( \cup_{K\in\mesh}I_K)$ --- but, in that case, $\pwpart_i = \emptyset$.
\end{itemize}
\end{remark}

\begin{remark}[Quadrature rule]
The Gradient Scheme \eqref{stefan:gs} is usually implemented by assembling cell contributions. When the source term $f$ is continuous on each cell, and since $\PiD v$ is constant on each $\pwpart_i$, it is customary to use the simple --- apparently low-order --- quadrature rule defined by \eqref{eq:QD.nodes}. 

We also note that, since $f$ and $\bu$ belong to $C(\mesh)$ by Remarks \ref{rem:gWs.cont} and \ref{rem:uCmesh}, the formula \eqref{eq:QD.nodes} can be used to compute $\QD f$ and $\QD \bu$. These are the only values of $\QD$ of interest in the following analysis.
\end{remark}

In the rest of this section, we write $a\lesssim b$ as a shorthand for ``$a\le Cb$ with $C$ not depending on $\mesh$ or $\pwpart$, but possibly depending on $\rho$,  $k$ and $C_{\nabla}$''.

\begin{theorem}[High-order error estimate] \label{th:high.order}
Under Assumption \ref{assum:discs}, let $\ell\ge 0$ be an integer and suppose that the local quadrature rules defined by $\QD$ are exact at degree $k+\ell$ (where $k$ is the degree in Item \eqref{PK} of Assumption \ref{assum:discs}), that is:
\begin{equation}\label{assum:quad.qr}
\forall K\in\mesh\,,\forall q\in \Poly{k+\ell}\,,\quad \int_K q= \int_K \QD q=\sum_{i\in I_K} |\pwpart_i\cap K|q(\x_i).
\end{equation}
Let $s\in\{1,\ldots,\ell+2\}$ be such that Assumption \ref{assum:data.bu} holds. Then, 
the solution $u$ to \eqref{stefan:gs} satisfies
\begin{multline}\label{eq:improved.disc.est}
\norm[L^2]{\nablaD [\ID \zeta(\bu) - \zeta(u)]}\\
\lesssim  W_{\discs}(\Lambda\nabla\zeta(\bu))+\norm[L^2]{\nablaD \ID \zeta(\bu)-\nabla\zeta(\bu)} + h^{s}(1+C_{\discs})\norm[W^{s,\infty}(\mesh)]{\beta(\bu)-f}.
\end{multline}
\end{theorem}

Let us make a few remarks.

\begin{remark}[Quadrature rule]\label{rem:quadrule}
The quadrature rule \eqref{assum:quad.qr} bears similarities with the conditions on quadrature rules highlighted for Finite Elements in \cite{ciarlet,CJRT}. However, in the proof below, the exactness condition \eqref{assum:quad.qr} responds to a different need than the ones encountered in the analysis of mass-lumped Finite Elements for linear equations. 

We also note that the precise geometry of the sets $\pwpart_i$ is not important as long as \eqref{assum:quad.qr} holds. This is due to the fact that, in the scheme \eqref{stefan:gs}, given the definitions \eqref{def:PiD.ML} and \eqref{eq:QD.nodes} of $\PiD$ and $\QD$, these sets $\pwpart_i$ only appear through the quantities $|\pwpart_i\cap K|$.
\end{remark}

\begin{remark}[Local polynomial space]\label{rem:PK.quad}
Following Remark \ref{rem:replace.PK}, if $\Poly{k}$ is replaced by $P_K$ in Item (\ref{PK}) of Assumption \ref{assum:discs}, then an inspection of the proof below (see in particular the polynomial \eqref{eq:design.q}) shows that \eqref{assum:quad.qr} only has to be assumed for $q$ belonging to the smaller space $P_K \Poly{l}$. This is similar to what has been noticed in \cite{GMV}, in the context of mass-lumped $\Poly{k}$ Finite Elements for linear equations.
\end{remark}

\begin{remark}[Rates of convergence]\label{rem:rate.high}
If $\discs$ is the Gradient Discretisation corresponding to conforming $\Poly{k}$ Finite Elements, we have $W_{\discs}\equiv 0$ and, if $\zeta(\bu)\in H^{k+1}(\Omega)$, $\norm[L^2]{\nablaD \ID \zeta(\bu)-\nabla\zeta(\bu)}\lesssim  h^{k}$; see \cite[Proposition 8.11 and Remark 8.12]{gdm}. In this case, \eqref{eq:improved.disc.est} yields an $\mathcal O(h^{\min(s,k)})$ estimate on $\norm[L^2]{\nablaD [\ID \zeta(\bu) - \zeta(u)]}$, which is a drastic improvement over \eqref{eq:est.general} (see Remark \ref{rem:rate.low}).

The same $\mathcal O(h^{\min(s,k)})$ bound on $\norm[L^2]{\nablaD [\ID \zeta(\bu) - \zeta(u)]}$ holds for the Gradient Discretisation corresponding to DG schemes of degree $k$, provided that $\Lambda\nabla\zeta(\bu)\in H^{\min(s,k)}(\Omega)^d$ (see \cite[Lemmas 11.14 and 11.15]{gdm}).
\end{remark}

Before proving Theorem \ref{th:high.order}, we describe in Tables \ref{tab:quadratures.1D} and \ref{tab:quadratures.2D} a few examples of choices of $(\x_i,\pwpart_i\cap K)_{i\in I_K}$ that satisfy \eqref{assum:quad.qr} in dimensions one and two. These rules will be used in the numerical tests in Section \ref{sec:num}, and they assume that the cell $K$ is a simplex (interval if $d=1$, triangle if $d=2$). Note that some of these rules are sub-optimal in terms of degree of exactness vs.\ number of quadrature points; they will serve to illustrate the optimality of Theorem \ref{th:high.order}.

\begin{table}[ht]
\resizebox{\textwidth}{!}{
\setlength\extrarowheight{3pt}
\begin{tabular}{|c|c|c|c|c|}
\hline
Name & $(\x_i)_{i\in I_K}$ & $(|\pwpart_i\cap K|)_{i\in I_K}$ & DOE & Illustration\\[.2em]
\hline
Trapezoidal & $(a,b)$ & $(\frac12 |K|,\frac12 |K|)$ & 1 & \includegraphics[width=0.3\linewidth]{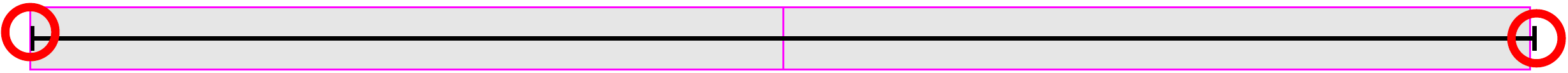}\\[.2em]
\hline
Simpson & $(a,\frac{a+b}{2},b)$ & $(\frac16 |K|, \frac23 |K|, \frac16 |K|)$ & 3 &\includegraphics[width=0.3\linewidth]{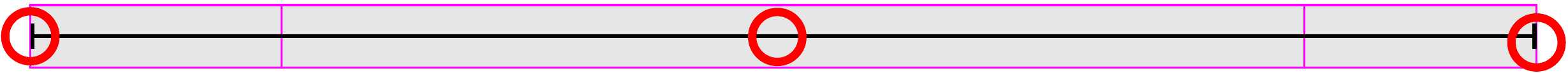}\\[.2em]
\hline
Equi6 & $(a,\frac{2a+b}{3},\frac{a+2b}{3},b)$ &  $(\frac16 |K|, \frac13 |K|, \frac13|K|, \frac16 |K|)$ & 1&\includegraphics[width=0.3\linewidth]{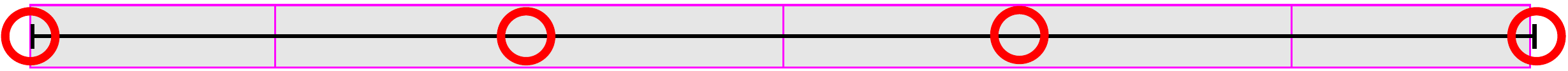}\\[.2em]
\hline
Equi8 & $(a,\frac{2a+b}{3},\frac{a+2b}{3},b)$ &  $(\frac18 |K|, \frac38 |K|, \frac38|K|, \frac18 |K|)$ & 3&\includegraphics[width=0.3\linewidth]{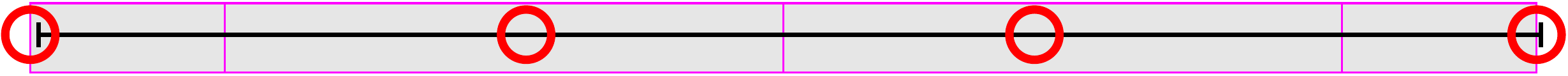}\\[.2em]
\hline
Gauss--Lobatto & $(a,\frac{5+\sqrt{5}}{10}a+\frac{5-\sqrt{5}}{10}b,\frac{5-\sqrt{5}}{10}a+\frac{5+\sqrt{5}}{10}b,b)$ & $(\frac{1}{12}|K|,\frac{5}{12}|K|,\frac{5}{12}|K|,\frac{1}{12}|K|)$ & 5&\includegraphics[width=0.3\linewidth]{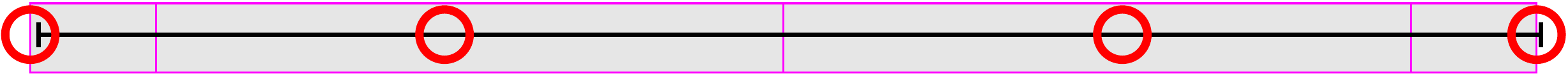}\\[.2em]
\hline
\end{tabular}}

\caption{Examples of quadrature rules satisfying \eqref{assum:quad.qr} in dimension $d=1$, with $K=(a,b)$. DOE stands for degree of exactness, and corresponds to $k+\ell$ in \eqref{assum:quad.qr}. In the illustrations, the circles represent the nodes $\x_i$ and the sets $\pwpart_i\cap K$ are the intervals delimited by vertical bars.}
\label{tab:quadratures.1D}
\end{table}

\begin{table}
\resizebox{\textwidth}{!}{\begin{tabular}{|c|c|c|c|c|}
\hline
Name & $(\x_i)_{i\in I_K}$ & $(|\pwpart_i\cap K|)_{i\in I_K}$ & DOE & Illustration\\[.2em]
\hline
Vertex & $(\mathbf{a},\mathbf{b},\mathbf{c})$ & $(\frac13|K|,\frac13|K|,\frac13|K|)$ & 1&\includegraphics[width=0.1\linewidth]{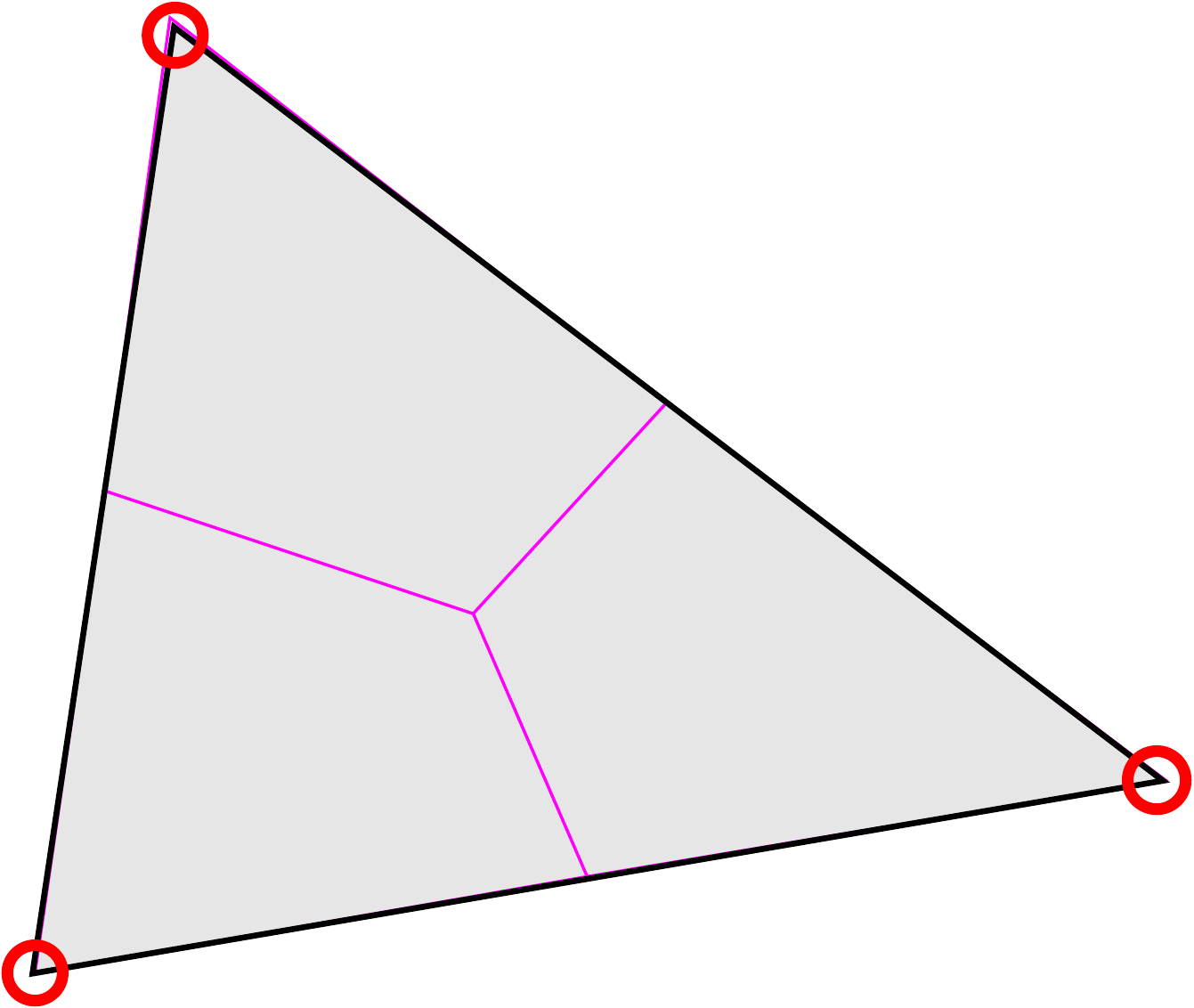}\\[.2em]
\hline
Vertex+Edge Midpoint &  $(\mathbf{a},\mathbf{b},\mathbf{c},\frac{\mathbf{a+b}}{2},\frac{\mathbf{a+c}}{2},\frac{\mathbf{b+c}}{2})$ &
$(0,0,0,\frac13|K|,\frac13|K|,\frac13|K|)$ & 2&\includegraphics[width=0.1\linewidth]{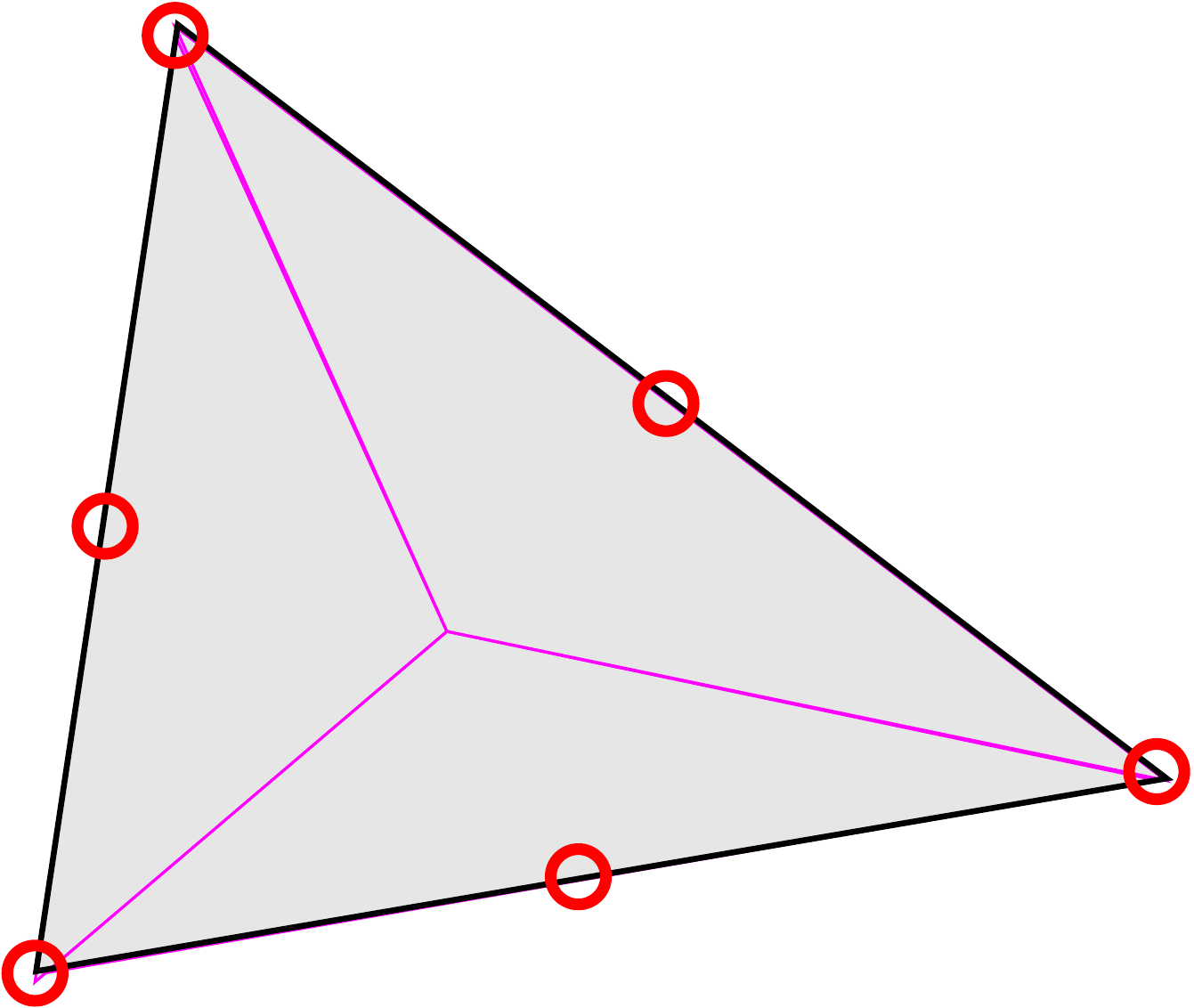}\\[.2em]
\hline
\end{tabular} }
\caption{Examples of quadrature rules satisfying \eqref{assum:quad.qr} in dimension $d=2$, with $K$ triangle $(\mathbf{a},\mathbf{b},\mathbf{c})$. DOE stands for degree of exactness, and corresponds to $k+\ell$ in \eqref{assum:quad.qr}. In the illustrations, the nodes $\x_i$ are the circles and the sets $\pwpart_i\cap K$ are the regions delimited by straight lines.}
\label{tab:quadratures.2D}
\end{table}

\begin{proof}[Proof of Theorem \ref{th:high.order}]
The inequality \eqref{eq:improved.disc.est} follows from Theorem \ref{th:error.est}, estimating in the present context the terms $\term_{\disc}(\bu,u)$ and $R_{\disc,\discs}(\bu,f)$.

\begin{inparaenum}[(i)]
\item \emph{Term $\term_{\disc}(\bu,u)$}. For all $K\in\mesh$, all $i\in I_K$, and all $\x\in\pwpart_i\cap K$, Definition \ref{def:pw} of $\PiD$ and Items \eqref{def:ID} and \eqref{def:QD} in Assumption \ref{assum:discs} imply that
\[
\PiD\ID\zeta(\bu)(\x)=(\ID\zeta(\bu))_i=\zeta(\bu)(\x_i)=(\zeta(\bu))_{|K}(\x_i)=\zeta(\bu_{|K}(\x_i))=
\zeta(\QD\bu(\x)).
\]
Hence, $\PiD\ID\zeta(\bu)=\zeta(\QD\bu)$ and $\term_{\disc}(\bu,u)=0$.

\item \emph{Term $R_{\disc,\discs}(\bu,f)$}. For the sake of brevity, set $g = \beta(\bu)-f$. By definition \eqref{eq:QD.nodes} of $\QD$, we have $\QD g=\beta(\QD \bu)-\QD f$ and thus, to bound $R_{\disc,\discs}(\bu,f)$ above by the last term in \eqref{eq:improved.disc.est}, we have to
establish that, for all $v\in\XDz$,
\begin{equation}\label{eq:gPiDs}
\left| \int_\Omega (\QD g\, \PiD v- g\,\PiDs v)\right|
\lesssim h^{s}(1+C_{\discs})\norm[W^{s,\infty}(\mesh)]{g}\norm[L^2]{\nablaD v}.
\end{equation}
Let $A_{\disc,\discs}(g,v)$ be the integral in the left-hand side of \eqref{eq:gPiDs}. We have
\begin{align}
A_{\disc,\discs}{}&(g,v):=\sum_{K\in\mesh}\left(\sum_{i\in I_K} |U_i\cap K|g_{|K}(\x_i)v_i-\int_K g\,\PiDs v \right)\nonumber\\
={}&\sum_{K\in\mesh}\left(\sum_{i\in I_K} |U_i\cap K|g_{|K}(\x_i)(\PiDs v)_{|K}(\x_i)-\int_K g\,\PiDs v \right)
=\sum_{K\in\mesh}\mathfrak{E}_K(g\PiDs v),
\label{est:RD.1}
\end{align}
where, in the second line, we have used $(\PiDs v)_{|K}(\x_i)=v_i$ (see Item \eqref{nodes} in Assumption \ref{assum:discs}), and we have defined the error in the local quadrature rule on $K$ by
\[
\forall w\in C(\overline{K})\,,\quad\mathfrak{E}_K(w):=\sum_{i\in I_K} |U_i\cap K|w_{|K}(\x_i)-\int_K w.
\]
By \eqref{assum:quad.qr} and a straightforward estimate,
\begin{align}
\label{prop:fIK.exact}
&\forall q\in \Poly{k+\ell}\,,\quad \mathfrak{E}_K(q)=0,\mbox{ and}\\
\label{prop:fIK.bound}
&\forall w\in C(\overline{K})\,,\quad |\mathfrak{E}_K(w)|\le 2|K|\norm[L^\infty(K)]{w}.
\end{align}
For a polynomial degree $r\ge 0$, let $\proj[K]^r:L^2(K)\to \Poly{r}$ denote the $L^2(K)$-orthogonal projector on $\Poly{r}$ and notice that, since $(\PiDs v)_{|K}\in\Poly{k}$ (Item \eqref{PK} in Assumption \ref{assum:discs}) and $k\ge 1$, the function
\begin{equation}\label{eq:design.q}
q:=(\proj[K]^\ell g)(\PiDs v)_{|K} + (\proj[K]^0 (\PiDs v)_{|K})(\proj[K]^{\ell+1}g-\proj[K]^\ell g)
\end{equation}
belongs to $\Poly{\ell+k}+\Poly{0+\ell+1}\subset  \Poly{k+\ell}$.
Using \eqref{prop:fIK.exact} with this $q$ yields
\begin{align*}
\mathfrak{E}_K(g\PiDs v)={}&\mathfrak{E}_K\left(g\PiDs v-(\proj[K]^\ell g)(\PiDs v)_{|K} - (\proj[K]^0 (\PiDs v)_{|K})(\proj[K]^{\ell+1}g-\proj[K]^\ell g)\right)\\
={}&\mathfrak{E}_K\left([g-\proj[K]^\ell g][(\PiDs v)_{|K} - \proj[K]^0 (\PiDs v)_{|K}]+(\proj[K]^0(\PiDs v)_{|K})[g-\proj[K]^{\ell+1}g]\right).
\end{align*}
Invoking then the bound \eqref{prop:fIK.bound} and the straightforward estimate $\norm[L^\infty(K)]{\proj[K]^0 (\PiDs v)_{|K}}\le \norm[L^\infty(K)]{\PiDs v}$, we infer
\begin{equation}
\begin{aligned}
|\mathfrak{E}_K(g\PiDs v)|\le{}& 2 \norm[L^\infty(K)]{g-\proj[K]^\ell g}|K|\norm[L^\infty(K)]{(\PiDs v)_{|K}- \proj[K]^0 (\PiDs v)_{|K}}\\
&+2|K|\norm[L^\infty(K)]{\PiDs v}\norm[L^\infty(K)]{g-\proj[K]^{\ell+1}g}.
\end{aligned}
\label{est:fIK.1}
\end{equation}
Under Item \eqref{assum:rho} of Assumption \ref{assum:discs}, \cite[Lemma 3.4]{DD15} shows that, for any natural numbers $a\ge 0$ and $b\in \{0,\ldots,a+1\}$, and any $w\in W^{b,\infty}(K)$,
\[
\norm[L^{\infty}(K)]{w-\proj[K]^a w}\lesssim h_K^{b}\norm[W^{b,\infty}(K)]{w}.
\]
Applying this estimate with $(a,b,w)=(\ell,\min(s,\ell+1),g)$, $(a,b,w)=(0,1,(\PiDs v)_{|K})$ and $(a,b,w)=(\ell+1,s,g)$, \eqref{est:fIK.1} leads to
\begin{equation*}
\begin{aligned}
|\mathfrak{E}_K(g\PiDs v)|\lesssim{}& h_K^{\min(s,\ell+1)} \norm[W^{\min(s,\ell+1),\infty}(K)]{g}
|K|h_K\norm[L^\infty(K)]{\nabla(\PiDs v)_{|K}}\\
&+|K|\norm[L^\infty(K)]{\PiDs v}h_K^{s}\norm[W^{s,\infty}(K)]{g}.
\end{aligned}
\end{equation*}
The discrete inverse Lebesgue embedding of \cite[Lemma 5.1]{DD15} gives, if $q\in\Poly{k}(K)$,
$|K|\norm[L^\infty(K)]{q}\lesssim |K|^{\frac12}\norm[L^2(K)]{q}$. Applied to $q=(\PiDs v)_{|K}$ and $q=$ components of $\nabla(\PiDs v)_{|K}$, and since $\min(s,\ell+1)+1=\min(s+1,\ell+2)\ge s$, we obtain
\begin{equation*}
|\mathfrak{E}_K(g\PiDs v)|\lesssim h_K^{s} \norm[W^{s,\infty}(K)]{g}|K|^{\frac12}\left(
\norm[L^2(K)]{\nabla(\PiDs v)_{|K}}+\norm[L^2(K)]{\PiDs v}\right).
\end{equation*}
Plugging this estimate into \eqref{est:RD.1}, using a discrete Cauchy--Schwarz inequality on the sums, and recalling Item \eqref{bound.grad} in Assumption \ref{assum:discs}, we obtain
\[
|A_{\disc}(g,v)|\lesssim h^s \norm[W^{s,\infty}(\mesh)]{g}\left(\norm[L^2]{\nablaD v}+\norm[L^2]{\PiDs v}\right).
\]
The estimate \eqref{eq:gPiDs} follows recalling the definition \eqref{def:CD} of $C_{\discs}$. 
\end{inparaenum}
\end{proof}

\section{Numerical illustrations}\label{sec:num}

In this section, we present numerical tests to explore the optimality of the estimate in Theorem \ref{th:high.order}, and the necessity of the condition \eqref{assum:quad.qr} on the chosen quadrature rules. This exploration will be conducted using mass-lumped Finite Elements (FE) and mass-lumped SIPG Discontinuous Galerkin (DG) schemes. As seen in Remark \ref{rem:rate.high}, when $\zeta(\bu)$ is smooth enough, the expected rate of convergence of these schemes is $h^{\min(\ell+2,k)}$, where $k$ is the degree of the underlying FE or DG scheme. We will illustrate through examples that this rate can be optimal, and that if \eqref{assum:quad.qr} is not even satisfied for $\ell=0$ then the rate of convergence falls to $h$ (basic order one convergence for mass-lumped schemes, see \cite[Sections 8.4 and 9.6]{gdm}). This illustration will be performed on both the porous medium equation and the Stefan model, in a variety of situations: with or without forcing term $f$ (the latter being closer to genuine physical models), and also in the case where the right-hand side contains a term $\div F$ (in which case the convergence is hindered by the lack of regularity of $\zeta(\bu)$, see Remark \ref{rem:gWs.cont}).

In the following, each considered mass-lumped Gradient Discretisation $\disc$ shares the same elements $(\XDz,\nablaD,\ID,\QD)$ as the corresponding $\discs$. We therefore start by describing the non-mass-lumped $\discs$, after which, in the context of Assumption \ref{assum:discs}, $\disc$ is completely determined by specifying the particular choices of nodes $(\x_i)_{i\in I_K}$ and weights $(|\pwpart_i\cap K|)_{i\in I_K}$ for each cell $K$, that is, of the local quadrature rules \eqref{assum:quad.qr}. The rules described in Tables \ref{tab:quadratures.1D} and \ref{tab:quadratures.2D} will serve as examples to construct the mass-lumped Gradient Discretisations $\disc$.

\subsection{Setting for the tests}

The convergences are assessed through the following quantities: 
\begin{equation*}
\begin{aligned}
&E_{\beta,\ID}^{\Pi} = \norm[L^2(\Omega)]{\beta(\QD\bu) -  \PiD\beta(u)},&&E_{\zeta,\ID}^{\Pi} = \norm[L^2(\Omega)]{\PiD (\ID \zeta(\bu) -  \zeta(u))},\\
&E_{\zeta,\ID}^{\nabla} = \norm[L^2(\Omega)]{\nablaD( \ID \zeta(\bu) -  \zeta(u))},&&E_{\zeta}^{\nabla} = \norm[L^2(\Omega)]{\nabla \zeta(\bu) -  \nablaD\zeta(u)}.
\end{aligned}
\end{equation*}
measuring approximation errors on $\beta(\QD\bu)$, the interpolation of $\zeta(\bu)$ (for both function and gradient reconstruction), and on $\nabla\zeta(\bu)$ using high-order quadrature rules.
A first order polynomial fit is done on the logarithms of these errors with respect to $-\frac 1 d\log({\rm Card}(I))$, which yields an approximation under the form
\[
 E \simeq C {\rm Card}(I)^{-\alpha/d}.
\]
Our outputs give the numerical values of $C$ and $\alpha$, the latter providing a numerical convergence order with respect to an evaluation of the mesh size (the number of unknowns, ${\rm Card}(I)$, growing linearly with the number of cells).

\medskip

All the 1D and 2D tests refer to the following situations: $\Lambda={\rm Id}$, $\beta = \rm{Id}$, and $\zeta\in \{{\rm Id},\zeta_p,\zeta_s\}$, where the ``porous medium'' function $\zeta_p$ is defined by
\begin{equation*}
\forall s\in\R,\ \zeta_p(s) = \max(s,0)^2,
\end{equation*}
and the ``Stefan'' function $\zeta_s$ is defined by 
\begin{equation*}
\forall s\in\R,\ \zeta_s(s) = \begin{cases} s&\hbox{ if }s<0,\\ 0 &\hbox{ if }0\le s\le 1,\\ s-1&\hbox{ if }1<s.\end{cases}
\end{equation*}
In all the numerical tests, the approximate solution remains numerically bounded. There is therefore no need to re-define $\zeta_p$ on the negative axis in order to explicitly satisfy the super-linear bound in Assumption \eqref{assum:zeta}. Let us now give the complete continuous cases which are approximated below in 1D or in 2D.
The profiles of the corresponding exact solutions are presented in Figure \ref{fig:profiles.1D}.

\tcaseR{tc:casundcinfty}{Regular problem, $f\neq 0$, $F=0$} This problem corresponds to $\zeta = \mathrm{Id}$ (the model is therefore linear, but we still apply the mass-lumping process) and, for $x\in (0,1)$, the source term and solution are given by $f(x) = 4xe^x $ and $\bu(x) = x(1-x)e^x$. 

\tcaseP{tc:casundporfnz}{Porous medium problem, homogeneous Dirichlet BC, $f\neq 0$, $F=0$}
This test is on the porous medium equation, with $\zeta = \zeta_p$. The source term and exact solutions are defined as follows: for $x\in (0,1)$, setting $y_x=\max(x-0.2,0)$ and $z_x=\max(0.8- x,0)$, we take 
\[
f(x) = (y_xz_x)^{3/2}-6y_xz_x(z_x^2-3y_xz_x+y_x^2)\quad\mbox{ and }\quad\bu(x) = (y_xz_x)^{3/2}.
\]

\tcaseP{tc:casundporfz}{Porous medium problem, non-homogeneous Dirichlet BC, $f=0$, $F=0$}
Still taking for $\zeta$ the porous medium function $\zeta = \zeta_p$, this test takes $\bu(x) = \max(x-\frac 1 5,0)^2/12$ for $x\in (0,1)$, which corresponds to the source term $f = 0$, and non-homogeneous Dirichlet boundary conditions are imposed on $\zeta(\bu)$. 

\tcaseS{tc:casundstefnz}{Stefan problem, homogeneous Dirichlet BC, $f\neq 0$, $F=0$}
In this test, the nonlinearity is given by the Stefan-like function $\zeta = \zeta_s$. The source term is given by $f(x) = 3 (\frac 1 2 - g(x))$ for $x\in (0,1)$, where $g(x) = |\frac 1 2-x|$.
To describe $\bu$, we first let $\gamma\in(0,\frac 1 2)$ such that $\bu(x)=f(x)$ for $g(x)\in(\gamma,\frac 1 2)$ and $\bu(x)\ge 1$ for $g(x)\in(0,\gamma)$. The ODE in \eqref{stefan:strong} can then be solved on each sub-interval and gives $\zeta(\bu(x)) = 0$ for $g(x)\in(\gamma,\frac 1 2)$ and, for some $a,b \in\R$,
$\zeta(\bu(x)) = a e^{g(x)} + b e^{-{g(x)}} + 3(\frac 1 2 - {g(x)})-1$ for $g(x)\in(0,\gamma)$.
Hence, $\bu(x) = a e^{g(x)} + b e^{-{g(x)}} + 3(\frac 1 2 - {g(x)})$ for $g(x)\in(0,\gamma)$.
These values $a$, $b$ and $\gamma$ are found by expressing the matching conditions ensuring that $\zeta(\bu)\in H^2(0,1)$ (since $(\zeta(\bu))''=\bu-f\in L^2(0,1)$), namely $\zeta(\frac 1 2\pm\gamma) = 0$ and $\zeta'(\frac 1 2\pm\gamma) = 0$; the symmetry of the problem also imposes $\zeta'(\frac 1 2) = 0$. This leads to the following equations:
\[
 3\left(\frac 1 2 - \gamma\right)-1 + a e^\gamma  + b e^{-\gamma} = 0,\quad -3 + a e^\gamma-b e^{-\gamma} = 0, \quad\mbox{ and }\quad -3 + a-b = 0.
\]
Numerically solving this nonlinear system of equations gives $\gamma\simeq 0.33036$, $a\simeq 1.2545$ and $b\simeq -1.7455$.

\tcaseS{tc:casundstefz}{Stefan problem, non-homogeneous Dirichlet BC, $f =  0$, $F=0$}

As in the previous test, $\zeta = \zeta_s$. The source term is fixed at $f=0$ and, for any $\gamma\in [0,1]$, a solution is given by: 
\[
\bu(x)=\left\{
\begin{array}{ll}
\frac 1 2 ( e^{x-\gamma} + e^{-(x-\gamma)})=\cosh(x-\gamma) &\forall x\in (\gamma,1),\\
0&\forall x\in (0,\gamma).
\end{array}
\right.
\]
Non-homogeneous Dirichlet conditions are imposed at $x=1$ to match the value of $\bu$ there, and the tests are run with $\gamma = \frac 1 3$.

\tcaseS{tc:casundstefhmun}{Stefan problem, homogeneous Dirichlet BC, $f\neq 0$ and $F\neq 0$}
We let $\zeta = \zeta_s$ and
\begin{equation}
 \begin{array}{llll}
 f(x) = 0\,, & F(x) = 4\frac {\sinh(1/4)} {\cosh(1/4)}\,, & \bu(x) = 0 & \quad\forall x\in(0,\frac 1 4);\\
 f(x) = 5\,, & F(x) = 0\,,                                & \bu(x) = 5-4\frac {\cosh(x-1/2)} {\cosh(1/4)}  & \quad\forall x\in(\frac 1 4,\frac 3 4);\\
 f(x) = 0\,, & F(x) = - 4\frac {\sinh(1/4)} {\cosh(1/4)}\,, & \bu(x) = 0 & \quad\forall x\in(\frac 3 4,1).
  \end{array}
\end{equation}

\begin{figure}
\begin{tabular}{ccc}
\includegraphics[width=0.3\linewidth]{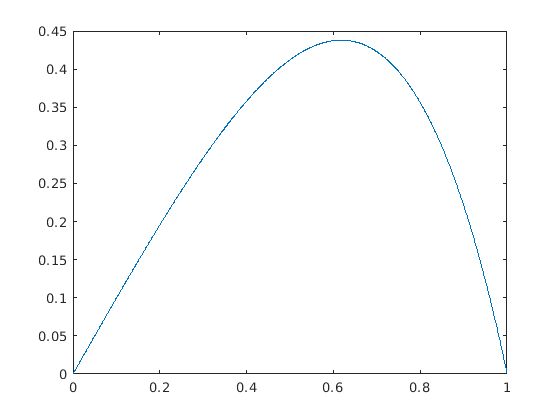} & \includegraphics[width=0.3\linewidth]{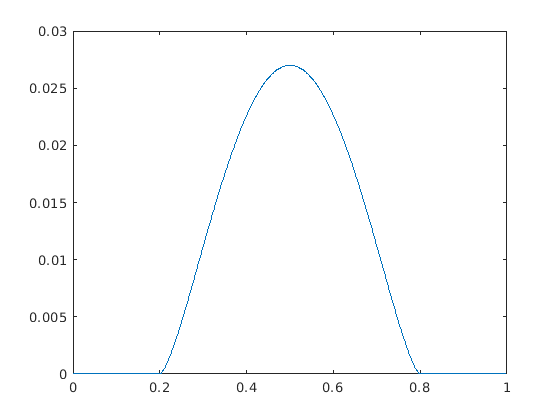} & \includegraphics[width=0.3\linewidth]{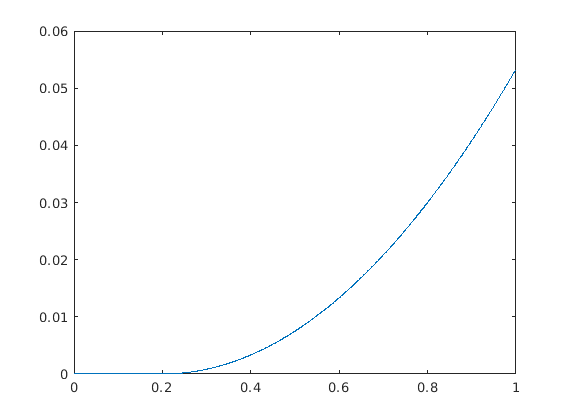}\\
Test case \refR{tc:casundcinfty} & Test case \refP{tc:casundporfnz} & Test case \refP{tc:casundporfz}\\
\includegraphics[width=0.3\linewidth]{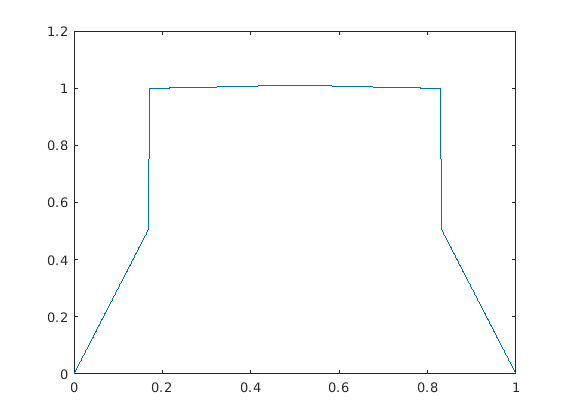} & \includegraphics[width=0.3\linewidth]{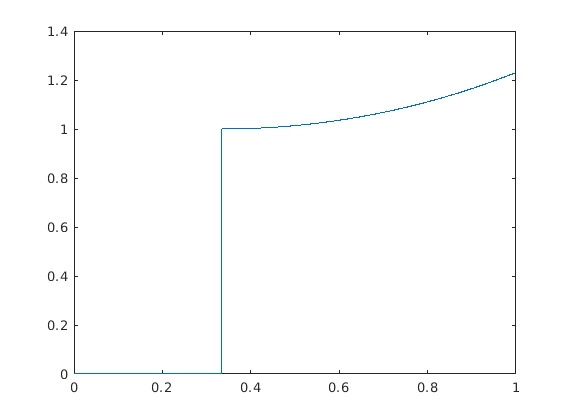} & \includegraphics[width=0.3\linewidth]{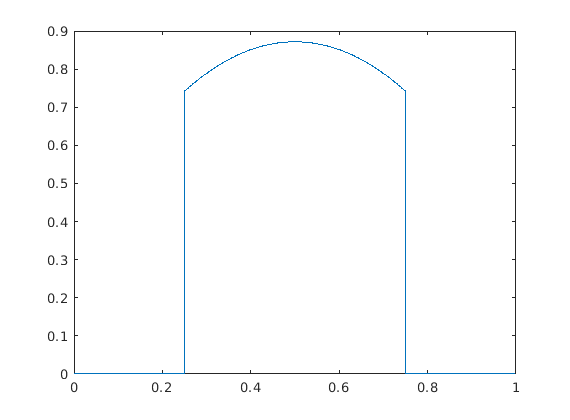}\\
Test case \refS{tc:casundstefnz} & Test case \refS{tc:casundstefz} & Test case \refS{tc:casundstefhmun}
\end{tabular}
\caption{Profiles of the various exact solutions for the 1D tests.}
\label{fig:profiles.1D}
\end{figure}

\subsection{Mass-lumped Finite Elements}\label{sec:EF}
  
For a  conforming simplicial mesh $\mesh$ and using the notations in Assumption \ref{assum:discs}, the Gradient Discretisations $\discs=\discs[k,\fe]$, for $k\in\{1,2,3\}$, corresponding to conforming $\Poly{k}$ Finite Elements are defined by the following elements.
\begin{itemize}[\hspace*{.4em}$\bullet$]
\item The points $(\vec{x}_i)_{i\in I}$ are:
\begin{itemize}[\hspace*{0pt}$\star$]
\item If $k=1$: the mesh vertices (in dimension 1 or 2),
\item If $k=2$: the mesh vertices and one point in each cell (in dimension 1), or the mesh vertices and the edge midpoints (in dimension 2),
\item If $k=3$: the mesh vertices and two points in each cell (in dimension 1), or the mesh vertices, two additional points on each edge, and one point in each cell (in dimension 2).
\end{itemize}
$I_{\partial\Omega}$ is the set of indices $i\in I$ such that $\vec{x}_i\in\partial\Omega$ and, for $K\in\mesh$, $I_K:=\{i\in I\,:\,\vec{x}_i\in\overline{K}\}$.
\item For each simplex $K\in\mesh$ and $v=(v_i)_{i\in I}\in\XDz$, $(\PiDs v)_{|K}$ is the unique polynomial in $\Poly{k}$ that takes the values $v_i$ at $\vec{x}_i$ for all $i\in I_K$.
\item The gradient reconstruction is given by $(\nablaD v)_{|K}=\nabla(\PiDs v)_{|K}$ for all $v\in\XDz$ and $K\in\mesh$.
\end{itemize}

\subsubsection{Numerical tests for mass-lumped Finite Elements in dimension 1}\label{sec:EF1D}

We consider two families of meshes of $\Omega=(0,1)$ with $N$ cells each, for $N\in\{16,32,64,512,1024,2048\}$. The first one is the uniform mesh $\umesh$ with mesh step $h = 1/N$. The second one is a random mesh $\rmesh$ such that each cell has size $h_i = H_i/\sum_j H_j$, where $H_i = (3 + \rho_i)$ with $\rho_i$ following a random uniform law on $(0,1)$. As mentioned above, all the Gradient Discretisations are mass-lumped versions of the corresponding $\Poly{k}$ GD. We describe in Table \ref{tab:DFE.1D} the remaining elements to fully define each GD, that is: the degree $k$ and the chosen quadrature rule, in each cell, using the nomenclature introduced in Table \ref{tab:quadratures.1D}. The nodes of each FE method are the union of the quadrature nodes in each cell; it is easily checked that the chosen quadrature rules always have the correct number of nodes to uniquely define the corresponding $\Poly{k}$ functions in each cell, and the global continuity of the FE space is ensured since all considered quadrature rules include the cell endpoints as nodes.

\begin{table}
\begin{tabular}{|c|c|c|c|}
\hline
Name of GD & Degree $k$  & Quadrature rule & $\ell$\\[.2em]
\hline
\hline
$\disc[\gen][1,\fe]$ & 1   & Trapezoidal & 0\\[.2em]
\hline
\hline
$\disc[\gen][2,\fe]$ & 2  &  Simpson & 1\\[.2em]
\hline
\hline
$\disc[\gen,a][3,\fe]$ & 3  &  Equi6 & --\\[.2em]
\hline
$\disc[\gen,b][3,\fe]$ & 3  &  Equi8 & 0\\[.2em]
\hline
$\disc[\gen,c][3,\fe]$ & 3  &  Gauss--Lobatto & 2\\[.2em]
\hline
\end{tabular}
\caption{Descriptions of the mass-lumped Finite Element GDs in dimension $d=1$. The degree $k$ is that of the local polynomial space, and $\ell$ is the degree in \eqref{assum:quad.qr} for the chosen quadrature rule (`--' means that \eqref{assum:quad.qr} is not even satisfied with $\ell=0$). The subscript $\gen$ can take the values $\unif$, when the considered mesh is uniform, or $\rand$, for random meshes.}
\label{tab:DFE.1D}
\end{table}

\medskip

{\sc Test case \refR{tc:casundcinfty}: Regular problem, $f\neq 0$}
 
The results provided in Table \ref{tab:casundcinfty} show a super-convergence, for $k=1,2$, of the function and gradient reconstruction when quadrature or interpolation are accounted for in the measure of the error: the errors $E_{\beta,\ID}^{\Pi}$, $E_{\zeta,\ID}^{\Pi}$ and $E_{\zeta,\ID}^{\nabla}$ appear to be at least $\mathcal O(h^{k+1})$ (even almost $\mathcal O(h^{k+2})$ for $k=2$ and the function reconstruction errors). The rate for $E_{\zeta}^{\nabla}$ falls to $h^k$ as expected since this is the optimal rate, when using piecewise $\Poly{k}$ polynomials, to approximate a smooth non-polynomial function.
These rates for the approximation of the gradient are actually above those predicted by our analysis: as seen in Table \ref{tab:DFE.1D}, for all these methods we can take $\ell=0$ in Theorem \ref{th:high.order} and thus, following Remark \ref{rem:rate.high}, the expected decay of $E_{\zeta,\ID}^{\nabla}$ is $h^{\min(2,k)}=h^k$. 

Regarding $k=3$, the schemes based on the $\disc[\gen,c][3,\fe]$ variant appear to have a worse rate for the errors $E_{\beta,\ID}^{\Pi}$ and $E_{\zeta,\ID}^{\Pi}$ than $\disc[\gen,a][3,\fe]$ or $\disc[\gen,b][3,\fe]$, but focusing on the constant $C$ we notice that these errors are actually much better. The scheme $\disc[\gen,c][3,\fe]$ also clearly outperforms the other two variants when considering the gradient reconstruction. Focusing on the latter, the convergence rate $\mathcal O(h)$ for $\disc[\gen,a][3,\fe]$ can be explained recalling that this variant does not even satisfy \eqref{assum:quad.qr}; even though $\term_{\disc}(\bu,u)=0$ for this method (see the proof of Theorem \ref{th:high.order}), we do not have any better estimate on $R_{\disc,\discs}(\bu,f)$ than in the proof of Corollary \ref{cor:rates}, which was precisely expected to be $\mathcal O(h)$ (see Remark \ref{rem:rate.low}).

On the contrary, referring again to Table \ref{tab:DFE.1D}, $\disc[\gen,b][3,\fe]$ enables us to take $\ell=0$ in Theorem \ref{th:high.order} and we recover the expected $\mathcal O(h^2)$ estimate on $E_{\zeta,\ID}^{\nabla}$ mentioned in Remark \ref{rem:rate.high}. For $\disc[\gen,c][3,\fe]$ we can even take $\ell=2$ and Table \ref{tab:casundcinfty} clearly shows that this leads to an improved and optimal $\mathcal O(h^3)$ estimate on the gradient (again, something predicted in Remark \ref{rem:rate.high}). 

These results clearly demonstrate that the key factor in choosing a proper mass-lumped version for a high-order scheme is the exactness property \eqref{assum:quad.qr} --- not satisfying this property leads to decreased rates of convergence. They also indicate, at least for $k=3$, the sharpness of the error estimate established in Theorem \ref{th:high.order}.

{
\setlength\extrarowheight{2pt}
\begin{table}
\begin{tabular}{|c||c|c||c|c||c|c||c|c||c|c|}
\hline
 &  \multicolumn{2}{c||}{$E_{\beta,\ID}^{\Pi}$} & \multicolumn{2}{c||}{$E_{\zeta,\ID}^{\Pi}$}& \multicolumn{2}{c||}{$E_{\zeta,\ID}^{\nabla}$}& \multicolumn{2}{c||}{$E_{\zeta}^{\nabla}$}\\
\hline
 GD &  $C$ & $\alpha$ &  $C$ & $\alpha$&  $C$ & $\alpha$&  $C$ & $\alpha$\\
\hline
\hline
$\disc[\unif][1,\fe]$ &  4.6e-01 & 2.00 & 4.6e-01 & 2.00 & 4.4e-01 & 2.00 & 1.3e+00 & 1.00 \\
\hline 
$\disc[\rand][1,\fe]$ &  2.5e-01 & 1.89 & 2.5e-01 & 1.89 & 3.1e-01 & 1.90 & 1.2e+00 & 0.99  \\
\hline 
\hline
$\disc[\unif][2,\fe]$ &  8.8e-02 & 3.83 & 8.8e-02 & 3.83 & 1.4e-01 & 3.00 & 4.4e-01 & 2.00  \\
\hline 
$\disc[\rand][2,\fe]$ &  7.4e-02 & 3.77 & 7.4e-02 & 3.77 & 1.3e-01 & 2.98 & 4.2e-01 & 1.98  \\
\hline 
\hline
$\disc[\unif,a][3,\fe]$ &  1.8e-01 & 2.00 & 1.8e-01 & 2.00 & 1.5e-01 & 1.00 & 1.5e-01 & 1.00  \\
\hline 
$\disc[\rand,a][3,\fe]$ &  2.0e-01 & 2.01 & 2.0e-01 & 2.01 & 1.5e-01 & 1.00 & 1.5e-01 & 1.00  \\
\hline 
$\disc[\unif,b][3,\fe]$ &  9.4e-02 & 3.00 & 9.4e-02 & 3.00 & 2.0e-01 & 2.00 & 2.0e-01 & 2.00  \\
\hline 
$\disc[\rand,b][3,\fe]$ &  9.6e-02 & 2.99 & 9.6e-02 & 2.99 & 2.0e-01 & 1.99 & 2.0e-01 & 1.99 \\
\hline 
$\disc[\unif,c][3,\fe]$ &  6.4e-08 & 1.73 & 6.4e-08 & 1.73 & 2.0e-04 & 2.95 & 7.2e-02 & 3.00  \\
\hline 
$\disc[\rand,c][3,\fe]$ &  9.0e-08 & 1.80 & 9.0e-08 & 1.80 & 2.4e-04 & 2.97 & 7.6e-02 & 3.00  \\
\hline 
\end{tabular}
\caption{Constants and rates for Test Case \refR{tc:casundcinfty}.}
\label{tab:casundcinfty}
\end{table}
}

\medskip

{\sc Test Case \refP{tc:casundporfnz}: Porous medium problem, homogeneous Dirichlet BC, $f\neq 0$}

The results are presented in Table \ref{tab:casundporfnz}.
The functions $f$ and $\bu$ are only piecewise smooth, and the discontinuity of their derivatives is not necessarily aligned with the mesh. As a consequence, Assumption \ref{assum:data.bu} does not hold. Compared to the smooth case studied in Test Case \refR{tc:casundcinfty}, the convergence is overall degraded. However, the rates mostly remain not far from the linear case (especially for gradient approximations), and the main features discussed for the smooth case can also be found here: some errors display super-convergence behaviours (when using quadrature or interpolation of the exact solution), and the rates of convergence drop drastically if the local quadrature rule \eqref{assum:quad.qr} does not hold with a high enough $\ell$.

{
\setlength\extrarowheight{2pt}
\begin{table}
\begin{tabular}{|c||c|c||c|c||c|c||c|c||c|c|}
\hline
 &  \multicolumn{2}{c||}{$E_{\beta,\ID}^{\Pi}$} & \multicolumn{2}{c||}{$E_{\zeta,\ID}^{\Pi}$}& \multicolumn{2}{c||}{$E_{\zeta,\ID}^{\nabla}$}& \multicolumn{2}{c||}{$E_{\zeta}^{\nabla}$}\\
\hline
 GD &  $C$ & $\alpha$ &  $C$ & $\alpha$&  $C$ & $\alpha$&  $C$ & $\alpha$\\
\hline
\hline
$\disc[\unif][1,\fe]$ &  2.3e+02 & 1.68 & 5.6e+00 & 2.01 & 1.2e+01 & 2.00 & 3.2e+00 & 1.00  \\ 
\hline 
$\disc[\rand][1,\fe]$ &  3.4e+02 & 1.71 & 5.3e+00 & 2.00 & 1.2e+01 & 1.98 & 3.4e+00 & 1.01  \\ 
\hline 
\hline
$\disc[\unif][2,\fe]$ &  1.9e+02 & 1.71 & 1.3e+00 & 2.69 & 4.3e+00 & 2.45 & 6.9e+00 & 2.01  \\ 
\hline 
$\disc[\rand][2,\fe]$ &  9.5e+01 & 1.50 & 1.3e+01 & 3.16 & 1.6e+01 & 2.69 & 6.4e+00 & 1.99  \\ 
\hline 
\hline
$\disc[\unif,a][3,\fe]$ &  8.0e+01 & 1.82 & 4.4e-01 & 2.01 & 4.1e-01 & 1.03 & 4.0e-01 & 1.02  \\ 
\hline 
$\disc[\rand,a][3,\fe]$ &  9.2e+01 & 1.74 & 5.2e-01 & 2.03 & 4.2e-01 & 1.03 & 4.2e-01 & 1.03  \\ 
\hline 
$\disc[\unif,b][3,\fe]$ &  8.6e+01 & 1.74 & 2.8e+00 & 2.90 & 2.7e+00 & 1.99 & 2.7e+00 & 1.99  \\ 
\hline 
$\disc[\rand,b][3,\fe]$ &  3.0e+01 & 1.46 & 3.9e+00 & 3.01 & 2.3e+00 & 1.95 & 2.4e+00 & 1.96 \\ 
\hline 
$\disc[\unif,c][3,\fe]$ &  1.7e+01 & 1.41 & 1.0e+00 & 2.92 & 1.2e+00 & 2.42 & 2.7e+00 & 2.41  \\ 
\hline 
$\disc[\rand,c][3,\fe]$ &  3.5e+01 & 1.52 & 7.0e-02 & 2.17 & 2.2e-01 & 1.98 & 1.7e+00 & 2.28  \\ 
\hline 
\end{tabular}
\caption{Constants and rates for Test Case \refP{tc:casundporfnz}.}
\label{tab:casundporfnz}
\end{table}
}

\medskip

{\sc Test Case \refP{tc:casundporfz}: Porous medium problem, non-homogeneous Dirichlet BC, $f=0$}

Table \ref{tab:casundporfz} details the outcomes of this test. The source term is obviously smooth, but the solution is only piecewise smooth. Despite this, the results show that, except for the very small constants previously observed for $\disc[\gen,c][3,\fe]$, the schemes behave here in a very similar way as for the completely smooth situation of Test Case \refR{tc:casundcinfty}. Here again we notice the importance of choosing proper local quadrature rules \eqref{assum:quad.qr} when designing mass-lumped schemes from high-order methods.

{
\setlength\extrarowheight{2pt}
\begin{table}
\begin{tabular}{|c||c|c||c|c||c|c||c|c||c|c|}
\hline
 &  \multicolumn{2}{c||}{$E_{\beta,\ID}^{\Pi}$} & \multicolumn{2}{c||}{$E_{\zeta,\ID}^{\Pi}$}& \multicolumn{2}{c||}{$E_{\zeta,\ID}^{\nabla}$}& \multicolumn{2}{c||}{$E_{\zeta}^{\nabla}$}\\
\hline
 GD &  $C$ & $\alpha$ &  $C$ & $\alpha$&  $C$ & $\alpha$&  $C$ & $\alpha$\\
 \hline 
\hline
$\disc[\unif][1,\fe]$ &  1.2e+01 & 1.99 & 2.2e-01 & 2.00 & 1.9e-01 & 2.00 & 1.3e+00 & 1.00   \\ 
\hline 
$\disc[\rand][1,\fe]$ &  1.5e+01 & 2.02 & 3.9e-01 & 2.09 & 3.8e-01 & 1.97 & 1.4e+00 & 1.01   \\ 
\hline 
\hline
$\disc[\unif][2,\fe]$ &  2.9e+00 & 2.50 & 2.1e-01 & 3.97 & 1.7e-01 & 2.99 & 5.3e-01 & 2.00   \\ 
\hline 
$\disc[\rand][2,\fe]$ &  2.0e+00 & 2.41 & 2.2e-01 & 3.94 & 1.8e-01 & 2.98 & 5.2e-01 & 1.99   \\ 
\hline
\hline 
$\disc[\unif,a][3,\fe]$ &  3.9e+00 & 2.00 & 2.3e-01 & 2.00 & 1.4e-01 & 1.00 & 1.4e-01 & 1.00   \\ 
\hline 
$\disc[\rand,a][3,\fe]$ &  4.1e+00 & 2.00 & 2.4e-01 & 2.00 & 1.5e-01 & 1.00 & 1.5e-01 & 1.00   \\ 
\hline 
$\disc[\unif,b][3,\fe]$ &  3.9e+00 & 2.50 & 1.9e-01 & 3.00 & 2.4e-01 & 2.00 & 2.4e-01 & 2.00   \\ 
\hline 
$\disc[\rand,b][3,\fe]$ &  3.5e+00 & 2.47 & 2.0e-01 & 2.99 & 2.4e-01 & 1.99 & 2.4e-01 & 1.99   \\ 
\hline 
$\disc[\unif,c][3,\fe]$ &  2.7e-01 & 2.40 & 5.2e-07 & 2.33 & 2.7e-04 & 3.10 & 9.9e-02 & 3.00  \\ 
\hline 
$\disc[\rand,c][3,\fe]$ &  1.4e+00 & 2.76 & 2.8e-06 & 2.64 & 1.4e-03 & 3.46 & 1.1e-01 & 3.00  \\ 
\hline 
\end{tabular}
\caption{Constants and rates for Test Case \refP{tc:casundporfz}.}
\label{tab:casundporfz}
\end{table}
}

\medskip

{\sc Test Case \refS{tc:casundstefnz}: Stefan problem, homogeneous Dirichlet BC, $f\neq 0$}

The results for this test case are presented in Table \ref{tab:casundstefnz}. This test case is a much more severe one
than the porous medium case, since the solution $\bu$ is discontinuous. This explains the poor convergence of
$E_{\beta,\ID}^{\Pi}$ for all considered methods. On the contrary, $\zeta(\bu)$ is continuous and $E_{\zeta,\ID}^{\Pi}$ thus
behaves much better, with an order $2$ decay for all schemes. The order of decay of $E_{\zeta,\ID}^{\nabla}$ is also similar for all methods (around $1.6$), except for the GDs $\disc[\gen,a][3,\fe]$ for which it drops to 1; this reduction can be explained, as in the previous case, by recalling that these GDs do not satisfy the local quadrature rules \eqref{assum:quad.qr} even for $\ell=0$.

Based on our previous discussion, we could expect the schemes corresponding to $\disc[\gen,c][3,\fe]$ to have a higher rate of convergence than the other methods, but it should be noted that $\zeta(\bu)$ belongs only to $H^2$, not $H^3$ since $(\zeta(\bu))''=\bu-f$ is discontinuous. This limits the application of Theorem \ref{th:high.order} to $s=2$ (despite $\ell=2$ being a valid choice in this case).

We notice that $E_{\zeta}^{\nabla}$ has a quite poor convergence (or does not seem to converge) on random meshes. Given that the difference between this error and $E_{\zeta,\ID}^{\nabla}$ solely lies in the interpolation error $\norm[L^2]{\nablaD\ID\zeta(\bu)-\nabla\zeta(\bu)}$, this apparently indicates that this interpolation error does not converge on random meshes. It is actually not the case, but for these meshes the regularity factor and maximum size oscillate a lot from one mesh to the other; combined with the low regularity of the solution (which implies an expected slow rate of convergence), this explains that the regression performed on the interpolation errors struggles to capture the correct convergence when considering a finite family of meshes.

{
\setlength\extrarowheight{2pt}
\begin{table}
\begin{tabular}{|c||c|c||c|c||c|c||c|c||c|c|}
\hline
 &  \multicolumn{2}{c||}{$E_{\beta,\ID}^{\Pi}$} & \multicolumn{2}{c||}{$E_{\zeta,\ID}^{\Pi}$}& \multicolumn{2}{c||}{$E_{\zeta,\ID}^{\nabla}$}& \multicolumn{2}{c||}{$E_{\zeta}^{\nabla}$}\\
\hline
 GD &  $C$ & $\alpha$ &  $C$ & $\alpha$&  $C$ & $\alpha$&  $C$ & $\alpha$\\
\hline
\hline
$\disc[\unif][1,\fe]$ &  1.8e+01 & 0.41 & 1.2e+01 & 1.97 & 1.2e+01 & 1.87 & 2.8e+00 & 1.00   \\ 
\hline 
$\disc[\rand][1,\fe]$ &  3.7e+01 & 0.54 & 1.6e+01 & 2.15 & 3.2e+00 & 1.66 & 2.6e-01 & -0.17   \\ 
\hline 
\hline
$\disc[\unif][2,\fe]$ &  6.0e+01 & 0.76 & 1.1e+00 & 2.04 & 6.2e-01 & 1.54 & 2.5e+00 & 1.61   \\ 
\hline 
$\disc[\rand][2,\fe]$ &  3.2e+01 & 0.50 & 1.2e+00 & 2.04 & 1.0e+00 & 1.65 & 3.3e-03 & -0.16  \\ 
\hline 
\hline
$\disc[\unif,a][3,\fe]$ &  7.9e+01 & 0.84 & 1.2e+00 & 2.03 & 3.7e-01 & 1.03 & 4.4e-01 & 1.06  \\ 
\hline 
$\disc[\rand,a][3,\fe]$ &  1.0e+02 & 0.83 & 1.2e+00 & 2.15 & 3.8e-01 & 1.05 & 3.6e-02 & 0.28  \\ 
\hline 
$\disc[\unif,b][3,\fe]$ &  8.6e+01 & 0.84 & 3.8e-01 & 1.95 & 7.2e-01 & 1.61 & 8.9e-01 & 1.53  \\ 
\hline 
$\disc[\rand,b][3,\fe]$ &  5.1e+01 & 0.64 & 3.4e-01 & 1.84 & 2.8e-01 & 1.53 & 1.4e+03 & 2.77   \\ 
\hline 
$\disc[\unif,c][3,\fe]$ &  5.4e+01 & 0.67 & 4.6e-01 & 2.08 & 3.6e-01 & 1.58 & 8.5e-01 & 1.56   \\ 
\hline 
$\disc[\rand,c][3,\fe]$ &  5.1e+01 & 0.61 & 2.9e+00 & 2.41 & 3.9e-01 & 1.59 & 4.7e-01 & 0.37  \\ 
\hline 
\end{tabular}
\caption{Constants and rates for Test Case \refS{tc:casundstefnz}.}
\label{tab:casundstefnz}
\end{table}
}

For this test case where the singularities of $(\zeta(\bu))''$ are located at specific points (namely, the discontinuities $\frac 1 2 \pm \gamma$ of $\bu$), it is interesting to consider the errors far from these points. Table \ref{tab:casundstefnzexcl} presents the regression data when the errors are computed excluding the two intervals of length $2/10$ around $\frac12 \pm \gamma$. We observe that $\disc[\gen,b][3,\fe]$ and $\disc[\gen,c][3,\fe]$ then yield a second order rate of convergence for $E_{\zeta,\ID}^{\nabla}$ (at least on uniform meshes), which is an improvement over the rate $h^{1.6}$ obtained when errors are computed over the entire domain (Table \ref{tab:casundstefnz}). We however do not recover a full $h^3$ rate of convergence for the methods $\disc[\gen,c][3,\fe]$, a sign that the localised lack of regularity of $\zeta(\bu)$ impacts the errors over the entire domain (which is expected for an elliptic equation with an infinite propagation speed). Noticeably, the variants $\disc[\gen,a][3,\fe]$ still only provide an order one rate of convergence, which is not surprising since the accuracy of these schemes is limited not by a lack of regularity of the solution, but by an improper choice of quadrature rules, which impacts the error on the entire domain.

{
\setlength\extrarowheight{2pt}
\begin{table}
\begin{tabular}{|c||c|c||c|c||c|c||c|c||c|c|}
\hline
 &  \multicolumn{2}{c||}{$E_{\beta,\ID}^{\Pi}$} & \multicolumn{2}{c||}{$E_{\zeta,\ID}^{\Pi}$}& \multicolumn{2}{c||}{$E_{\zeta,\ID}^{\nabla}$}& \multicolumn{2}{c||}{$E_{\zeta}^{\nabla}$}\\
\hline
 GD &  $C$ & $\alpha$ &  $C$ & $\alpha$&  $C$ & $\alpha$&  $C$ & $\alpha$\\
\hline
\hline
$\disc[\unif,a][3]$ &  1.2e+00 & 2.03 & 1.2e+00 & 2.03 & 3.0e-01 & 1.04 & 3.0e-01 & 1.04  \\ 
\hline 
$\disc[\rand,a][3]$ &  1.1e+00 & 2.21 & 1.1e+00 & 2.21 & 2.7e-01 & 1.03 & 1.9e-02 & 0.19  \\ 
\hline 
$\disc[\unif,b][3]$ &  3.3e-01 & 1.96 & 3.3e-01 & 1.96 & 1.5e+00 & 2.00 & 1.2e+00 & 1.94  \\ 
\hline 
$\disc[\rand,b][3]$ &  6.4e-01 & 2.08 & 6.4e-01 & 2.08 & 4.3e-01 & 1.75 & 7.1e-02 & 0.70  \\ 
\hline 
$\disc[\unif,c][3]$ &  3.9e-01 & 2.09 & 3.9e-01 & 2.09 & 1.1e-02 & 2.10 & 7.1e-05 & 0.49  \\ 
\hline 
$\disc[\rand,c][3]$ &  2.0e-01 & 1.88 & 2.0e-01 & 1.88 & 6.8e-01 & 2.12 & 3.2e-07 & -2.04  \\ 
\hline 
\end{tabular}
\caption{Constants and rates for Test Case \refS{tc:casundstefnz}, with errors computed excluding the discontinuities.}
\label{tab:casundstefnzexcl}
\end{table}
}

%

\medskip

{\sc Test Case \refS{tc:casundstefz}: Stefan problem, non-homogeneous Dirichlet BC, $f =  0$}

The results, presented in Table \ref{tab:casundstefz}, are comparable to those
obtained with a non-zero source term in Test Case \refS{tc:casundstefnz} (reduced convergence for $\disc[\gen,a][3,\fe]$, limitation of the convergence for $\disc[\gen,c][3,\fe]$ due to the limited regularity of $\zeta(\bu)$). In this case, however, the gradient $\nablaD\ID\zeta(\bu)$ of the interpolate seems to enjoy better convergence property even on random meshes, which preserve a reasonable convergence of $E_{\zeta}^{\nabla}$.

{
\setlength\extrarowheight{2pt}
\begin{table}
\begin{tabular}{|c||c|c||c|c||c|c||c|c||c|c|}
\hline
 &  \multicolumn{2}{c||}{$E_{\beta,\ID}^{\Pi}$} & \multicolumn{2}{c||}{$E_{\zeta,\ID}^{\Pi}$}& \multicolumn{2}{c||}{$E_{\zeta,\ID}^{\nabla}$}& \multicolumn{2}{c||}{$E_{\zeta}^{\nabla}$}\\
\hline
 GD &  $C$ & $\alpha$ &  $C$ & $\alpha$&  $C$ & $\alpha$&  $C$ & $\alpha$\\
\hline 
\hline
$\disc[\unif][1,\fe]$ &  2.0e+00 & 0.50 & 2.6e-01 & 1.98 & 1.5e-01 & 1.48 & 7.7e-01 & 1.00  \\ 
\hline 
$\disc[\rand][1,\fe]$ &  4.7e+00 & 0.67 & 5.2e-02 & 1.78 & 3.8e-02 & 1.32 & 7.6e-01 & 1.00  \\ 
\hline
\hline 
$\disc[\unif][2,\fe]$ &  2.3e+00 & 0.49 & 1.2e-01 & 2.02 & 8.6e-02 & 1.50 & 2.0e-01 & 1.50   \\ 
\hline 
$\disc[\rand][2,\fe]$ &  2.4e+00 & 0.53 & 1.6e-01 & 2.13 & 1.0e-01 & 1.61 & 2.1e-01 & 1.51   \\ 
\hline
\hline 
$\disc[\unif,a][3,\fe]$ &  3.4e+00 & 0.50 & 9.3e-02 & 2.00 & 8.9e-02 & 1.01 & 9.2e-02 & 1.01  \\ 
\hline 
$\disc[\rand,a][3,\fe]$ &  2.9e-02 & -0.21 & 6.8e-02 & 1.94 & 8.5e-02 & 1.00 & 8.8e-02 & 1.00   \\ 
\hline 
$\disc[\unif,b][3,\fe]$ &  4.1e+00 & 0.53 & 5.6e-02 & 2.03 & 8.0e-02 & 1.50 & 1.1e-01 & 1.50   \\ 
\hline 
$\disc[\rand,b][3,\fe]$ &  1.7e+01 & 1.10 & 1.5e-01 & 2.28 & 1.8e-01 & 1.75 & 9.3e-02 & 1.51  \\ 
\hline 
$\disc[\unif,c][3,\fe]$ &  3.1e+00 & 0.50 & 4.9e-02 & 2.01 & 5.3e-02 & 1.49 & 9.3e-02 & 1.50  \\ 
\hline 
$\disc[\rand,c][3,\fe]$ &  3.1e+00 & 0.71 & 6.0e-02 & 2.16 & 7.8e-02 & 1.78 & 6.1e-02 & 1.52  \\ 
\hline 
\end{tabular}
\caption{Constants and rates for Test Case \refS{tc:casundstefz}.}
\label{tab:casundstefz}
\end{table}
}

\medskip
 
{\sc Test Case \refS{tc:casundstefhmun}: Stefan problem, homogeneous Dirichlet conditions, $f\neq 0$ and $F\neq 0$}

The term $-\int_\Omega F\cdot \nablaD v$ in the Gradient Scheme \eqref{stefan:gs} is exactly computed, without numerical quadrature, using the relation 
\[
 -\int_\Omega F\cdot \nablaD v=-\int_0^1 F(s) (\PiDs v)'(s) ds = -4\frac {\sinh(1/4)} {\cosh(1/4)} \left(\PiDs v\left(\frac 1 4\right) + \PiDs v\left(\frac 3 4\right)\right).
\]
The outcome of the test can be seen in Table \ref{tab:casundstefhmun}. We note that these data do not satisfy the assumptions of Theorem \ref{th:high.order}, and no high-order rate can therefore be expected. Actually, the solution displays a very low regularity since $\zeta(\bu)$ only belongs to $H^1$, not even $H^2$. This is represented in the results by the fact that, for each given error and type of mesh (random/uniform), the rates of convergence for all degrees $k$ are in the same range. We notice also that, across the board, the schemes perform better on regular grids rather than random grids.

{
\setlength\extrarowheight{2pt}
\begin{table}
\begin{tabular}{|c||c|c||c|c||c|c||c|c||c|c|}
\hline
 &  \multicolumn{2}{c||}{$E_{\beta,\ID}^{\Pi}$} & \multicolumn{2}{c||}{$E_{\zeta,\ID}^{\Pi}$}& \multicolumn{2}{c||}{$E_{\zeta,\ID}^{\nabla}$}& \multicolumn{2}{c||}{$E_{\zeta}^{\nabla}$}\\
\hline
 GD &  $C$ & $\alpha$ &  $C$ & $\alpha$&  $C$ & $\alpha$&  $C$ & $\alpha$\\
\hline
\hline
$\disc[\unif][1,\fe]$ &  3.8e+01 & 0.50 & 3.5e+01 & 2.01 & 7.7e+00 & 1.49 & 1.2e+00 & 0.71  \\ 
\hline 
$\disc[\rand][1,\fe]$ &  2.3e+01 & 0.42 & 4.0e-01 & 0.81 & 5.8e-01 & 0.82 & 5.7e-01 & 0.34  \\ 
\hline
\hline 
$\disc[\unif][2,\fe]$ &  2.2e+01 & 0.50 & 3.6e+00 & 2.00 & 1.6e+00 & 1.50 & 3.7e-01 & 0.51  \\ 
\hline 
$\disc[\rand][2,\fe]$ &  1.2e+01 & 0.42 & 3.1e-03 & -0.29 & 7.2e-02 & 0.26 & 6.6e-01 & 0.44  \\ 
\hline 
\hline
$\disc[\unif,a][3,\fe]$ &  2.2e+01 & 0.50 & 3.3e+00 & 2.01 & 6.5e-01 & 1.18 & 3.6e-01 & 0.51  \\ 
\hline 
$\disc[\rand,a][3,\fe]$ &  2.9e+00 & 0.17 & 4.7e-03 & -0.03 & 6.3e-02 & 0.14 & 3.6e-01 & 0.35  \\ 
\hline 
$\disc[\unif,b][3,\fe]$ &  1.8e+01 & 0.50 & 2.3e+00 & 2.00 & 1.0e+00 & 1.50 & 3.6e-01 & 0.50  \\ 
\hline 
$\disc[\rand,b][3,\fe]$ &  5.4e+00 & 0.22 & 4.3e-01 & 0.95 & 2.2e-01 & 0.46 & 7.0e-01 & 0.50  \\ 
\hline 
$\disc[\unif,c][3,\fe]$ &  1.5e+01 & 0.50 & 8.8e-01 & 2.00 & 5.7e-01 & 1.50 & 3.5e-01 & 0.50  \\ 
\hline 
$\disc[\rand,c][3,\fe]$ &  5.5e+00 & 0.26 & 9.1e-03 & -0.05 & 8.3e-02 & 0.33 & 6.4e-01 & 0.48  \\ 
\hline 
\end{tabular}
\caption{Constants and rates for Test Case \refS{tc:casundstefhmun}.}
\label{tab:casundstefhmun}
\end{table}
}

\subsubsection{Numerical tests for mass-lumped Finite Elements in dimension 2}\label{sec:EF2D}

In the following 2D cases, we consider the domain $\Omega = (0,1)\times(0,1)$, the polynomial degrees $k=1,2$, and the following meshes (see Figure \ref{fig2Dmesh}):
\begin{enumerate}[\hspace*{.4em}$\bullet$]
 \item Triangular meshes which are as equilateral as possible, with edge length $1/N$ for $N\in\{25,50,100\}$. The Gradient Discretisations on these meshes will have the subscript ``\equil'', e.g., $\disc[\equil][k]$.
 \item Rectangular triangular meshes obtained by splitting $N^2$ squares in 2, for $N\in\{25,50,100\}$. We use the subscript ``\spl'' for these GDs, e.g., $\disc[\spl][k]$.
 \item Random meshes based on the three meshes \texttt{mesh1\_3}, \texttt{mesh1\_4} and \texttt{mesh1\_5} from the FVCA5 benchmark \cite{2Dbench}. The randomness is obtained moving the internal nodes by a uniform random factor, and we use the subscript ``\rand'' for these GDs, such as in $\disc[\rand][k]$.
\end{enumerate}

Based on these meshes, the mass-lumped versions of $\discs[k]$ ($k=1,2$) are described in Table \ref{tab:DFE.2D}. The quadrature rules refer to the rules described in Table \ref{tab:quadratures.2D}, and the nodes of the Finite Element method are the union of the quadrature nodes in all the cells. Note that $\disc[\gen,1/4][1,\fe]$ has degree $k=1$, but has the same unknowns as $\disc[\gen][2,\fe]$, corresponding to $k=2$, on the original mesh. 

\begin{table}
\begin{tabular}{|c|c|c|c|}
\hline
Name of GD & Degree $k$  & Quadrature rule & $\ell$\\[.2em]
\hline
\hline
$\disc[\gen][1,\fe]$ & 1  & Vertex & 0\\[.2em]
\hline
$\disc[\gen][2,\fe]$ & 2  &  Vertex+Edge Midpoint & 0\\[.2em]
\hline
\multicolumn{4}{|l|}{$\disc[\gen,1/4][1,\fe]=\disc[\gen][1,\fe]$ on the submesh described in Fig. \ref{fig:mlP2}.}\\[.2em]
\hline
\end{tabular}
\caption{Descriptions of the mass-lumped Finite Element GDs in dimension $d=2$. The degree $k$ is that of the local  polynomial space, and $\ell$ is the degree in \eqref{assum:quad.qr} for the chosen quadrature rule. 
Here, $\gen\in \{\equil,\spl,\rand\}$ depending if the GD uses the triangular equilateral meshes, the triangular meshes coming from splitting rectangles in two, or the random triangular meshes.}
\label{tab:DFE.2D}
\end{table}

%
%

\begin{figure}[ht!]
\resizebox{\textwidth}{!}{\includegraphics[trim={3cm 1cm 3cm 0},clip]{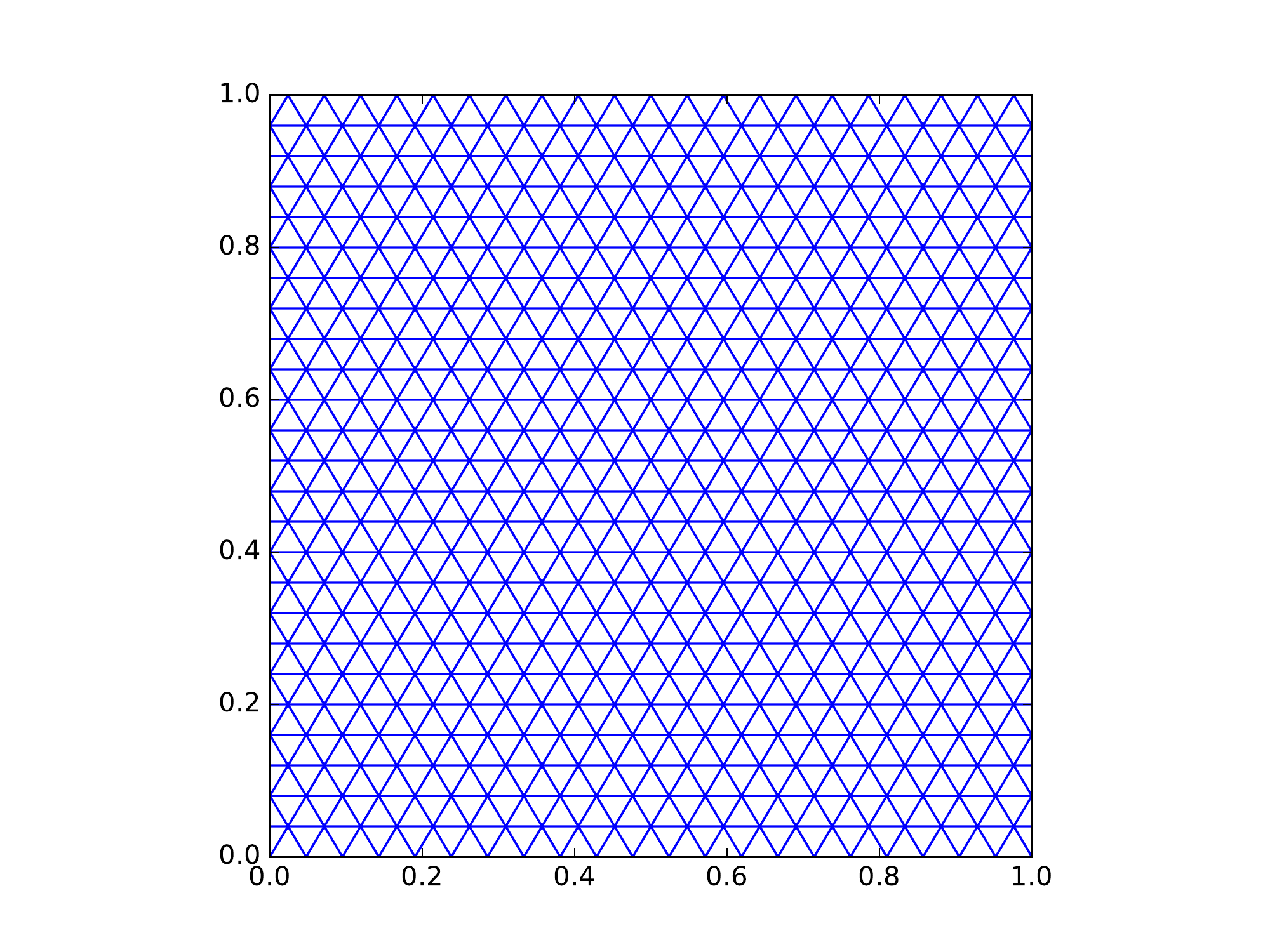}\includegraphics[trim={3cm 1cm 3cm 0},clip]{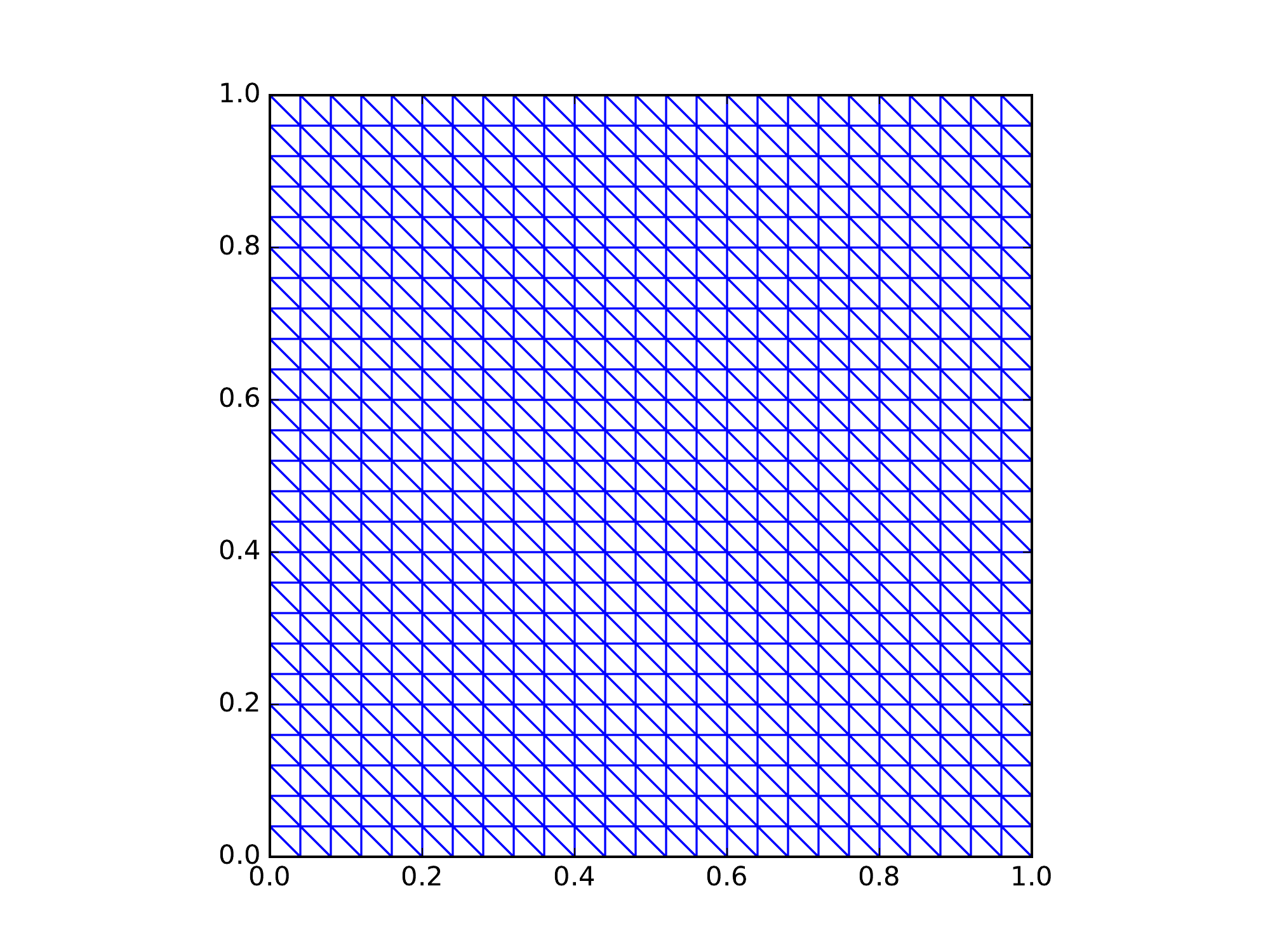}\includegraphics[trim={3cm 1cm 3cm 0},clip]{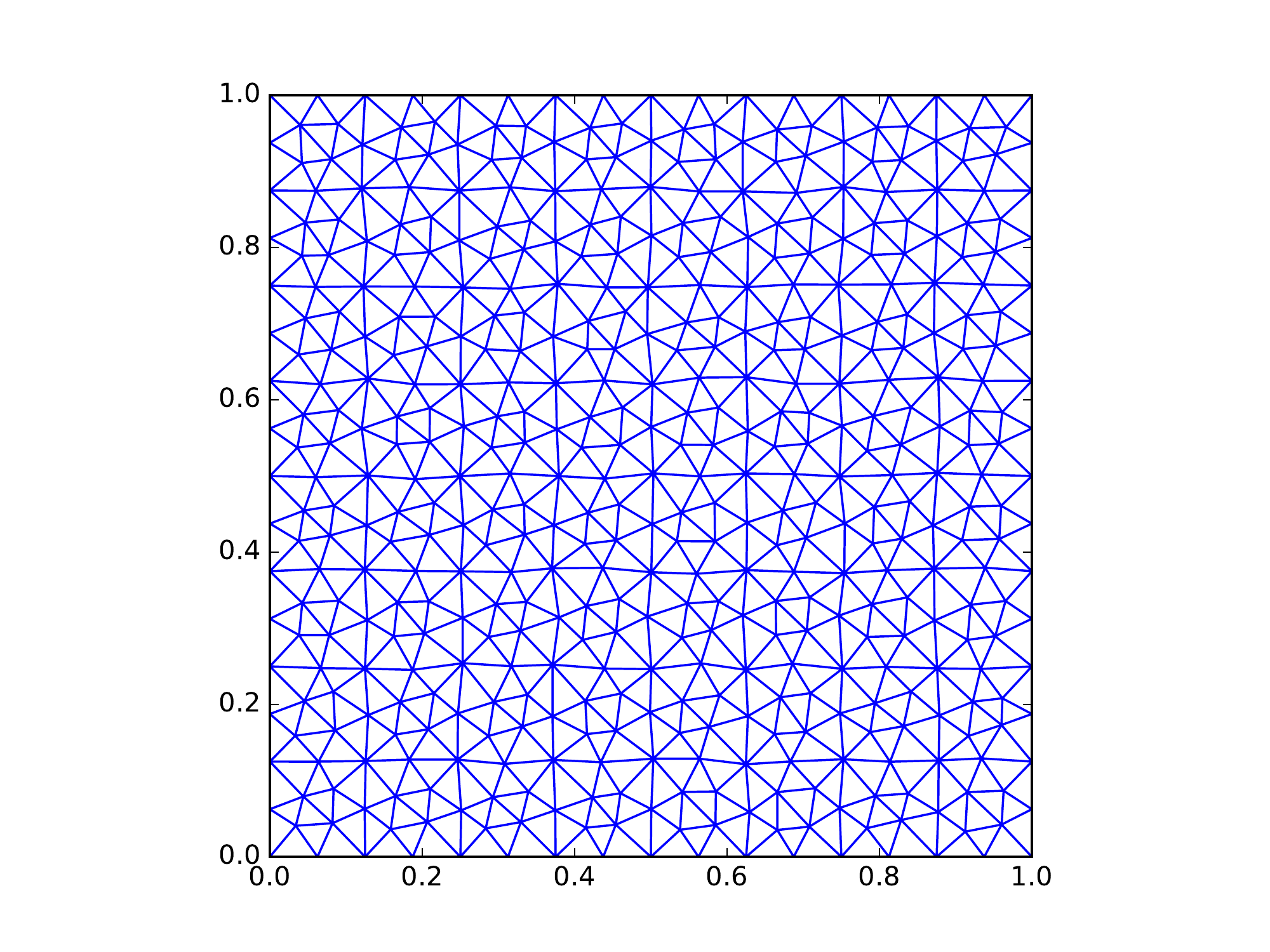}}
\caption{The three types of 2D meshes: ``\equil'' mesh (left), ``\spl'' mesh (middle), ``\rand'' mesh (right).}
\label{fig2Dmesh}\end{figure}

\begin{figure}
\begin{center}
\includegraphics{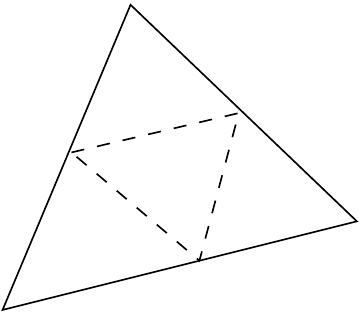}
\end{center}
\caption{
Division of a mesh triangle into four to construct $\disc[\gen,\di][1,\fe]$.}
\label{fig:mlP2}
\end{figure}

\begin{remark}[Implementation for $k=2$]\label{rem:P2ml}
If $k=2$, the function reconstruction obtained by mass-lumping does not see the vertex unknowns ($\pwpart_i=\emptyset$ if $\x_i$ is a mesh vertex). The corresponding mass matrix is therefore singular, which is of course an issue when considering explicit discretisations of time-dependent (even linear) problems; solving this issue requires the usage of enriched $\Poly{2}$ elements \cite{CJRT}. However, in the context of implicit time stepping, or equivalently of stationary problems, this is not an issue since the stiffness matrix is always non-singular.

The case of stationary nonlinear degenerate equations such as \eqref{stefan:strong} requires nonetheless an implementation trick. Since the diffusion term acts on $\zeta(u)$, in the nonlinear iterations the stiffness matrix is multiplied by $\zeta'(u)$ which can vanish, and the diffusion term does not yield in itself a control of all the unknowns $u_i$. It does however enable a control of the unknowns $\zeta(u)_i=\zeta(u_i)$. Even though these unknowns, especially for the Stefan problem, do not determine $u$ entirely, this gives a way to implement the scheme in a non-singular way. Instead of writing an equation on $(u_i)_{i\in I}$, we write an equation on $((u_i)_{i\in I_e},(\zeta(u)_i)_{i\in I_v})$, where $I_e$ is the set of indices corresponding to edge midpoints, and $I_v$ is the set of indices corresponding to the vertices. The unknowns $(u_i)_{i\in I_e}$ are controlled by the mass-lumped reaction term, and the unknowns $(\zeta(u)_i)_{i\in I_v}$ by the diffusion term. This implementation does not entirely determine a solution $u$ to the scheme, only its values at the edge midpoints and the values of $\zeta(u)$ at the vertices, but this is expected given Lemma \ref{lem:exuniq} and Remark \ref{rem:u.not.unique}.
\end{remark}

\begin{remark}[The case $k=3$]
 If $k=3$, it is possible to satisfy the local quadrature rules \eqref{assum:quad.qr} with $\ell=0$ (i.e. to have rules exact for third degree polynomials). This is done fixing $\alpha\in (0,(\frac {3} {44})^{1/2})$ and choosing the nodes $\x_i$ and weight proportions $|\pwpart_i\cap K|/|K|$ as follows:
\begin{enumerate}[\hspace*{.4em}$\bullet$]
\item the vertices of the mesh, each one associated with proportion $\frac {3-44\alpha^2} {60(1-4\alpha^2)}$;
\item two points on each edge located at the barycentric coordinates $(\frac 12 \pm \alpha, \frac 12 \mp \alpha)$ on the edge, associated with proportion $\frac {1} {15(1-4\alpha^2)}$;
\item the centers of mass of the triangles, associated with proportion $9/20$.
\end{enumerate}
To have local quadratures of degree four (that is, \eqref{assum:quad.qr} with $\ell=1$), one must set $\alpha^2 = 1/12$, which leads to the negative weight proportion $-1/60$ at the vertices of the triangle, a situation which is incompatible with the mass-lumping setting. To properly mass-lump the $\Poly{3}$ Finite Elements while preserving their high-order, an enriched version of these elements must be considered \cite{CJRT}.
\end{remark}

The data we consider in the following test case are the same as for the 1D case, using the diagonal as 1D coordinate. For example, if $g$ is a solution or source term for a 1D test case, the solution or source term for the corresponding 2D case is computed by setting $\widetilde{g}(x,y)=g((x+y)/\sqrt{2})$. All these 2D test cases therefore have non-homogeneous Dirichlet boundary conditions.

 \medskip

{\sc Tests with $\disc[\gen][k]$, $k=1,2$}

Tables \ref{tab:casddporfnz}--\ref{tab:casddstefz} present the results for the 2D versions of the Test Cases  \refP{tc:casundporfnz}, \refP{tc:casundporfz}, \refS{tc:casundstefnz} and \refS{tc:casundstefz}, that is: porous medium with $f\neq 0$, porous medium with $f=0$, Stefan problem with $f\neq 0$, and Stefan problem with $f=0$. Plots of solutions for the Stefan problems are given in Figure \ref{figure_t3_2Dm_u} (2D version of Test Case \refS{tc:casundstefnz}, $f\neq 0$) and Figure \ref{figure_t4_2Dm_u} (2D version of Test Case \refS{tc:casundstefz}, $f=0$).

All the considered Gradient Discretisations satisfy the local quadrature rules \eqref{assum:quad.qr} with $\ell=0$. Accordingly, if the solution and source were smooth, rates of convergence for $E_{\zeta,\ID}^{\nabla}$ should be $\mathcal O(h^{k})$ for $k=1,2$ (see Remark \ref{rem:rate.high}). The results show that, for the porous medium case, we are above these rates for all meshes, except for the random mesh --probably more representative of genuine situations-- where we are at these rates (or slightly above). As in the 1D case, the Stefan problem is more challenging and, probably due to the loss of regularity of the solution, the rates are a little bit worse. They do however remain at or above $\mathcal O(h)$ for $k=1$, and only drop to around $\mathcal O(h^{1.5})$ for $k=2$.

{
\setlength\extrarowheight{2pt}
\begin{table}
\resizebox{\textwidth}{!}{
\begin{tabular}{|c||c|c||c|c||c|c||c|c||c|c||c|c|}
\hline
GD &  \multicolumn{2}{c||}{$\disc[\equil][1,\fe]$} & \multicolumn{2}{c||}{$\disc[\equil][2,\fe]$}& \multicolumn{2}{c||}{$\disc[\spl][1,\fe]$} & \multicolumn{2}{c||}{$\disc[\spl][2,\fe]$}& \multicolumn{2}{c||}{$\disc[\rand][1,\fe]$}& \multicolumn{2}{c|}{$\disc[\rand][2,\fe]$}\\
\hline
 &  $C$ & $\alpha$ &  $C$ & $\alpha$ &  $C$ & $\alpha$ &  $C$ & $\alpha$&  $C$ & $\alpha$&  $C$ & $\alpha$\\
\hline
$E_{\beta,\ID}^{\Pi}$    &   9.1e-03 & 2.06  &  2.9e-02 & 2.95   &  5.0e-03 & 2.03   &   3.0e-03 & 2.59   &  2.0e-02 & 2.02   &   8.1e-03 & 2.50    \\
\hline
$E_{\zeta,\ID}^{\Pi}$    &  3.3e-04 & 2.07   &  2.8e-03 & 3.53   &  1.5e-04 & 2.04   &   2.0e-03 & 4.04  &  1.5e-03 & 2.03   &    1.2e-03 & 3.01   \\
\hline
$E_{\zeta,\ID}^{\nabla}$ &  1.4e-03 & 1.51   &   7.7e-03 & 2.52  &  7.5e-04 & 2.04   &   6.9e-03 & 3.02   &  2.0e-03 & 1.02   &    2.8e-03 & 2.02    \\
\hline
\end{tabular}
}
\caption{Constants and rates for the 2D version of Test Case \refP{tc:casundporfnz}.}
\label{tab:casddporfnz}
\end{table}

\begin{table}
\resizebox{\textwidth}{!}{
\begin{tabular}{|c||c|c||c|c||c|c||c|c||c|c||c|c|}
\hline
GD &  \multicolumn{2}{c||}{$\disc[\equil][1,\fe]$} & \multicolumn{2}{c||}{$\disc[\equil][2,\fe]$}& \multicolumn{2}{c||}{$\disc[\spl][1,\fe]$} & \multicolumn{2}{c||}{$\disc[\spl][2,\fe]$}& \multicolumn{2}{c||}{$\disc[\rand][1,\fe]$}& \multicolumn{2}{c|}{$\disc[\rand][2,\fe]$}\\
\hline
&  $C$ & $\alpha$ &  $C$ & $\alpha$ &  $C$ & $\alpha$ &  $C$ & $\alpha$&  $C$ & $\alpha$&  $C$ & $\alpha$\\
\hline
$E_{\beta,\ID}^{\Pi}$    &  1.4e-01 & 1.90  &  8.2e-02 & 1.69  &  4.8e-02 & 1.70   &  3.7e-03 & 1.02   &  3.1e-02 & 1.42  &   9.3e-02 & 1.63    \\
\hline
$E_{\zeta,\ID}^{\Pi}$    &  1.8e-03 & 2.06  &  4.0e-02 & 3.48  & 9.5e-04  & 2.05   &  1.8e-02 & 3.22  &  1.7e-03 & 2.02   &   1.8e-03 & 2.56   \\
\hline
$E_{\zeta,\ID}^{\nabla}$ &  2.5e-02 & 2.03  &  1.2e-01 & 2.52  &  1.3e-02 & 2.01  &   3.9e-02 & 2.38  &  2.6e-03 & 1.18  &     5.4e-02 & 2.29    \\
\hline
\end{tabular}
}
\caption{Constants and rates for the 2D version of Test Case \refP{tc:casundporfz}.}
\label{tab:casddporfz}
\end{table}

\begin{table}
\resizebox{\textwidth}{!}{
\begin{tabular}{|c||c|c||c|c||c|c||c|c||c|c||c|c|}
\hline
GD &  \multicolumn{2}{c||}{$\disc[\equil][1,\fe]$} & \multicolumn{2}{c||}{$\disc[\equil][2,\fe]$}& \multicolumn{2}{c||}{$\disc[\spl][1,\fe]$} & \multicolumn{2}{c||}{$\disc[\spl][2,\fe]$}& \multicolumn{2}{c||}{$\disc[\rand][1,\fe]$}& \multicolumn{2}{c|}{$\disc[\rand][2,\fe]$}\\
\hline
&  $C$ & $\alpha$ &  $C$ & $\alpha$ &  $C$ & $\alpha$ &  $C$ & $\alpha$&  $C$ & $\alpha$&  $C$ & $\alpha$\\
\hline
$E_{\beta,\ID}^{\Pi}$    &  1.3e-01 & 0.45  &  1.9e-01 & 0.52   &  1.7e-01 & 0.57   &  7.6e-02 & 0.32  &  8.3e-02 & 0.38   &    6.0e-02 & 0.29    \\
\hline
$E_{\zeta,\ID}^{\Pi}$    &  1.7e-02 & 2.04  &  4.5e-02 & 2.64  &  4.1e-03 & 1.88   &  4.6e-03 & 1.95 &  1.2e-02 & 1.96   &   2.7e-02 & 2.22   \\
\hline
$E_{\zeta,\ID}^{\nabla}$ &  6.9e-02 & 1.66  &  6.0e-02 & 1.53  &  1.1e-02 & 1.33   &  4.8e-02 & 1.50  &   1.5e-02 & 1.07 &    1.0e-01 & 1.64   \\
\hline
\end{tabular}
}
\caption{Constants and rates for the 2D version of Test Case \refS{tc:casundstefnz}.}
\label{tab:casddstefnz}
\end{table}
}

\begin{figure}[ht!]
\resizebox{\textwidth}{!}{\includegraphics[trim={4cm 1cm 4cm 2cm},clip]{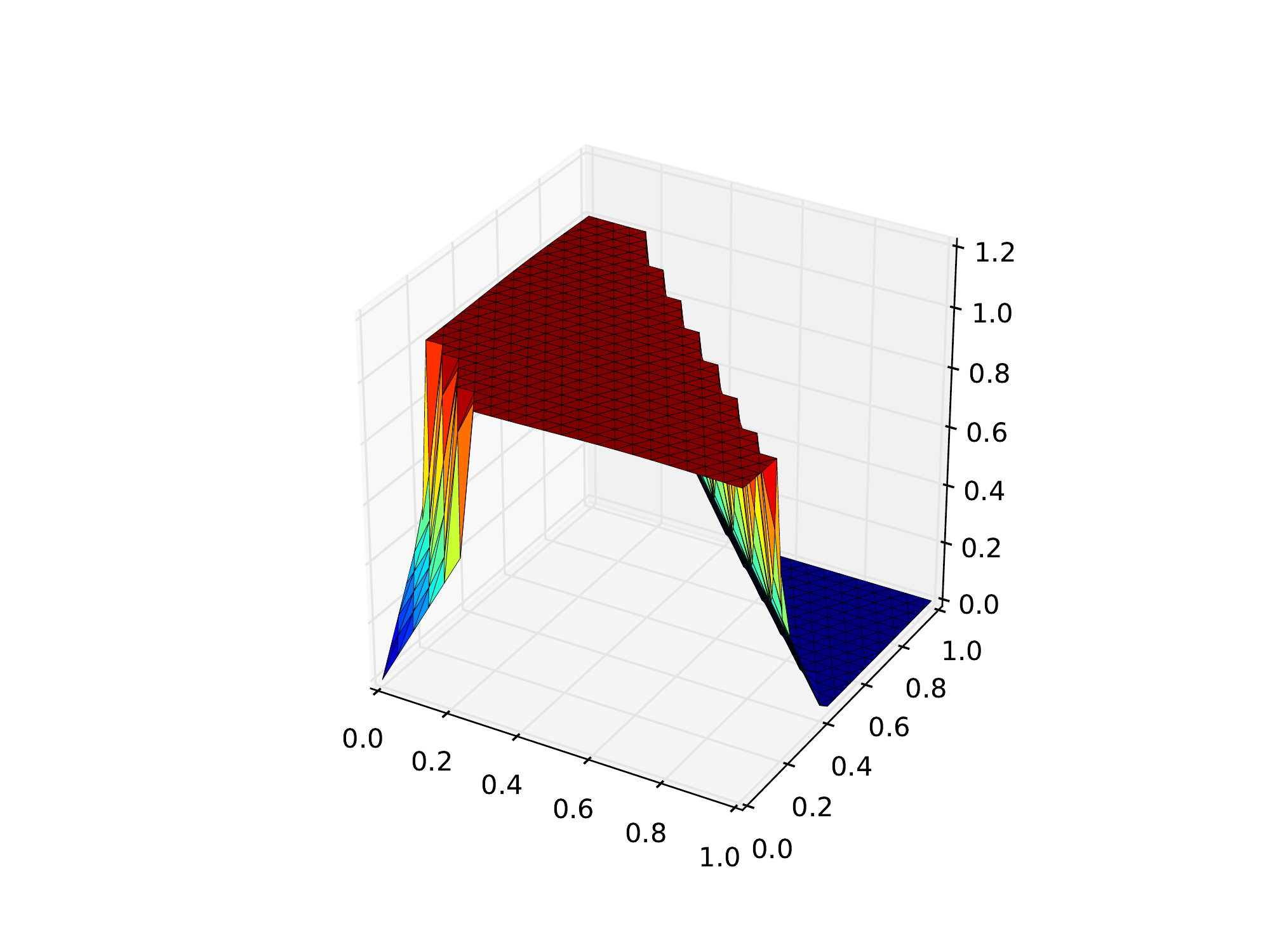}\includegraphics[trim={4cm 1cm 4cm 2cm},clip]{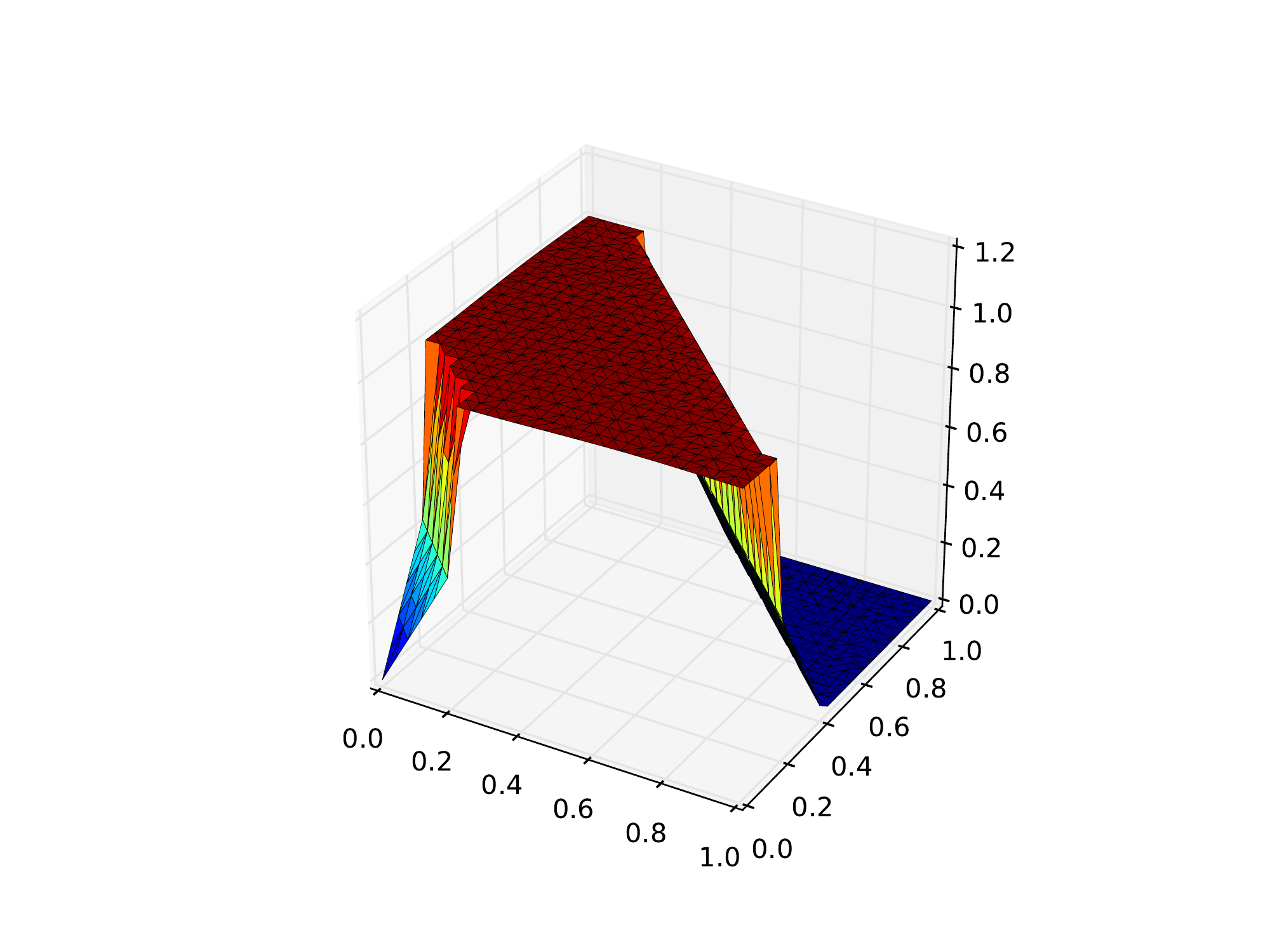}\includegraphics[trim={4cm 1cm 4cm 2cm},clip]{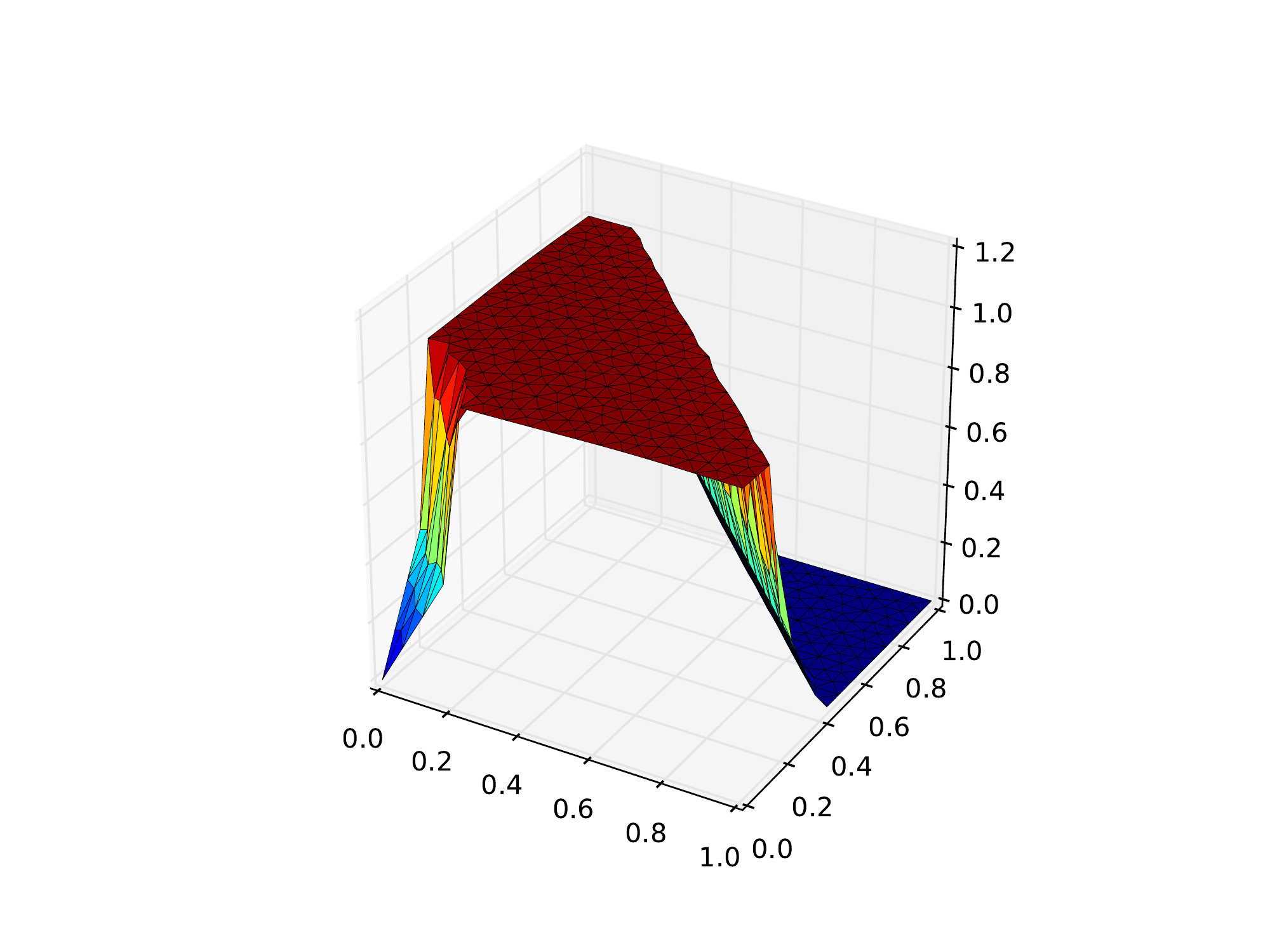}}
\resizebox{\textwidth}{!}{\includegraphics[trim={4cm 1cm 4cm 2cm},clip]{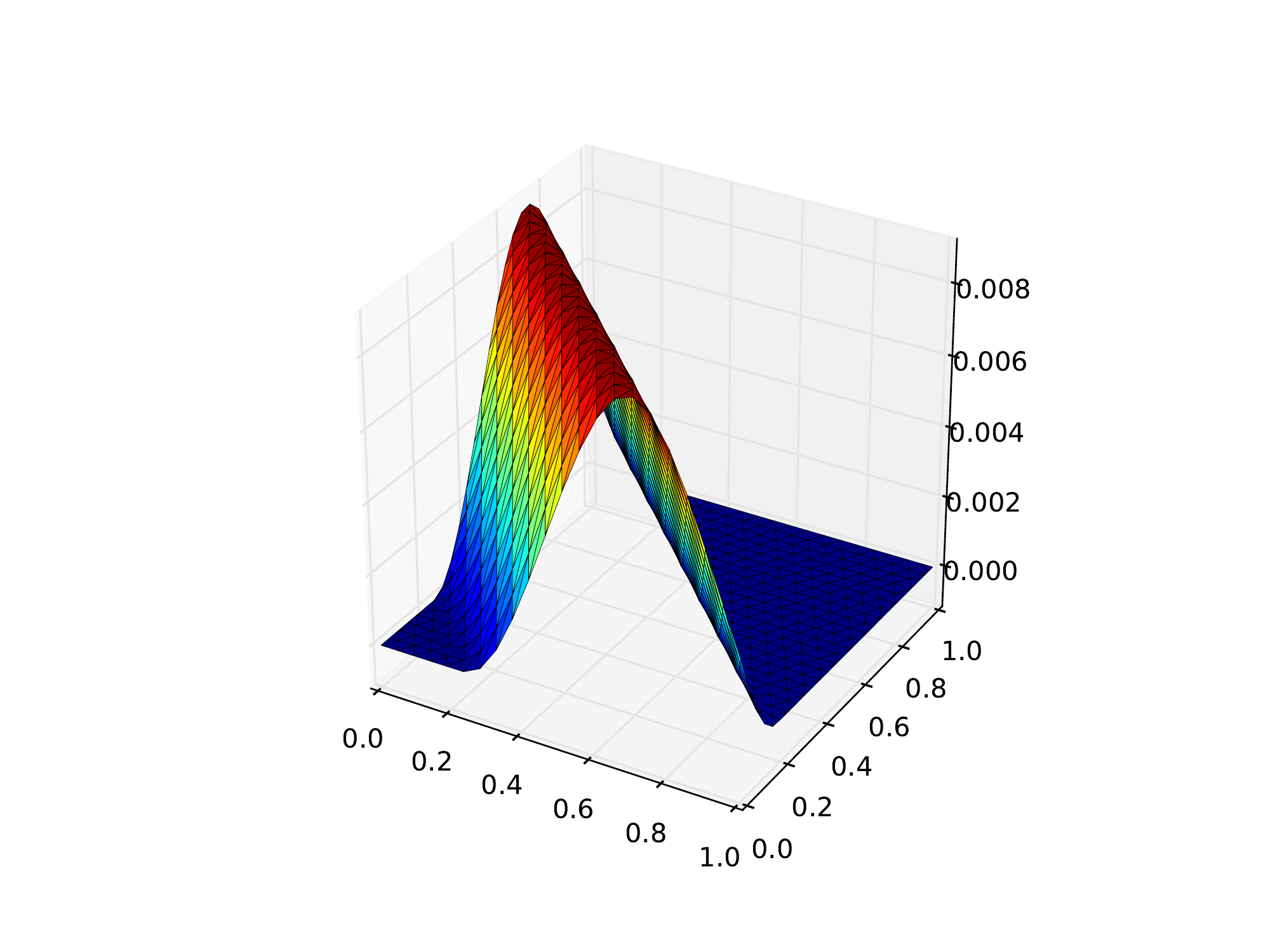}\includegraphics[trim={4cm 1cm 4cm 2cm},clip]{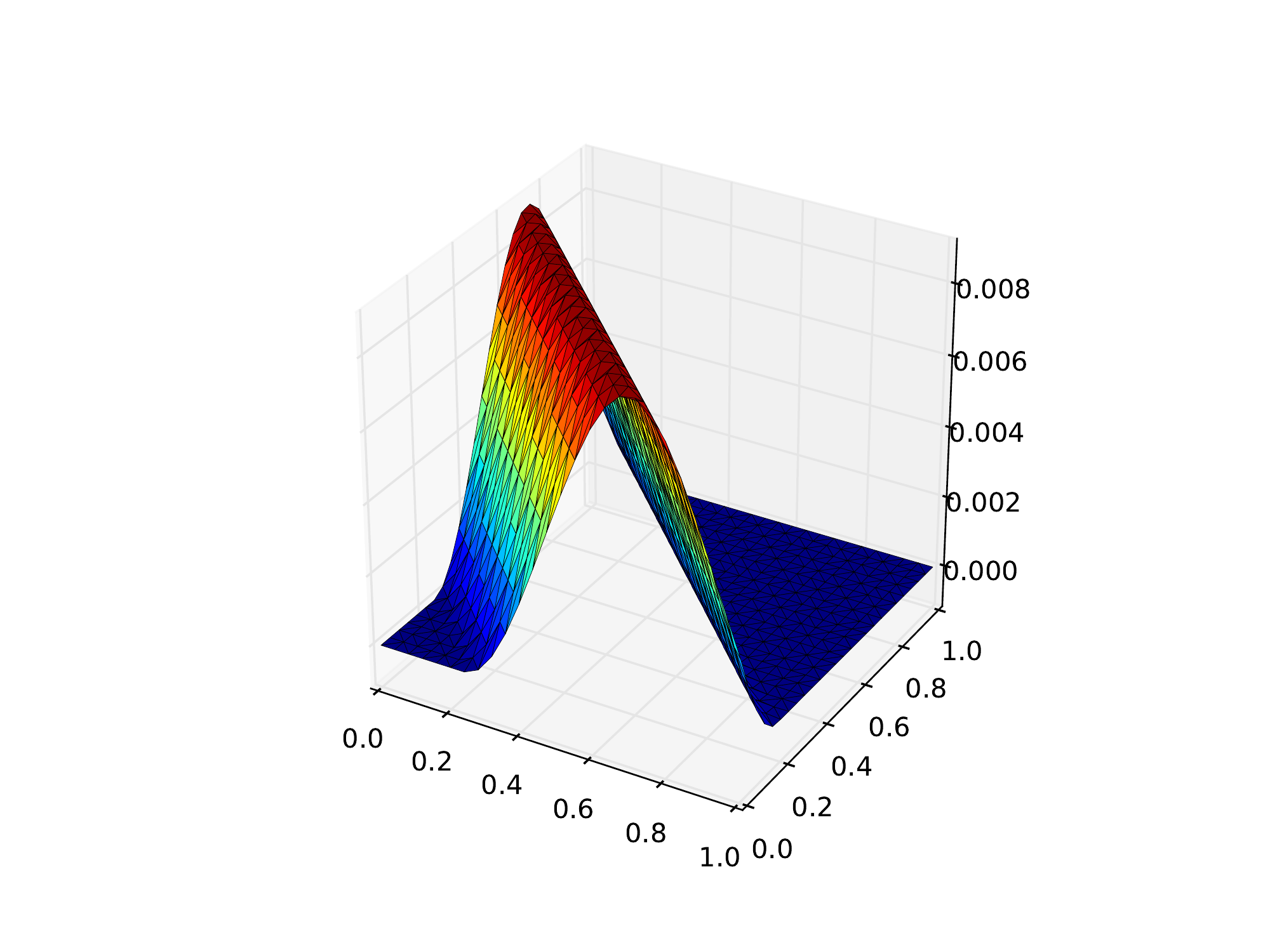}\includegraphics[trim={4cm 1cm 4cm 2cm},clip]{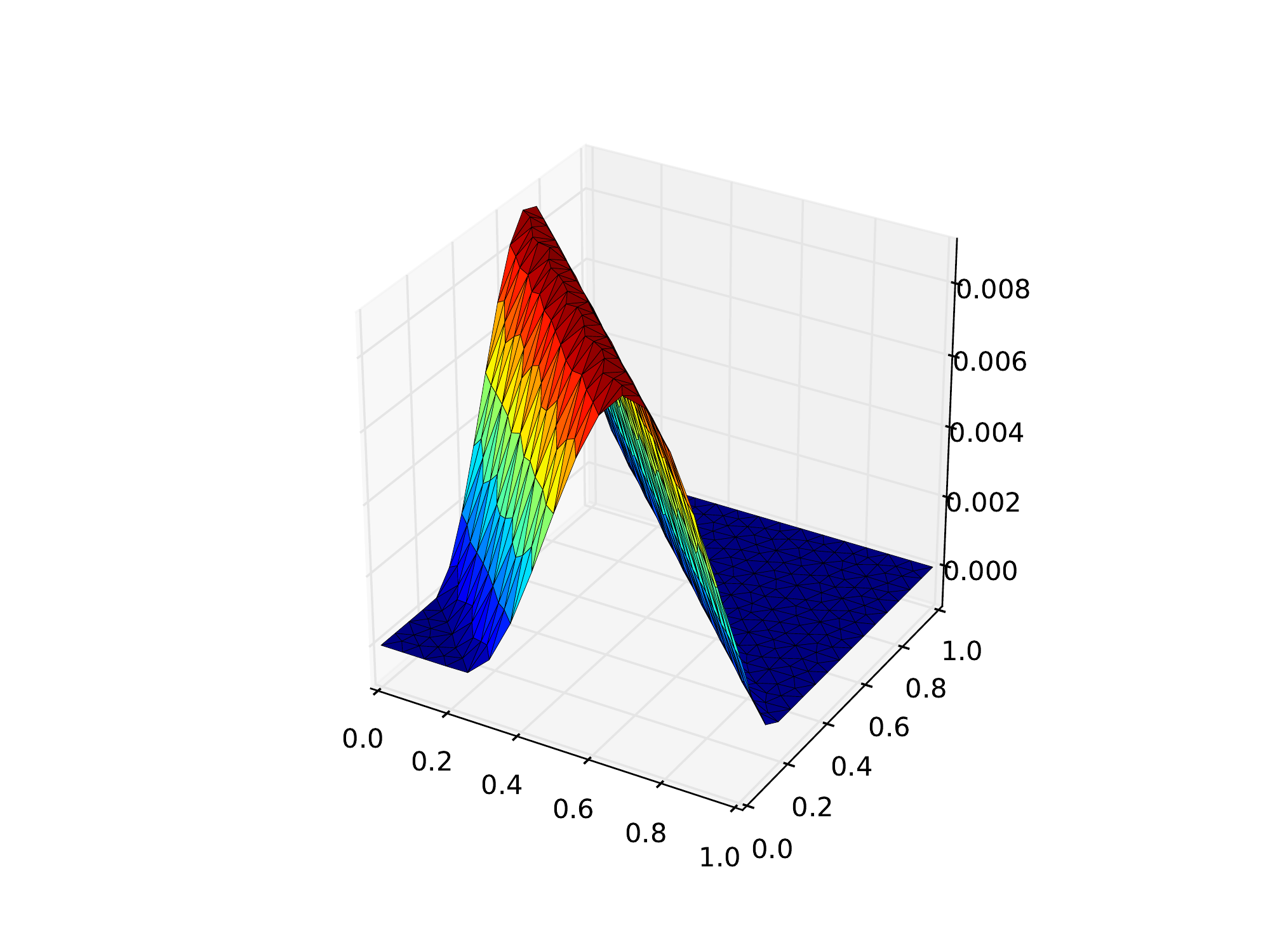}}
\caption{Approximate functions (top: $u$, bottom: $\zeta(u)$) for  the 2D version of Test Case \refS{tc:casundstefnz}. From left to right: $\disc[\equil][1,\fe]$, $\disc[\spl][1,\fe]$, $\disc[\rand][1,\fe]$.}
\label{figure_t3_2Dm_u}\end{figure}

{
\setlength\extrarowheight{2pt}
\begin{table}
\resizebox{\textwidth}{!}{
\begin{tabular}{|c||c|c||c|c||c|c||c|c||c|c||c|c|}
\hline
GD &  \multicolumn{2}{c||}{$\disc[\equil][1,\fe]$} & \multicolumn{2}{c||}{$\disc[\equil][2,\fe]$}& \multicolumn{2}{c||}{$\disc[\spl][1,\fe]$} & \multicolumn{2}{c||}{$\disc[\spl][2,\fe]$}& \multicolumn{2}{c||}{$\disc[\rand][1,\fe]$}& \multicolumn{2}{c|}{$\disc[\rand][2,\fe]$}\\
\hline
&  $C$ & $\alpha$ &  $C$ & $\alpha$ &  $C$ & $\alpha$ &  $C$ & $\alpha$&  $C$ & $\alpha$&  $C$ & $\alpha$\\
\hline
$E_{\beta,\ID}^{\Pi}$    & 2.5e-01 & 0.48   &  4.0e-01 & 0.58  &  8.1e-02 & 0.35   &  5.0e-01 & 0.68  &  1.4e-01 & 0.37   &    7.7e-01 & 0.72    \\
\hline
$E_{\zeta,\ID}^{\Pi}$    &  2.4e-02 & 2.10  &  1.6e-02 & 2.10  &  2.4e-02 & 2.24   &   3.2e-02 & 2.23  &  3.8e-02 & 2.06   &   1.5e-02 & 1.86    \\
\hline
$E_{\zeta,\ID}^{\nabla}$ &  7.5e-02 & 1.57  &  9.5e-02 & 1.51  &  7.2e-02 & 1.71   &  9.5e-02 & 1.52  &  5.0e-02 & 1.02  &    9.8e-02 & 1.49     \\
\hline
\end{tabular}
}
\caption{Constants and rates for the 2D version of Test Case \refS{tc:casundstefz}.}
\label{tab:casddstefz}
\end{table}
}

\begin{figure}[ht!]
\resizebox{\textwidth}{!}{\includegraphics[trim={4cm 1cm 4cm 2cm},clip]{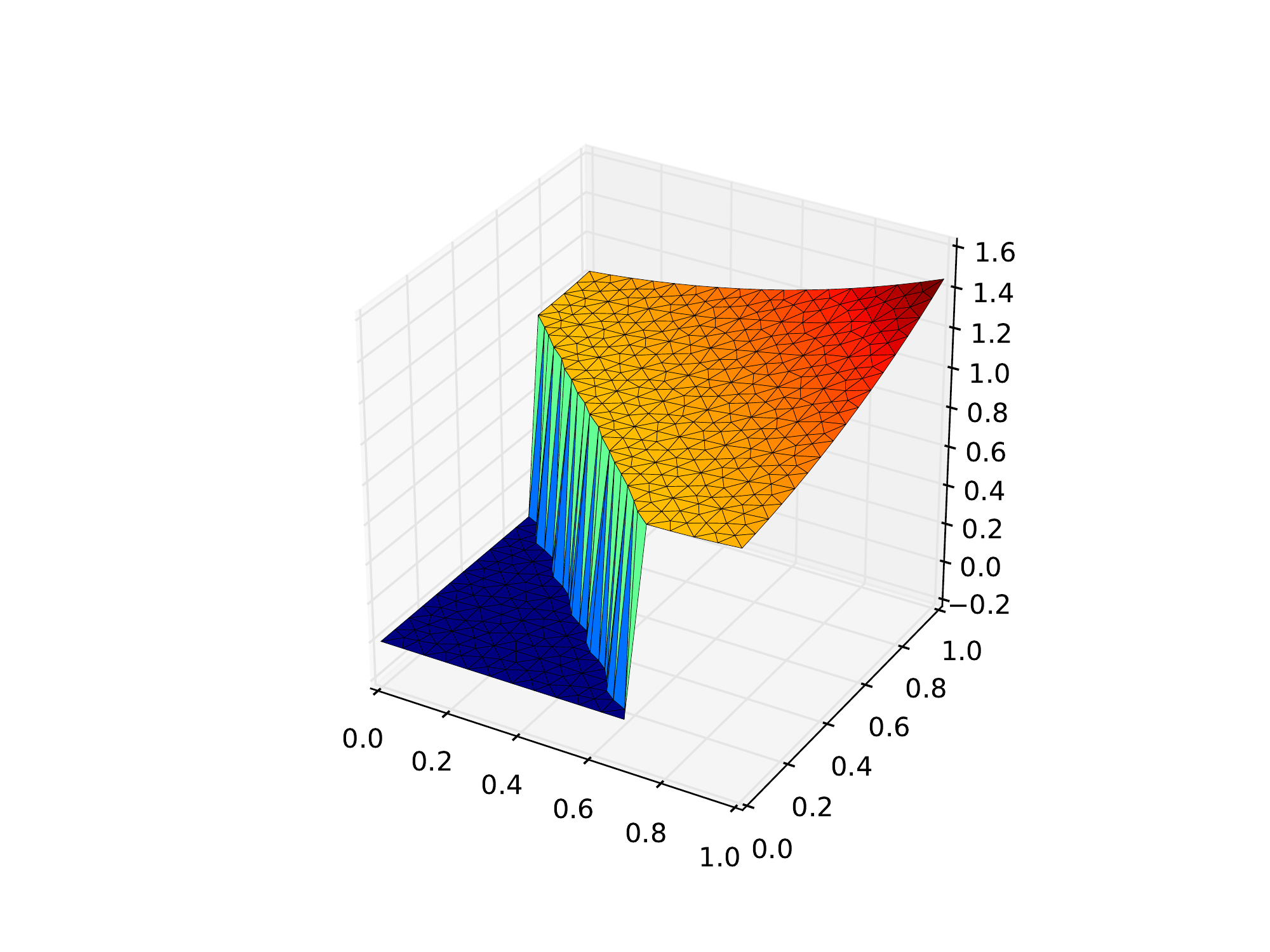}\includegraphics[trim={4cm 1cm 4cm 2cm},clip]{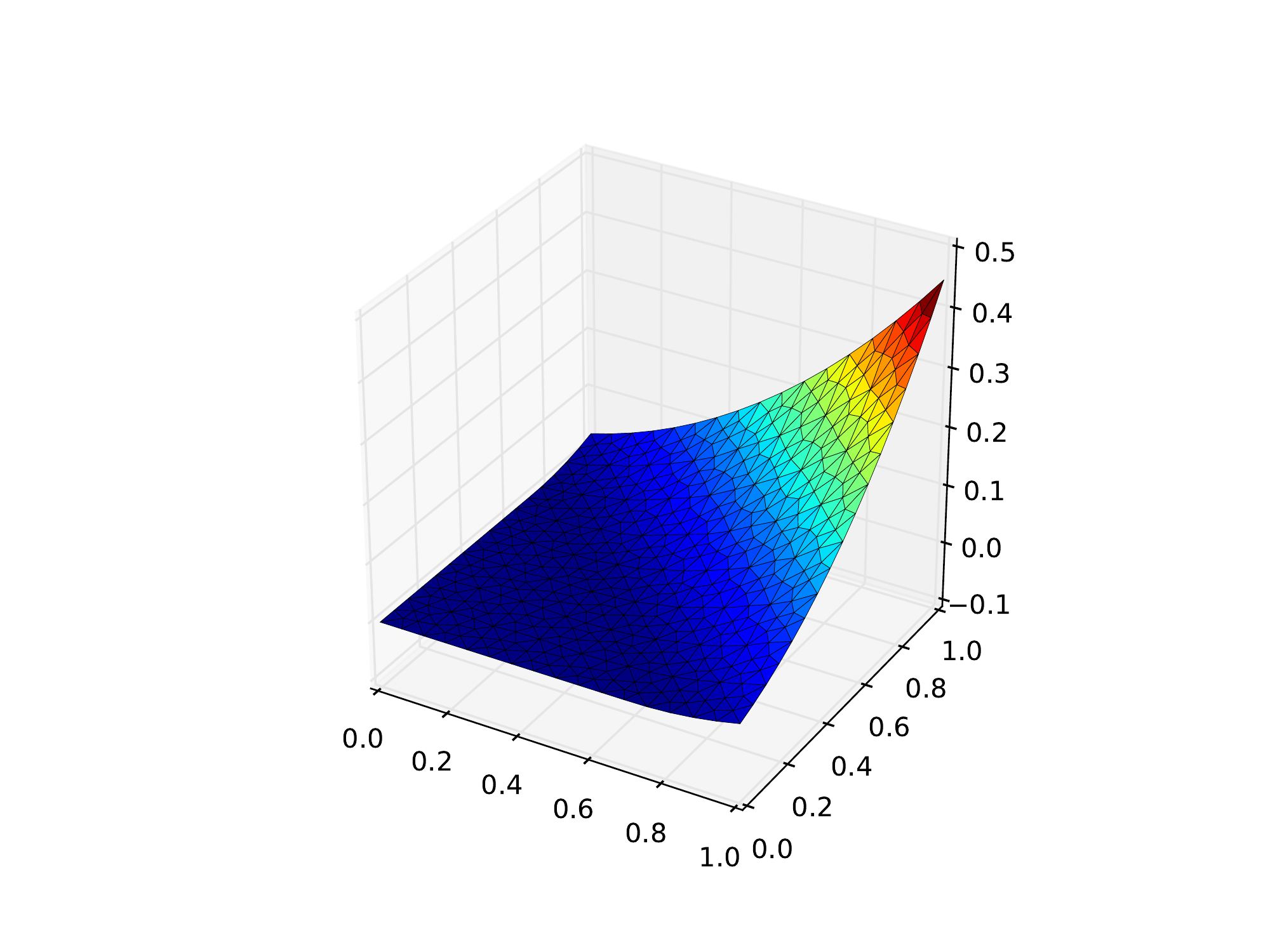}}
\caption{Approximate functions $u$ (left) and $\zeta(u)$ (right) for  the 2D version of Test Case \refS{tc:casundstefz}, using $\disc[\rand][1,\fe]$.}
\label{figure_t4_2Dm_u}\end{figure}

\medskip

{\sc Test with $\disc[\gen,1/4][1,\fe]$ and $\disc[\gen][2,\fe]$: comparison between degree 1 and degree 2}

To properly assess the interest of using a 2nd order scheme over a 1st order scheme, we now look, on the same triangular mesh, at the outputs of $\disc[\rand][2,\fe]$ and $\disc[\rand,\di][1,\fe]$. This makes for a fair comparison since these two schemes have the same number of unknowns. For each errors $E=E_{\beta,\ID}^{\Pi}$, $E=E_{\zeta,\ID}^{\Pi}$ and $E=E_{\zeta,\ID}^{\nabla}$, letting $E_k$ be the error corresponding to $\disc[\rand,\di][1,\fe]$ if $k=1$, or to $\disc[\rand][2,\fe]$ if $k=2$, we compute the ratios $r=E_2/E_1$ for all the tests on the three random meshes based on \texttt{mesh1\_3}, \texttt{mesh1\_4} and \texttt{mesh1\_5}. Assuming that each error $E_k$ is of the form $C_k (\frac {h} {h_0})^{\alpha_k}$, where $h_0$ is the size of the reference mesh \texttt{mesh1\_3}, the ratio between the two errors should be given by 
\[
 r=\frac {E_2} {E_1} = \frac {C_2} {C_1} \left(\frac {h} {h_0}\right)^{\alpha_2-\alpha_1}.
\]
Table \ref{tab:casddstefzcompar} performs a $C(h/h_0)^\alpha$ regression of the ratio $r$. Hence, the $C$ values in this table can be considered as approximations of $ \frac {C_2} {C_1}$, and the $\alpha$ values as approximations of $\alpha_2-\alpha_1$. The results show a clear advantage (smaller $C_k$, larger $\alpha_k$) of the second order method over the first order method, and also that this advantage still holds, albeit reduced, for irregular (Stefan) test cases.

{
\setlength\extrarowheight{2pt}
\begin{table}
\begin{tabular}{|c||c|c||c|c||c|c||c|c|}
\hline
Case &  \multicolumn{2}{c||}{Test Case \refP{tc:casundporfnz}} & \multicolumn{2}{c||}{Test Case \refP{tc:casundporfz}}& \multicolumn{2}{c||}{Test Case \refS{tc:casundstefnz}} & \multicolumn{2}{c|}{Test Case \refS{tc:casundstefz}}\\
\hline
&  $C$ & $\alpha$ &  $C$ & $\alpha$ &  $C$ & $\alpha$ &  $C$ & $\alpha$\\
\hline
$E_{\beta,\ID}^{\Pi}$ 2/1           &   5.8e-02 & 1.63     &   1.3e+00 & 0.99    &  1.2e+00 & 0.37     &  1.5e+00 & 0.53  \\
\hline
$E_{\zeta,\ID}^{\Pi}$ 2/1           &   1.6e-02 & 2.00     &   1.4e-01 & 1.75   &    6.9e-01 & 1.27   &   7.2e-01 & 0.97 \\
\hline
$E_{\zeta,\ID}^{\nabla}$ 2/1        &    4.4e-02 & 1.83    &   4.2e-01 & 2.00    &    8.5e-01 & 1.26   &    4.6e-01 & 1.34   \\
\hline
\end{tabular}
\caption{Constants and rates for the comparison first/second order with the same number of degrees of freedom.}
\label{tab:casddstefzcompar}
\end{table}
}

\subsection{Numerical tests for mass-lumped DG schemes}\label{sec:DG1D}

The mesh $\mesh$ being a general polytopal mesh as in \cite[Definition 7.2]{gdm}, still using the notations in Assumption \ref{assum:discs} the Gradient Discretisation $\discs=\discs[k,\dg]$ for the SIPG method of order $k$ is defined as follows:
\begin{itemize}[\hspace*{.4em}$\bullet$]
\item For each cell $K\in\mesh$, points $(\x_i)_{i\in I_K}$ are chosen such that for each choice of real numbers $(w_i)_{i\in I_K}$ there is a unique $q\in\Poly{k}$ such that $q(\x_i)=w_i$ for all $i\in I_K$. Then $I = (\cup_{K\in\mesh} I_K)\cup I_{\partial\Omega}$ is the family that gathers the indices of all these points for all the cells, and of all the boundary points where a jump is accounted for in the expression of $\nablaD$.
\item For each $K\in\mesh$ and $v=(v_i)_{i\in I}\in\XDz$, $(\PiDs v)_{|K}$ is the unique polynomial in $\Poly{k}$ that takes the values $v_i$ at $\vec{x}_i$ for all $i\in I_K$.
\item The gradient reconstruction is given by $(\nablaD v)_{|K}=\nabla(\PiDs v)_{|K}+S_K(v)$ for all $v\in\XDz$ and $K\in\mesh$, where $S_K(v)$ is an appropriate stabilisation term accounting for the jumps appearing in the DG scheme (see \cite[Definition 11.1]{gdm} for details). 
\end{itemize}

\begin{remark}[Embedding the SIPG method into the GDM] 
The SIPG stabilisation term is accounted for in the design of the gradient reconstruction $\nablaD$ through a penalisation parameter (denoted by $\beta$ in \cite[Chapter 11]{gdm}) which is fixed at $0.6$ in all the tests.
\end{remark}

We take $k=3$ and consider the same families of uniform and random meshes of $\Omega=(0,1)$ with $N$ cells each as in Section \ref{sec:EF1D}. Table \ref{tab:DG.1D} provides the remaining elements to fully define the GDs. These elements follow closely the choices made for the 1D Finite Element meshes in Section \ref{sec:EF1D}, but there is a major difference in the choice of the global nodes. Since DG methods do not have any continuity conditions at the mesh vertices, each of these vertices corresponds to two different nodes (one for each cell the vertex belongs to, or one additional node to encode the boundary conditions for the domain endpoints), with different associated values of the unknowns/test functions. The nodes are therefore $\x_0 = \x_1 = 0 < \x_2 < \x_3 < \x_4 = \x_5 <\ldots \x_{4i} = \x_{4i+1} < \x_{4i+2}< \x_{4i+3}<\ldots \x_{4N} = \x_{4N+1} = 1$, with each cell corresponding to $(\x_{4i+1},\x_{4i+4})$ for $i=0,\ldots,N-1$.

{
\setlength\extrarowheight{2pt}
\begin{table}
\begin{tabular}{|c|c|c|c|}
\hline
Name of GD & Degree $k$  & Quadrature rule & $\ell$\\[.2em]
\hline
$\disc[\gen,a][3,\dg]$ & 3  & Equi6 & --\\[.2em]
\hline
$\disc[\gen,b][3,\dg]$ & 3  & Equi8 & 0\\[.2em]
\hline
$\disc[\gen,c][3,\dg]$ & 3  & Gauss--Lobatto & 2\\[.2em]
\hline
\end{tabular}
\caption{Descriptions of the mass-lumped Discontinuous Galerkin GDs in dimension $d=1$. The degree $k$ is that of the local polynomial spaces, and $\ell$ is the degree in \eqref{assum:quad.qr} for the chosen quadrature rule (`--' means that \eqref{assum:quad.qr} does not even hold for $\ell=0$). 
Here, $\gen=\unif$ or $\rand$ depending if the GD uses the uniform or random meshes.}
\label{tab:DG.1D}
\end{table}
}

The results in Table \ref{tab:casunddg} are in line with what we already observed for Finite Element methods (numerical tests conducted for $k\in\{1,2\}$, not presented here, lead to similar conclusions). The better local quadrature rules of $\disc[\gen,c][3,\dg]$ enable this variant to outperform $\disc[\gen,b][3,\dg]$ and $\disc[\gen,a][3,\dg]$ (this latter being badly hindered by its very low-order local quadrature rule), and preserve the expected order $3$ convergence for smooth data and solutions. This optimal convergence is even noticed in the fully nonlinear test case \refP{tc:casundporfz}. As before, the Stefan problem is much more challenging due to its reduced regularity, but even for this one we notice an interest in selecting a method with high enough local quadrature rules.

{
\setlength\extrarowheight{2pt}
\begin{table}
\begin{tabular}{|c||c|c||c|c||c|c||c|c||c|c|}
\hline
  &  \multicolumn{2}{c||}{Test Case \refR{tc:casundcinfty}} & \multicolumn{2}{c||}{Test Case \refP{tc:casundporfnz}}& \multicolumn{2}{c||}{Test Case \refP{tc:casundporfz}}& \multicolumn{2}{c||}{Test Case \refS{tc:casundstefz}}\\
\hline
 GD &  $C$ & $\alpha$ &  $C$ & $\alpha$&  $C$ & $\alpha$&  $C$ & $\alpha$\\
\hline 
\hline
$\disc[\unif,a][3,\dg]$ & 1.5e-01 & 1.01 & 4.2e-01 & 1.03 & 1.5e-01 & 1.01 & 9.0e-02 & 1.01     \\ 
\hline 
$\disc[\rand,a][3,\dg]$ & 1.7e-01 & 1.02 & 4.2e-01 & 1.03 & 1.5e-01 & 1.01 & 8.5e-02 & 1.00     \\ 
\hline 
$\disc[\unif,b][3,\dg]$ & 2.2e-01 & 2.00 & 2.9e+00 & 1.98 & 2.7e-01 & 2.00 & 8.2e-02 & 1.50     \\ 
\hline 
$\disc[\rand,b][3,\dg]$ & 2.2e-01 & 1.99 & 2.9e+00 & 1.97 & 2.8e-01 & 2.00 & 7.4e-02 & 1.57     \\ 
\hline 
$\disc[\unif,c][3,\dg]$ & 2.3e-02 & 3.25 & 1.4e+00 & 2.39 & 3.9e-02 & 3.42 & 5.4e-02 & 1.49     \\ 
\hline 
$\disc[\rand,c][3,\dg]$ & 1.1e-02 & 3.01 & 1.0e+00 & 2.32 & 1.9e-02 & 3.08 & 5.8e-02 & 1.58     \\ 
\hline  
\end{tabular}
\caption{Constants and rates for $E_{\zeta,\ID}^{\nabla}$ with  DG applied to some 1D test cases.}
\label{tab:casunddg}
\end{table}
}

\begin{figure}[ht!]
\resizebox{\textwidth}{!}{\includegraphics[trim={2cm 7.5cm 2cm 7cm},clip]{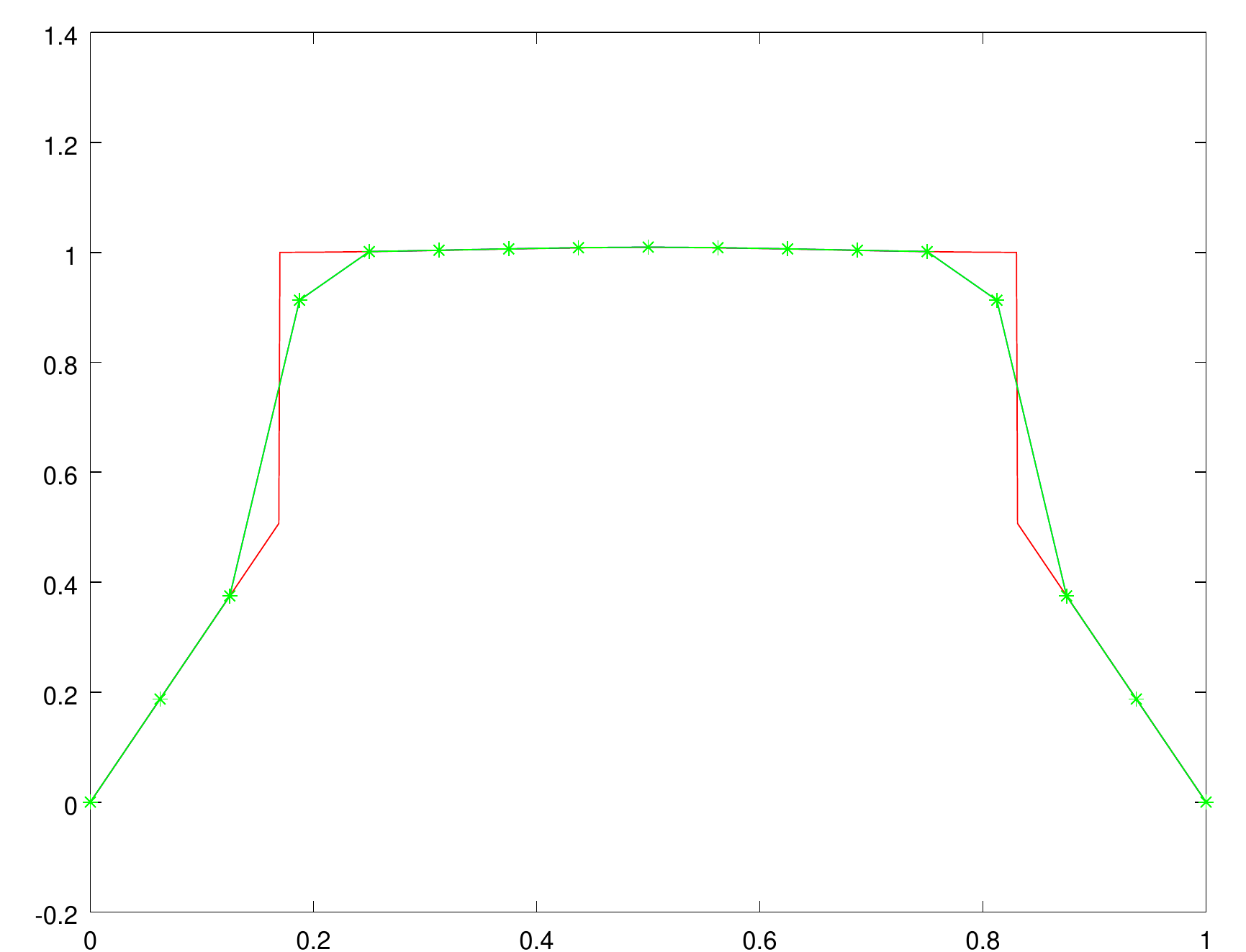}\includegraphics[trim={2cm 7.5cm 2cm 7cm},clip]{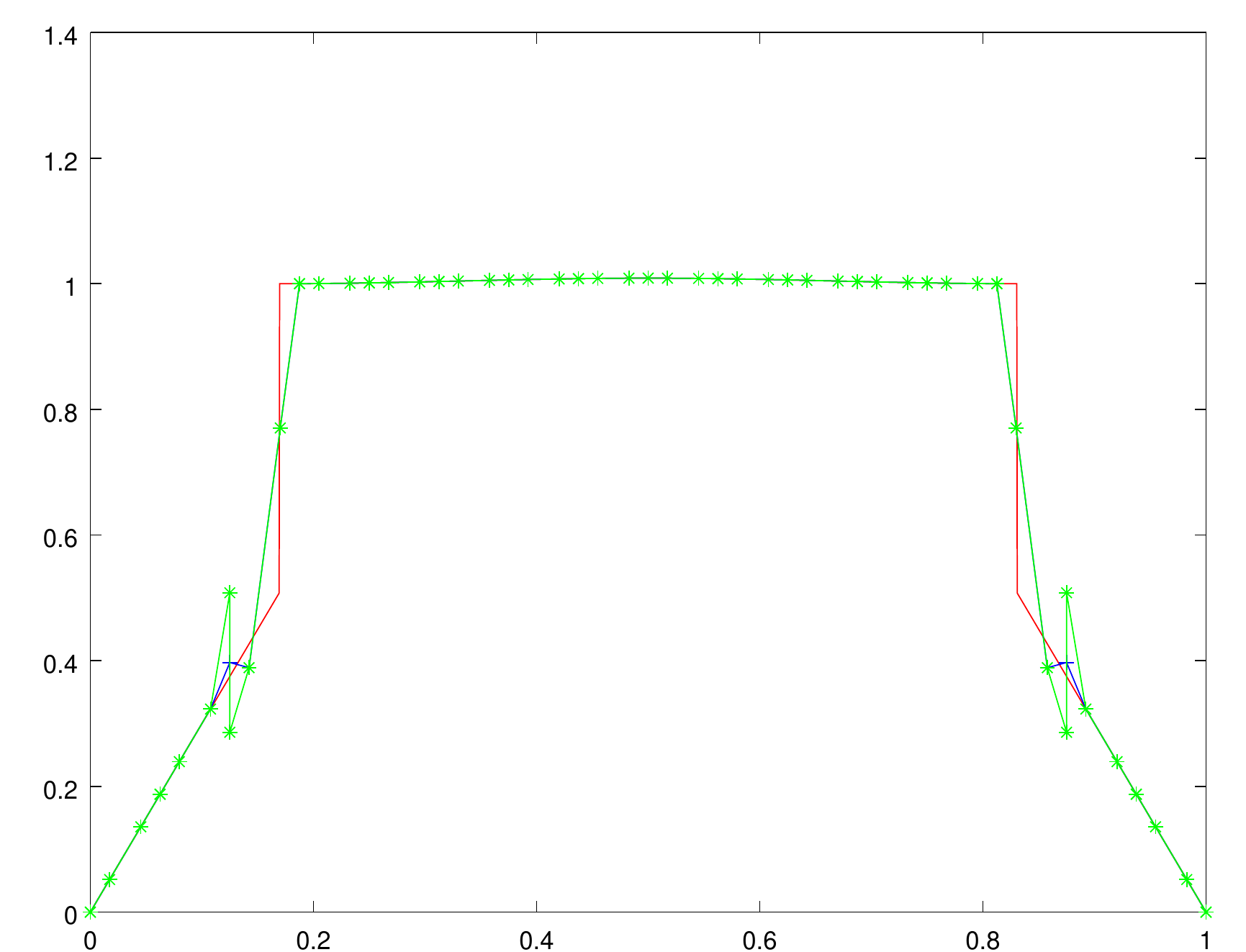}}
\caption{Comparison of approximate $u$ on Test  Case \refS{tc:casundstefnz}: $k=1$ (left), $k=3$ Gauss--Lobatto (right), Finite Element (blue,``+''), DG (green,``$*$''), exact (red), $N=16$, uniform mesh.}
\label{figcompefdg}\end{figure}

Figure \ref{figcompefdg} shows the solutions $u$ obtained with FE and DG schemes, for $k=1$ and $3$, on Test Case \refS{tc:casundstefnz} and with a relatively coarse mesh ($N=16$). As expected, the solutions obtained with $k=3$ are much more accurate. They however present oscillations (more severe for DG than for FE) in the vicinity of the discontinuity of $u$. The solutions (not shown here) for Test Case \refS{tc:casundstefz}, corresponding to $f=0$, do not display such oscillations.

\section{Conclusion}\label{sec:conclusion}

We presented a generic analysis framework, covering a range of methods, for the numerical approximation of nonlinear degenerate elliptic equations, stationary version of the Stefan or porous medium problems. We identified a particular structure of the method, the piecewise constant function reconstruction, which is sufficient and also appears to be necessary to establish the robustness of the schemes, and to obtain error estimates. We showed how to design mass-lumping versions of high-order numerical methods in order to preserve, despite the usage of piecewise constant approximations in the scheme, high-order approximations of the solution to this severely nonlinear model. Our numerical tests on mass-lumped Finite Element and Discontinuous Galerkin schemes corroborated the theoretical findings, showing that even for non-smooth solutions an elevated rate of convergence is obtained only if the mass-lumping is designed to satisfy proper local quadrature rules.

\appendix

\section{Existence and uniqueness of the weak solution}\label{appen:exist.uniq}

\begin{theorem}[Existence and uniqueness of the weak solution]\label{th:exist.uniq}
Under Assumption \eqref{eq:assum}, there is a unique solution $\bu$ to \eqref{stefan:weak}. This solution has the following regularity properties:
\begin{itemize}[\hspace*{.4em}$\bullet$]
\item $\Lambda\nabla\zeta(\bu)+F\in H_{\div}(\Omega)$;
\item if $d\le 3$ and $F\in L^p(\Omega)^d$ for some $p>d$, then $\zeta(\bu)\in C^\theta(\overline{\Omega})$ for some $\theta\in (0,1)$ depending only on $\Omega$, $\Lambda$ and $p$;
\item if $F=0$, $\Omega$ is convex and $\Lambda$ is Lipschitz-continuous, then $\zeta(\bu)\in H^2(\Omega)$.
\end{itemize}
\end{theorem}

\begin{proof}
The existence of a solution is a consequence of Theorem \ref{th:convergence.scheme}, together with Lemma \ref{lem:GD.exist} that establishes the existence of a proper sequence of Gradient Discretisations. To prove the uniqueness of this solution, consider $\bu_1$ and $\bu_2$ two solutions to \eqref{stefan:weak}, subtract their respective equations, and take $v=\zeta(\bu_1)-\zeta(\bu_2)\in H^1_0(\Omega)$ as a test function to get
\[
\int_\Omega (\beta(\bu_1)-\beta(\bu_2))(\zeta(\bu_1)-\zeta(\bu_2))+\int_\Omega \Lambda \nabla (\zeta(\bu_1)-\zeta(\bu_2))\cdot\nabla (\zeta(\bu_1)-\zeta(\bu_2))=0.
\]
The first term is non-negative since $\beta$ and $\zeta$ are non-decreasing, and thus $\nabla (\zeta(\bu_1)-\zeta(\bu_2))=0$.
This shows that $\zeta(\bu_1)=\zeta(\bu_2)$. The weak formulation \eqref{stefan:weak} also shows that
$\beta(\bu_1)-\Delta \zeta(\bu_1)=f+\div(F)=\beta(\bu_2)-\Delta \zeta(\bu_2)$ in the sense of distributions on $\Omega$; since $\zeta(\bu_1)=\zeta(\bu_2)$, this yields $\beta(\bu_1)=\beta(\bu_2)$. Hence, $\beta(\bu_1)+\zeta(\bu_1)=\beta(\bu_2)+\zeta(\bu_2)$ and Hypothesis \eqref{assum:incr} shows that $\bu_1=\bu_2$.

\medskip

We finally consider the regularity properties of $\zeta(\bu)$. This function is a weak solution of
\[
\zeta(\bu)\in H^1_0(\Omega)\mbox{ and }-\div(\Lambda\nabla\zeta(\bu)+F)=f-\beta(\bu)\in L^2(\Omega).
\]
This readily shows that $\Lambda\nabla\zeta(\bu)+F\in H_{\div}(\Omega)$. If $d\le 3$, then $L^2\subset W^{-1,q}(\Omega)$ for some $q>d$ and thus, assuming that $F\in L^p(\Omega)^d$ for $p>d$, $\zeta(\bu)$ is a solution in $H^1_0(\Omega)$ of
$-\div(\Lambda\nabla\zeta(\bu))=f+\div(F)-\beta(\bu)\in W^{-1,\min(q,p)}(\Omega)$; the results of \cite{stamp} then show that $\zeta(\bu)$ has the H\"older-regularity stated in the theorem. Finally, the $H^2$ regularity property is a straightforward consequence of the optimal elliptic regularity on convex domains for Lipschitz-continuous diffusion tensors. \end{proof}

\begin{lemma}[Existence of suitable sequences of GDs]\label{lem:GD.exist}
Under Assumption \eqref{assum:Omega}, there exists a sequence $(\disc[m])_{m\in\N}=(\XDmz,\PiDm,\nablaDm,\QD[m])_{m\in\N}$ of Gradient Discretisations, with piecewise constant reconstructions, that satisfy the coercivity, consistency, limit-conformity and compactness properties stated in Theorem \ref{th:convergence.scheme}.
\end{lemma}

\begin{proof} 
Let $(\widetilde{\mesh}_m)_{m\in\N}$ be a sequence of conformal simplicial meshes of $\R^d$ (see, e.g., \cite[Definition 7.4]{gdm}), such that $\lim_{m\to\infty}\max_{T\in\widetilde{\mesh}_m}{\rm diam}(T)\to 0$ and $(\mesh_m)_{m\in\N}$ is regular in the sense that the ratio of the diameter of $T\in\widetilde{\mesh}_m$ over the largest ball inside $T$ is bounded uniformly with respect to $T$ and $m$. We let $\mesh_m=\{T\in \widetilde{\mesh}_m\,:\,T\subset \Omega\}$ and define the polyhedral set $\Omega_m\subset \Omega$ as the interior of $\cup_{T\in\mesh_m}\overline{T}_m$.

The Gradient Discretisation $\disc[m]=(\XDmz,\PiDm,\nablaDm,\QD[m])$ is defined as the mass-lumped conforming $\Poly{1}$ Gradient Discretisation on the mesh $\mesh_m$ of $\Omega_m$ \cite[Section 8.4]{gdm}, with extensions to $\Omega$ by $0$ outside $\Omega_m$, and no quadrature rule. Letting $\vertices_m$ be the set of vertices of $\mesh_m$, we therefore set
\begin{itemize}[\hspace*{.4em}$\bullet$]
\item $\XDmz=\{v=(v_i)_{i\in \vertices_m}\,:\,v_i\in\R\,,\;v_i=0\mbox{ if $i\in\partial\Omega_m$}\}$;
\item for $v\in\XDmz$, $(\PiDm v)_{|\Omega_i}=v_i$ for all $i\in \vertices_m$, where $(\Omega_i)_{i\in\vertices_m}$ is the dual (Donald) mesh of $\mesh_m$, and $\PiDm v=0$ on $\Omega\backslash\Omega_m$;
\item for $v\in\XDmz$, $\nablaDm v$ is on $\Omega_m$ the gradient of the conforming $\Poly{1}$ reconstruction from the vertex values $(v_i)_{i\in\vertices_m}$, and $\nablaDm v=0$ on $\Omega\backslash\Omega_m$;
\item $\QD[m]={\rm Id}:L^2(\Omega)\to L^2(\Omega)$.
\end{itemize}
Since the functions and gradient reconstructions are extended by $0$ outside $\Omega_m$, $C_{\disc[m]}$ and $W_{\disc[m]}$ can be computed using norms and integrals over $\Omega_m$. The properties of mass-lumped $\Poly{1}$ GDs on $\Omega_m$ (see \cite[Theorem 8.17]{gdm}) then show that $(\disc[m])_{m\in\N}$ is coercive, limit-conforming and compact. It remains to analyse the consistency of $(\disc[m])_{m\in\N}$.

As seen in \cite[Lemma 2.16]{gdm}, the consistency follows if we prove that $S_{\disc[m]}(\varphi)\to 0$ when $\varphi\in C^\infty_c(\Omega)$. In that case, for $m$ large enough, $\varphi\in C^\infty_c(\Omega_m)$ and the norms in $S_{\disc[m]}(\varphi)$ can be restricted to $\Omega_m$. The estimate in \cite[Remark 8.18]{gdm} then shows that $S_{\disc[m]}(\varphi)\le C_\varphi \max_{T\in\mesh_m}{\rm diam}(T)$ with $C_\varphi$ not depending on $m$. This shows that $S_{\disc[m]}(\varphi)\to 0$ as $m\to\infty$, as required. \end{proof}

\section{Conforming scheme}\label{sec:conforming}

Throughout this section, we assume that $F=0$. Using Assumptions \eqref{assum:zeta}, \eqref{assum:beta} and \eqref{assum:incr}, we see that $\beta+\zeta:\R\to\R$ is bijective and we can therefore set $\mu(t)=\zeta((\beta+\zeta)^{-1}(t))$ and $\rho(t):=t-\mu(t)=\beta((\beta+\zeta)^{-1}(t))$. These functions are non-decreasing and $1$-Lipschitz continuous and, setting $\bw=(\beta+\zeta)(\bu)$, we see that \eqref{stefan:weak} is equivalent to: find $w\in L^2(\Omega)$ such that $\mu(\bw)\in H^1_0(\Omega)$ and
\begin{equation}\label{stefan:weak2}
\int_\Omega \rho(\bw) \bv +\int_\Omega \Lambda\nabla\mu(\bw) \cdot\nabla \bv = \int_\Omega f\bv\,,\qquad\forall \bv\in H^1_0(\Omega).
\end{equation}
Given a family $(V_m)_{m\in\N}$ of finite dimensional subspaces of $H^1_0(\Omega)$, conforming schemes for \eqref{stefan:weak2} are written: find $w_m\in V_m$ such that
\begin{equation}\label{stefan:cs}
\int_\Omega \rho(w_m)  v +\int_\Omega \Lambda\nabla \mu(w_m) \cdot\nabla v=\int_\Omega f v\,,\qquad\forall v\in V_m.
\end{equation}
Introducing the function $\nu:\R\to\R$ defined by $\nu(s)=\int_0^s \sqrt{\mu'(r)} dr$, we can then state the following convergence theorem.

\begin{theorem}[Convergence of the scheme]\label{th:convergence.schemecgr}
Assume that \eqref{eq:assum} holds and that, for all $\varphi\in H^1_0(\Omega)$, $\lim_{m\to\infty} \inf_{v\in V_m} \norm[H^1_0(\Omega)]{\varphi -v} = 0$.
Then, for any $m\in\N$, there exists $w_m$ a solution to \eqref{stefan:cs} and, if $\bw$ is the solution to \eqref{stefan:weak2}, as $m\to\infty$, we have $\mu(u_m)\to \mu(\bw)$ weakly in $H^1_0(\Omega)$ and strongly in $L^2(\Omega)$, $\nu(w_m)\to \nu(\bw)$  weakly in $H^1_0(\Omega)$ and strongly in $L^2(\Omega)$, and $\rho(w_m)\to \rho(\bw)$ weakly in $L^2(\Omega)$.

Moreover, if the energy equality
  \begin{equation}\label{property.bw}
  \int_\Omega \rho(\bw)\bw +\int_\Omega \Lambda\nabla \nu(\bw) \cdot\nabla \nu(\bw)= \int_\Omega f \bw
  \end{equation}
holds, then $\nabla \nu(w_m)\to \nabla \nu(\bw)$ and $w_m\to \bw$ strongly in $L^2(\Omega)$. 
\end{theorem}
\begin{remark}[On condition \eqref{property.bw}]
We observe that \eqref{property.bw} holds in the case where $\bw\in H^1_0(\Omega)$ since it can then be taken as a test function in \eqref{stefan:weak2}. But
it may also hold in some less regular situations. 
\end{remark}

\begin{proof}
We only sketch the proof.
Assuming the existence of a solution $w_m$ to the scheme, we let $v=v_m$ in \eqref{stefan:cs}, use the monotonicity of $\mu$ and $\rho$, the relation $\mu'(w_m)|\nabla w|^2 = |\nabla \nu(w_m)|^2$, the coercivity of $\Lambda$ and the Poincar\'e inequality, we write (with $a\lesssim b$ meaning $a\le Cb$ with $C$ independent of $m$):
\begin{equation}\label{conf.est.1}
\underline{\lambda}\norm[L^2]{\nabla\nu(w_m)}^2\le \norm[L^2]{f}\norm[L^2]{w_m}
\lesssim \norm[L^2]{f}(1+ \norm[L^2]{\mu(w_m)}) \lesssim \norm[L^2]{f} (1+\norm[L^2]{\nabla\mu(w_m)}).
\end{equation}
We have $|\nabla \mu(w_m)|^2 = \mu'(w_m)|\nabla \nu(w_m)|^2 \le |\nabla \nu(w_m)|^2$ and the estimate above therefore gives a bound on $\nu(w_m)$ in $H^1_0(\Omega)$, and thus also on $\mu(w_m)$. Using a coercivity property of $\mu$ similar to that of $\zeta$ we infer bounds in $L^2(\Omega)$ on $w_m$ and $\rho(w_m)$. A topological degree argument, similar to the one developed in the proof of Lemma \ref{lem:exuniq}, then ensures the existence of at least one solution $w_m$ to \eqref{stefan:cs}.

These bounds give $\bv\in H^1_0(\Omega)$ and $\bw\in L^2(\Omega)$ such that, up to a subsequence, $\mu(w_m)\to \bv$ strongly in $L^2(\Omega)$, $\nabla\mu(w_m)\to \nabla \bv$ weakly in $L^2(\Omega)^d$ and $w_m\to \bw$ weakly in $L^2(\Omega)$. By weak/strong convergence we infer that
\[
 \lim_{m\to\infty} \int_\Omega  w_m \mu(w_m) = \int_\Omega  \bw\;\bv
\]
and a Minty argument \cite[Lemma D.10]{gdm} yields $\bv = \mu(\bw)$, and thus $\rho(w_m) \to \bw - \mu(\bw) = \rho(\bw)$   weakly in $L^2(\Omega)$. We have $(\nu(a) - \nu(b))^2 \le (b-a)(\mu(b)-\mu(a))$ and the strong convergence of $\mu(w_m)$ in $L^2$ therefore shows that $\nu(w_m)\to \nu(\bw)$ in $L^2(\Omega)$. Since $(\nu(w_m))_{m\in\N}$ is bounded in $H^1_0(\Omega)$, this convergence also holds weakly in this space.

Letting $\varphi\in H^1_0(\Omega)$ and taking $v_m:= \mathrm{argmin}_{v\in V_m}  \norm[H^1_0(\Omega)]{\varphi -v}$ in \eqref{stefan:cs}, the above convergences enable us to take the limit as $m\to\infty$ to see that $\bw$ is the solution to \eqref{stefan:weak2}. The uniqueness of $\bw$ shows that the convergence property holds for the whole sequence.

Assuming that \eqref{property.bw} holds, we apply \eqref{stefan:cs} with $v=w_m$ to get
  \begin{equation}\label{cvc.1}
		\begin{aligned}
   \lim_{m\to\infty} \left(\int_\Omega \rho( w_m) w_m +\int_\Omega \Lambda\nabla \nu(w_m) \cdot\nabla \nu(w_m)\right)  ={}&\int_\Omega f \bw \\
	={}& \int_\Omega \rho(\bw)\bw +\int_\Omega \Lambda\nabla \nu(\bw) \cdot\nabla \nu(\bw).
	\end{aligned}
  \end{equation}
The weak convergence of $\nu(w_m)$ in $H^1_0(\Omega)$ ensures that
\begin{equation}\label{liminf.1}
   \liminf_{m\to\infty} \int_\Omega \Lambda\nabla \nu(w_m) \cdot\nabla \nu(w_m)\ge \int_\Omega \Lambda\nabla \nu(\bw) \cdot\nabla \nu(\bw).
\end{equation}
Developing the relation $\int_\Omega (\rho( w_m) -\rho(\bw))(w_m-\bw) \ge 0$
 and using the weak convergences $w_m\to\bw$ and $\rho(w_m)\to\rho(\bw)$ in $L^2(\Omega)$ we have 
  \begin{equation}\label{liminf.2}
   \liminf_{m\to\infty} \int_\Omega \rho( w_m) w_m \ge  \int_\Omega \rho(\bw)\bw.
  \end{equation}
  Using \eqref{liminf.1} and \eqref{liminf.2} together with \eqref{cvc.1} yields
  \[
    \int_\Omega \Lambda\nabla \nu(w_m) \cdot\nabla \nu(w_m)\to \int_\Omega \Lambda\nabla \nu(\bw) \cdot\nabla \nu(\bw)\ \mbox{ and }\ \int_\Omega \rho( w_m) w_m \to  \int_\Omega \rho(\bw)\bw.
  \]
	The first relation classically shows that $\nabla\nu(w_m)\to\nabla\nu(\bw)$ strongly in $L^2(\Omega)$. 
	Using the second relation and a weak/strong convergence argument on $\mu(w_m)w_m$ we infer that
\[
\int_\Omega w_m^2 = \int_\Omega \rho(w_m)w_m+\mu(w_m)w_m\to \int_\Omega \rho(\bw)\bw + \mu(\bw)\bw=\int_\Omega \bw^2,
\]
which gives the strong convergence in $L^2(\Omega)$ of $\bw$. \end{proof}

\begin{remark}[About the assumption $F=0$]
If $F\neq 0$, then an additional term $\int_\Omega F\cdot\nabla w_m$ appears in the sequence of inequalities \eqref{conf.est.1}, and this term cannot be estimated since no \emph{a priori} bound is expected on $w_m$ in $H^1_0(\Omega)$.
\end{remark}



\bibliographystyle{abbrv}
\bibliography{gdm_ho_ml}

\end{document}